%% file: main.tex
\documentclass{article}
     \usepackage[numbers]{natbib}
\usepackage[preprint]{neurips_2022}
\usepackage{hyperref}

\pdfoutput=1
\usepackage{etoolbox}
\usepackage{xcolor}
\definecolor{DarkGreen}{rgb}{0.1,0.5,0.1}
\definecolor{DarkRed}{rgb}{0.5,0.1,0.1}
\definecolor{DBlue}{rgb}{0.1,0.1,0.7}
\definecolor{Gray}{rgb}{0.2,0.2,0.2}
\definecolor{forestgreen}{rgb}{0.13, 0.55, 0.13}
\definecolor{mahogany}{rgb}{0.75, 0.25, 0.0}
\hypersetup{
    unicode=false,          
    pdftoolbar=true,        
    pdfmenubar=true,        
    pdffitwindow=false,      
    pdfnewwindow=true,      
    colorlinks=true,       
    linkcolor=mahogany,          
    citecolor=DBlue,        
    filecolor=mahogany,      
    urlcolor=mahogany,          
    pdftitle={Tensor network embedding},
    pdfauthor={Linjian Ma},
}

\usepackage[normalem]{ulem} 
\usepackage[english]{babel}
\usepackage[utf8]{inputenc}
\usepackage{amsmath,amsthm,amsopn,amsfonts,amssymb,mathtools,nicefrac}
\usepackage{mleftright}\mleftright
\usepackage{bm}
\usepackage{graphicx,epstopdf}
\usepackage[caption=false]{subfig}
\usepackage{graphicx}
\usepackage{mathrsfs}
\usepackage{enumitem}
\usepackage{algorithm}
\usepackage{comment}
\usepackage{algorithmic}
\usepackage[capitalize,nameinlink]{cleveref}
\usepackage{braket}
\usepackage{blkarray}
\usepackage{thm-restate}
\usepackage{thmtools}
\usepackage{tikz}
\usepackage{tikz-cd}
\usetikzlibrary{matrix}
\usepackage{filecontents}
\usepackage{adjustbox} 
\usepackage{wrapfig}
\usepackage{lipsum}
\numberwithin{equation}{section}

\input{diagrammacro}

\DeclareMathOperator*{\E}{{\mathbb{E}}}

\declaretheorem[numberwithin=section]{theorem}
\declaretheorem[sibling=theorem]{lemma}
\declaretheorem[sibling=theorem]{corollary}

\declaretheorem[style=definition]{definition}

\declaretheorem[numbered=no,style=definition,name=Fact]{fact*}

\newtheorem*{theorem*}{Theorem}
\newtheorem*{corollary*}{Corollary}
\newtheorem*{proposition*}{Proposition}
\newtheorem*{lemma*}{Lemma}
\newtheorem*{claim*}{Claim}
\newtheorem*{problem*}{Problem}

\newcommand{\BB}[1]{\mathbb{#1}}

\newcommand{\argmin}[1]{\underset{#1}{\arg\!\min}}
\newcommand{\bigO}[1]{{O}\left( #1 \right)}
\newcommand{\bigOmega}[1]{\Omega\left( #1 \right)}
\newcommand{\bigTheta}[1]{\Theta\left( #1 \right)}

\newcommand{\cost}{{\mathsf{cost}}}
\newcommand{\cut}{{\mathsf{cut}}}
\newcommand{\Ge}{G_{E}}
\newcommand{\Gd}{G_{D}}
\newcommand{\Ee}{E_{E}}
\newcommand{\Ed}{E_{D}}
\newcommand{\Ve}{V_{E}}
\newcommand{\Vd}{V_{D}}
\newcommand{\Ne}{N_{E}}
\newcommand{\Nd}{N_{D}}

\newcommand{\GL}{L}
\newcommand{\GR}{R}
\newcommand{\GA}{G_{S}}
\newcommand{\EA}{E_{S}}

\newcommand{\stD}[1]{\st{D}(e_{#1})}
\newcommand{\stX}[1]{X(e_{#1})}

\newcommand{\stt}[1]{t(e_{#1})}

\newcommand{\R}{\BB{R}}

\newcommand{\underflow}[2]{\underset{\kern-60mm \overbrace{#1} \kern-60mm}{#2}}

\newcommand{\tsr}[1]{{\mathcal{#1}}}
\newcommand{\st}[1]{\mathcal{#1}}

\newcommand{\vcr}[1]{{#1}}
\newcommand{\mat}[1]{{#1}}

\newcommand{\Comment}[1]{{\small{\textit{//   #1}}}}
\DeclareMathOperator*{\Prob}{{\text{Pr}}}

\author{
  Linjian Ma, Edgar Solomonik \\
  Department of Computer Science, University of Illinois at Urbana-Champaign\\
  \texttt{ \{lma16, solomon2\}@illinois.edu} \\
}

\title{
Cost-efficient Gaussian tensor network embeddings for tensor-structured inputs
}

\begin{document}

\maketitle

\begin{abstract}
This work discusses tensor network embeddings, which are random matrices ($S$) with tensor network structure. These embeddings have been used to perform dimensionality reduction of tensor network structured inputs $x$ and accelerate applications such as tensor decomposition and kernel regression. Existing works have designed embeddings for inputs $x$ with \textit{specific} structures, such as the Kronecker product or Khatri-Rao product, such that the computational cost for calculating $Sx$ is efficient. We provide a systematic way to design tensor network embeddings consisting of Gaussian random tensors, such that for inputs with \textit{more general} tensor network structures, both the sketch size (row size of $S$) and the sketching computational cost are low.

We analyze general tensor network embeddings that can be reduced to a sequence of sketching matrices. We provide a sufficient condition to  quantify the accuracy of such embeddings and derive sketching asymptotic cost lower bounds using embeddings that satisfy this condition and have a sketch size lower than any input dimension. We then provide an algorithm to efficiently sketch input data using such embeddings. The sketch size of the embedding used in the algorithm has a linear dependence on the number of sketching dimensions of the input. Assuming tensor contractions are performed with classical dense matrix multiplication algorithms, this algorithm achieves asymptotic cost within a factor of $O(\sqrt{m})$ of our cost lower bound, where $m$ is the sketch size. Further, when each tensor in the input has a dimension that needs to be sketched, this algorithm yields the optimal sketching asymptotic cost. We apply our sketching analysis to inexact tensor decomposition optimization algorithms. We provide a sketching algorithm for CP decomposition that is asymptotically faster than existing work in multiple regimes, and show optimality of an existing algorithm for tensor train rounding.

\end{abstract}

\section{Introduction}\label{sec:intro}
\input{intro}

\input{contribution}

\section{Experiments}\label{sec:exp}

\input{exp}

\section{Conclusions}\label{sec:conclu}
\input{conclu}

\section*{Acknowledgments}
This work is supported by the US NSF OAC via award No. 1942995.

\bibliographystyle{abbrv}
\bibliography{main}

\clearpage

\normalsize
\appendix

\section{Background}\label{sec:background}
\input{background}

\section{Definitions and basic properties of tensor network embedding}
\label{sec:linearization}
\input{linearization}

\input{algorithm}

\input{lower_bound}

\input{tree_embedding}

\input{application}

\end{document}

%% file: diagrammacro.tex
\definecolor{tensorblue}{rgb}{0.8,0.8,1}
\definecolor{tensorred}{rgb}{1,0.5,0.5}
\definecolor{tensorpurp}{rgb}{1,0.5,1}

\tikzset{ten/.style={fill=tensorblue}}
\tikzset{tenred/.style={fill=tensorred}}
\tikzset{tengreen/.style={fill=green!50!black!50}}
\tikzset{tenpurp/.style={fill=tensorpurp}}
\tikzset{tengrey/.style={fill=black!20}}
\tikzset{tenorange/.style={fill=orange!30}}
\tikzset{u/.style={fill=blue!20,draw=black}}
\tikzset{w/.style={fill=green!50!black!80,draw=black}}

\newcommand{\diagram}[1]{ \begin{array}{cc}{\begin{tikzpicture}[scale=.5,every node/.style={sloped,allow upside down},baseline={([yshift=+0ex]current bounding box.center)}] #1 \end{tikzpicture}} \end{array} }

\newcommand{\diagramsized}[2]{ \begin{array}{cc} {\begin{tikzpicture}[scale=#1,every node/.style={sloped,allow upside down},baseline={([yshift=+0ex]current bounding box.center)}] #2 \end{tikzpicture}} \end{array} }
\definecolor{awesome}{rgb}{0.33, 0.42, 0.18}

\usetikzlibrary{arrows,decorations.pathmorphing,backgrounds,positioning,fit,calc,patterns,shapes,external}

%% file: intro.tex
Sketching techniques, which randomly project high-dimensional data onto lower dimensional spaces while
still preserving relevant information in the data~\cite{sarlos2006improved}, have been widely used in numerical linear algebra, including for regression, low-rank approximation, and matrix multiplication~\cite{woodruff2014sketching}. One key step of sketching algorithms is to design an embedding matrix $\mat{S}\in\R^{m\times n}$ with $m\ll n$, such that  for any input (also called data throughout the paper) $\vcr{x}\in\R^{n}$, the projected vector norm is $(1\pm \epsilon)$-close to the input vector, $\|\mat{S}\vcr{x}\|_2 = (1\pm \epsilon)\|\vcr{x}\|_2$, with probability at least $1-\delta$ (defined as $(\epsilon,\delta)$-accurate embedding throughout the paper), and the multiplication $\mat{S}\vcr{x}$  can be computationally efficient. $\mat{S}$ is commonly chosen as a random matrix with each element being an i.i.d. Gaussian variable when $\vcr{x}$ is dense, or a random sparse matrix when $\vcr{x}$ is sparse, etc.

In this work, we focus on the case where $\vcr{x}$ has a tensor network structure. 
A tensor network~\cite{orus2014practical} uses a set of (small) tensors, where some or all of their dimensions are contracted according to some pattern, to implicitly represent a tensor.
Tensor network structured data is commonly seen in multiple applications, including kernel based statistical learning~\cite{pham2013fast,ahle2020oblivious,woodruff2022leverage,meister2019tight}, machine learning and data mining via tensor decomposition
methods~\cite{anandkumar2014tensor,sidiropoulos2017tensor,kolda2009tensor,sidiropoulos2017tensor}, and simulation of quantum systems~\cite{vidal2003efficient,ma2022low,schollwock2005density,hohenstein2012tensor,hummel2017low}. 
Commonly used embedding matrices are sub-optimal for sketching many such data.
For example, consider the case where $\vcr{x}\in\R^{s^N}$ is a chain of Kronecker products,
$   \vcr{x} = \vcr{x}_1 \otimes \cdots \otimes \vcr{x}_N,$
where $\vcr{x}_i\in\R^{s}$ for $i\in\{1,\ldots,N\}$.
 If $\mat{S}\in\R^{m \times s^N}$ is a Gaussian matrix, the multiplication $\mat{S}\vcr{x}$ has a computational cost of $\Omega\left(ms^N\right)$, and the exponential dependence on the tensor order $N$ makes the calculation impractical when $N$ or $s$ is large. 

The computational cost of the multiplication $\mat{S}\vcr{x}$ can be reduced when $\mat{S}$ has a  structure that can be easily multiplied with the target data.
One example is when $\mat{S}$ has a Kronecker product structure, $\mat{S} = \mat{S}_1\otimes \cdots \otimes \mat{S}_N$ and each $\mat{S}_i\in \R^{m^{1/N}\times s}$. When $\vcr{x} = \vcr{x}_1 \otimes \cdots \otimes \vcr{x}_N$, $\mat{S}\vcr{x}$ can then be calculated efficiently via $(\mat{S}_1\vcr{x}_1) \otimes \cdots \otimes (\mat{S}_N\vcr{x}_N)$, reducing the cost to $\bigO{Nm^{1/N}s + m}$.
Another example is when $\mat{S}$ has a Khatri-Rao product structure, $\mat{S} = (\mat{S}_1\odot \cdots \odot \mat{S}_N)^T$ and each $\mat{S}_i\in \R^{s\times m}$. $\mat{S}\vcr{x}$ can then be calculated efficiently via $(\mat{S}_1^T\vcr{x}_1) * \cdots * (\mat{S}_N^T\vcr{x}_N)$, where $*$ denotes the Hadamard product, which reduces the cost to $\bigO{Nms}$.
However, tensor-network-structured embedding matrices that can be easily multiplied with data may not necessarily minimize computational cost,
since the sketch size sufficient for accurate embedding can also increase. 
For example, 
the sketch size necessary for both Kronecker product and Khatri-Rao product embeddings  to be $(\epsilon,\delta)$-accurate is at least
exponential in $N$, which is inefficient for large tensor order $N$~\cite{ahle2020oblivious}. 
To find embeddings that are both accurate and computationally efficient, 
it is therefore of interest to investigate tensor network structures that can both yield small sketch size and be multiplied with data efficiently.

Existing works discuss tensor network embeddings with more efficient sketch size than Kronecker and Khatri-Rao product structure, such as tensor train~\cite{rakhshan2020tensorized} and balanced binary tree~\cite{ahle2020oblivious}.
In particular, Ahle et al.~\cite{ahle2020oblivious} designed a balanced binary tree structured embedding and showed that the sketch size 
sufficient for $(\epsilon,\delta)$-accurate embedding can have only linear dependence on $N$.
Using this embedding to sketch Kronecker product structured data yields a sketching cost that only has a  polynomial dependence on both $N$ and $s$. 
However, for data with other tensor network structures, these embeddings may not be the most computationally efficient.

\paragraph{Our contributions}
  \begin{wrapfigure}{r}{0.22\textwidth}
\vspace{-2mm}
\centering
\includegraphics[width=.22\textwidth, keepaspectratio]{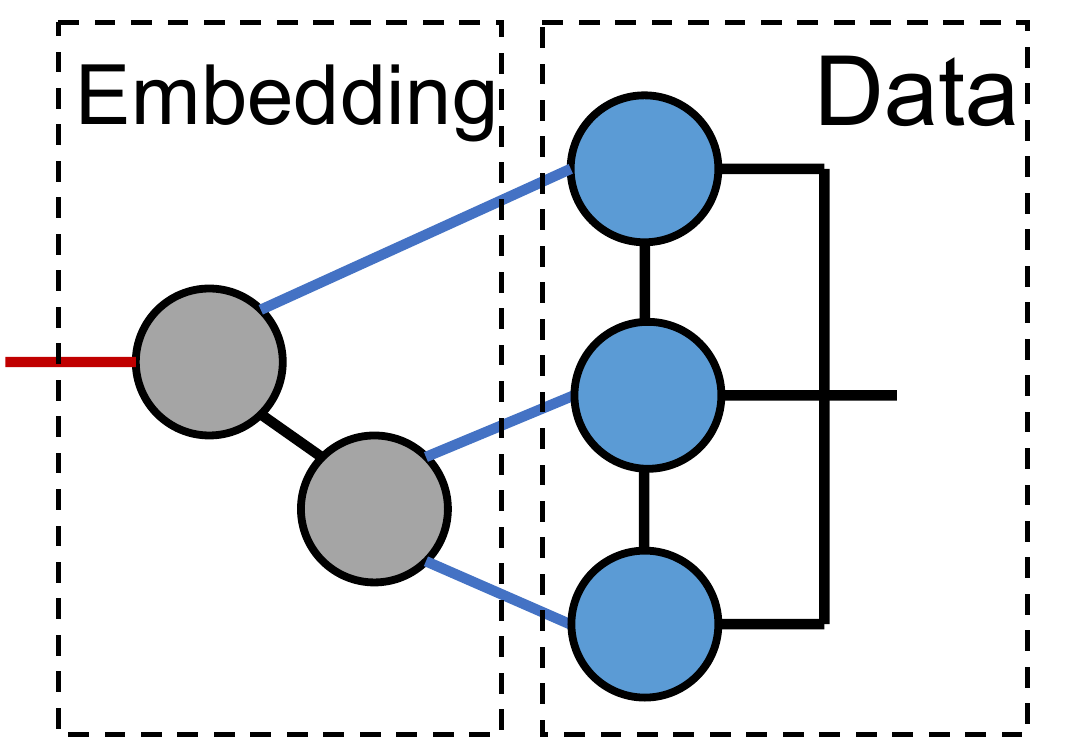}
	\vspace{-3mm}
  \caption{Illustration of target data and embedding. Blue edges have larger weights than the red edge.
  }
  \label{fig:example_intro}
	\vspace{-5mm}
  \end{wrapfigure}
In this work, we
design algorithms to efficiently sketch more general data tensor networks such that each dimension to be sketched has size lower bounded by the sketch size and is a dimension of only one tensor. One of such data tensor networks is shown in \cref{fig:example_intro}.
In particular, we look at the following question.

    \textit{For arbitrary data with a tensor network structure of interest, can we automatically sketch the data into one tensor with
    Gaussian tensor network embeddings
    that are accurate, have low sketch size, and also minimize the sketching asymptotic computational cost?}

Different from existing works~\cite{ahle2020oblivious,jin2019faster,ma2021fast} that construct the embedding based on fast sketching techniques, including Countsketch~\cite{charikar2002finding}, Tensorsketch~\cite{pagh2013compressed}, and fast Johnson-Lindenstraus (JL) transform using fast Fourier transform~\cite{ailon2006approximate}, we discuss the case where each tensor in the embedding contains i.i.d. Gaussian random elements. Gaussian-based embeddings yield larger computational cost, but have the most efficient sketch size for both unconstrained and constrained optimization problems~\cite{chen2020nearly,pilanci2015randomized}.
This choice also enables us use a simple computational model to analyze the sketching cost, where tensor contractions are performed with classical dense matrix multiplication algorithms.

While we allow for the data tensor network to be a hypergraph, we consider only graph embeddings,
(detailed definition in \cref{sec:def}), which include tree embeddings that have been previously studied~\cite{ahle2020oblivious,daas2021randomized,batselier2018computing}.
Each one of these embeddings consisting of $\Ne$ tensors can be reduced to a sequence of $\Ne$ sketches (random sketching matrices).
In \cref{subsec:summery_sketchsize}, we show that if each of these sketches is
$(\epsilon/\sqrt{\Ne},\delta)$-accurate, then the embedding is at least ($\epsilon, \delta$)-accurate.

In \cref{subsec:summary_efficient_ebmedding},
we provide an algorithm to sketch input data with an embedding that not only satisfies the $(\epsilon,\delta)$-accurate sufficient condition, but is computationally efficient and has low sketch size. 
Given a data tensor network and one data contraction tree $T_0$, this algorithm outputs a sketching contraction tree that is constrained on
$T_0$.
This setting is useful for application of sketching to alternating optimization in tensor-related problems, such as tensor decompositions. In alternating optimization, multiple contraction trees of the data $x$ are chosen in an alternating order to form multiple optimization subproblems, each updating part of the variables~\cite{phan2013fast,ma2021efficient,ma2020autohoot}.
Designing embeddings under the constraint can help reuse
contracted intermediates across subproblems.

The sketch size of the embedding used in the algorithm has a linear dependence on
the number of sketching dimensions of the input. As to the sketching asymptotic computational cost,
within all constrained sketching contraction trees with embeddings satisfying the $(\epsilon,\delta)$-accurate sufficient condition and only have one output sketch dimension, this algorithm
achieves asymptotic cost within a factor of $O(\sqrt{m})$ of the lower bound, where $m$ is the sketch size.
When the input data tensor network structure is a graph, the factor improves to $O(m^{0.375})$.
In addition, when
 each tensor in the input data has a dimension to be sketched, such as Kronecker product input and tensor train input,
this algorithm yields the optimal sketching asymptotic cost.

At the end of \cref{subsec:summary_efficient_ebmedding}, we look at cases where the widely discussed tree tensor network embeddings are efficient in terms of the sketching computational cost.
We show for input data graphs such that each data tensor has a dimension to be sketched and each contraction in the given data contraction tree $T_0$ contracts dimensions with size being at least the sketch size, sketching with tree embeddings can achieve the optimal asymptotic cost.

In \cref{sec:app}, we apply our sketching algorithm to two applications, CANDECOMP/PARAFAC (CP) tensor decomposition~\cite{hitchcock1927expression,harshman1970foundations} and tensor train rounding~\cite{oseledets2011tensor}. 
We present a new sketching-based alternating least squares (ALS) algorithm for CP decomposition.
Compared to existing sketching-based ALS algorithm, this algorithm yields better asymptotic computational cost under several regimes, such as when the CP rank is much lower than each dimension size of the input tensor.
We also provide analysis on the recently introduced randomized tensor train rounding algorithm~\cite{daas2021randomized}. We show that the tensor train embedding used in that algorithm satisfies the accuracy sufficient condition in \cref{subsec:summery_sketchsize} and yields the optimal sketching asymptotic cost, implying that this is an efficient algorithm, and embeddings with other structures cannot achieve lower asymptotic cost.

%% file: contribution.tex
\section{Definitions}\label{sec:def}
  \begin{wrapfigure}{r}{0.35\textwidth}
\vspace{-2mm}
\centering
\includegraphics[width=.35\textwidth, keepaspectratio]{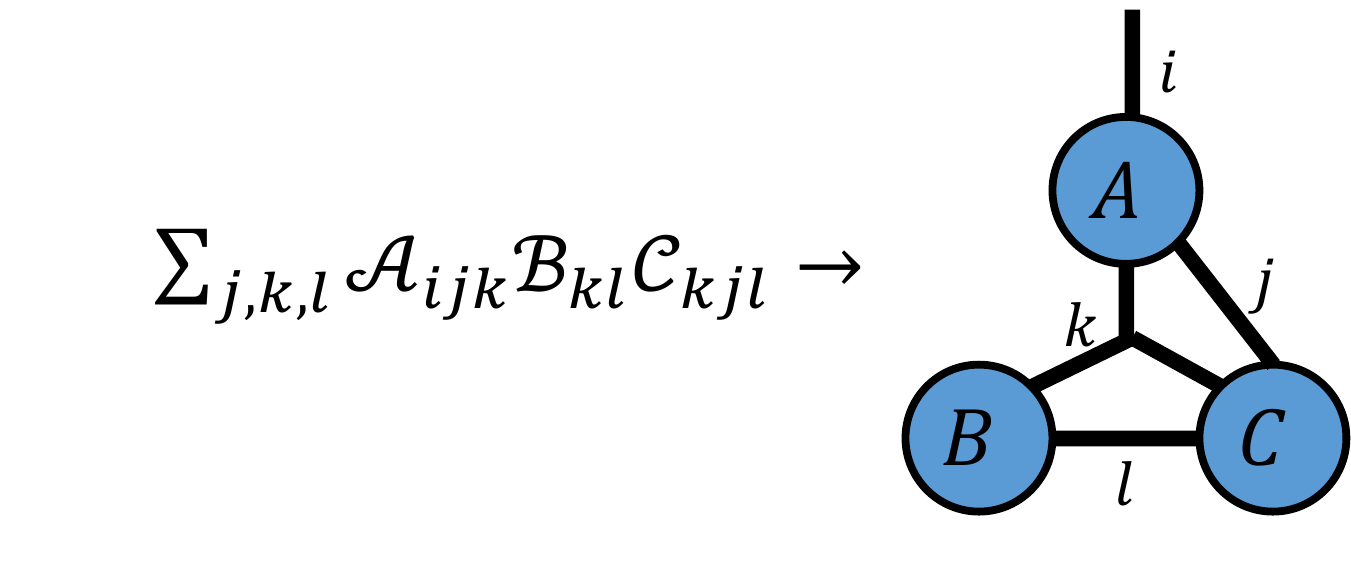}
	\vspace{-6mm}
  \caption{An example of tensor diagram notation.
  }
  \label{fig:tensordiagram}
	\vspace{-3mm}
  \end{wrapfigure}

We introduce some tensor network notation here, and provide additional definitions and background in \cref{sec:background}.
The structure of a tensor network can be described by an undirected hypergraph $G=(V,E,w)$, also called tensor diagram.
Each hyperedge $e\in E$ may be adjacent to either one or at least two vertices,
and we refer to hyperedges with a dangling end (one end not adjacent to any vertex) as uncontracted hyperedges, and those without dangling end as contracted hyperedges. We refer to the cardinality of a hyperedge as its number of ends.
An example is shown in \cref{fig:tensordiagram}.
We use $w$ to denote a function such that for each $e\in E$, $w(e) = \log(s)$ is the natural logarithm of the dimension sizes represented by hyperedge $e$. For a hyperedge set $E$, we use $w(E)=\sum_{e\in E}w(e)$ to denote the weighted sum of the hyperedge set.

A tensor network embedding is the matricization of a tensor described by a  tensor network, and each embedding can be described by $S=(\Ge,\bar{E})$, where $\Ge =(\Ve,\Ee,w)$ shows the embedding graph structure and 
$\bar{E}\subseteq \Ee$ is the edge set connecting data and the embedding.
In this work we only discuss the case where $\Ge$ is a graph, such that each uncontracted edge in $\Ee$ is adjacent to one vertex and contracted edge in $\Ee$ is adjacent to two vertices.
Let $E_1\subset \Ee$ be the subset of uncontracted edges, $S$ is a matricization such that uncontracted dimensions in $\bar{E}$ are grouped into the column of the matrix, and dimensions in $E_1$ are grouped into the row.
We use $N=|\bar{E}|$ to denote the order of the embedding, and $m=\exp(w(E_1))$ denotes the output sketch size. 
We use $\Gd= (\Vd,\Ed,w)$ to represent the data tensor network structure, and use $G=(\Gd,\Ge)$ to denote the overall tensor network structure.

Within the tensor network $G=(V,E,w)$, the contraction between two tensors represented by $v_i,v_j\in V$ is denoted by $(v_i,v_j)$. The contraction between two tensors that are the contraction outputs of $W_i\subset V$, $W_j\subset V$, respectively, is denoted by $(W_i,W_j)$.
A contraction tree on the tensor network $G=(V,E,w)$ is a rooted binary tree $T_B =(V_B, E_B)$ showing how the tensor network is fully contracted. Each vertex in $V_B$ can be represented by a subset of the vertices, $W\subseteq V$, and denotes
the contraction output of $W$. The two children of $W$, denoted as $W_1$ and $W_2$, must satisfy $W_1\cup W_2 = W$. Each leaf vertex must have $|W|=1$, and the root vertex is represented by $V$. Any topological sort of the contraction tree represents a contraction path (order) of the tensor network.

\section{Sufficient condition for accurate embedding}\label{subsec:summery_sketchsize}

  \begin{wrapfigure}{r}{0.6\textwidth}
\vspace{-2mm}
\centering
\includegraphics[width=.6\textwidth, keepaspectratio]{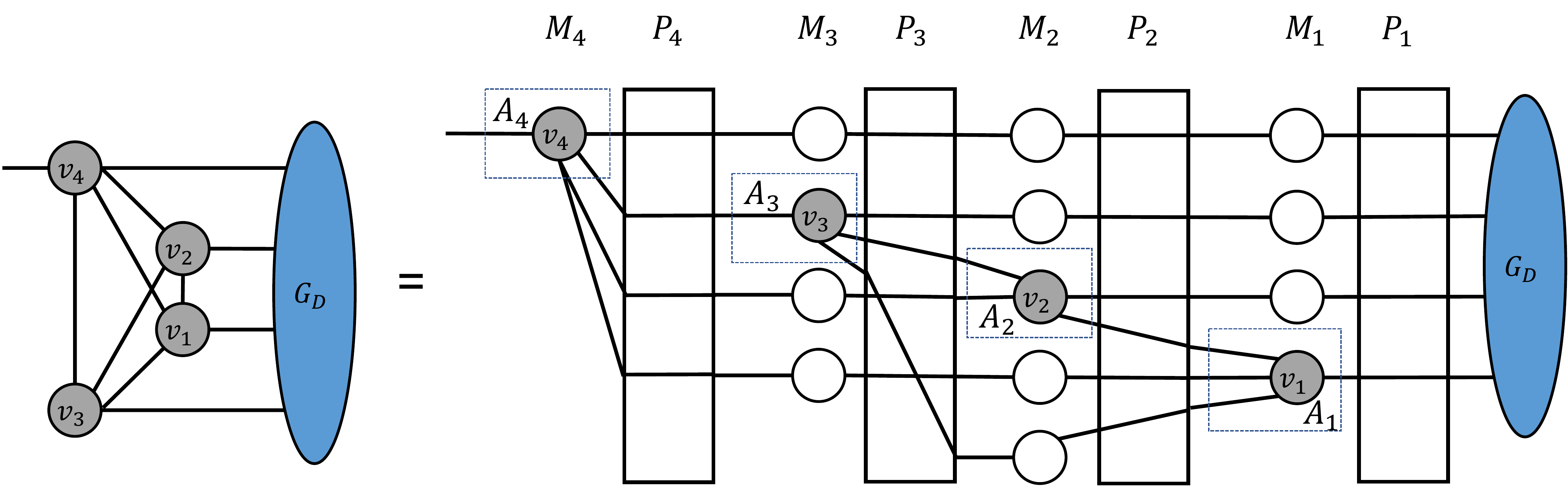}
	\vspace{-3mm}
  \caption{Illustration of embedding linearization. 
  Each gray vertex denotes a tensor of the embedding, each white vertex denotes an identity matrix, and each white box denotes a permutation matrix. 
  }
  \label{fig:embedding_examples}
	\vspace{-3mm}
  \end{wrapfigure}

We consider tensor network embeddings that have a graph structure, thus some embeddings, such as those with a Khatri-Rao product structure~\cite{rakhshan2020tensorized,chen2020nearly}, are not considered in this work. Such embeddings can be linearized to a sequence of sketches. 
Let $\Ne=|\Ve|$ denote the number of vertices in the embedding,
in each linearization, each vertex is given an unique index $i\in[\Ne]$\footnote{Throughout the paper we use $[N]$ to denote $\{1,\ldots, N\}$.} and denoted $v_i$. 
The $i$th tensor is denoted by $\tsr{A}_i$, and $\mat{A}_{i}$ denotes its matricization where we combine all uncontracted dimensions and contracted dimensions connected to $\tsr{A}_j$ with $j > i$ into the row, and other dimensions into the column. 
The embedding can then be represented as a chain of multiplications,
$
    \mat{S} = \mat{M}_{\Ne}\mat{P}_{\Ne} \cdots \mat{M}_1\mat{P}_1 ,
$
where  $\mat{M}_{i}$ is the Kronecker product of identity matrices with $\mat{A}_{i}$ for $i\in[\Ne]$, and $\mat{P}_i$ is a permutation matrix. 
We illustrate the linearization 
in \cref{fig:embedding_examples} using  a fully connected tensor network embedding.
We show in~\cref{thm:subgaussian_embedding} a sufficient condition for embeddings to be $(\epsilon,\delta)$-accurate.

\begin{theorem}[($\epsilon, \delta$)-accurate sufficient condition] \label{thm:subgaussian_embedding}
Consider a Gaussian tensor network embedding where
there exists a linearization such that each $\mat{A}_{i}$ for $i\in [\Ne]$ has row size $\Omega(\Ne\log(1/\delta)/\epsilon^2)$. Then the tensor network embedding is ($\epsilon, \delta$)-accurate.
\end{theorem}
\vspace{-3mm}
\begin{proof}
Based on the composition rules of JL moment~\cite{kane2011almost,kane2014sparser} in \cref{lem:kronecker_jl} and \cref{lem:composition_strongjl} in the appendix, in the linearization
all $\mat{M}_i\mat{P}_i$ satisfy the strong $\left(\frac{\epsilon}{L\sqrt{2N}}, \delta\right)$-JL moment property so  $\mat{S}$ satisfies the strong $\left({\epsilon}, \delta\right)$-JL moment property. This implies the embedding is $(\epsilon, \delta$)-accurate.
\end{proof}

\cref{thm:subgaussian_embedding} is a sufficient (but not necessary) condition for constructing ($\epsilon, \delta$)-accurate embedding. 
It also implies that specific tree embeddings are $(\epsilon,\delta)$-accurate, as we show below.

\begin{corollary}\label{cor:tree}
Consider a Gaussian embedding containing a tree tensor network structure, 
where
there is only one output sketch dimension with size $m=\bigTheta{\Ne\log(1/\delta)/\epsilon^2}$, and each dimension within the embedding has size $m$. Then the embedding is $(\epsilon,\delta)$-accurate.
\end{corollary}
\vspace{-3mm}
\begin{proof}
Consider the linearization such that vertices are labelled based on the reversed ordering of a breath-first search from the vertex adjacent to the edge associated with the output sketch dimension. Each $\mat{A}_i$ has row size $m = \bigTheta{\Ne\log(1/\delta)/\epsilon^2}$ thus the embedding satisfies \cref{thm:subgaussian_embedding}.
\end{proof}

One special case of \cref{cor:tree} is the 
tensor train~\cite{oseledets2011tensor} (also called
matrix product states (MPS)~\cite{schollwock2011density}) embedding, where the embedding tensor network has a 1D structure along with an output dimension adjacent to one of the endpoint tensors. Tensor train is widely used to efficiently represent high dimensional tensors in multiple applications, including numerical PDEs~\cite{dolgov2012fast,richter2021solving}, quantum physics~\cite{schollwock2005density}, high-dimensional data analysis~\cite{klus2018tensor,klus2015numerical} and machine learning~\cite{beylkin2009multivariate,stoudenmire2016supervised,obukhov2021spectral}.
Since the tensor train embedding contains $N$ vertices,
\cref{cor:tree} directly implies that a sketch size of $m=\bigTheta{N\log(1/\delta)/\epsilon^2}$ is sufficient for the MPS embedding to be  $(\epsilon,\delta)$-accurate. This embedding has already been used in applications including tensor train rounding~\cite{daas2021randomized} and low rank approximation of matrix product operators~\cite{batselier2018computing}.

Note that the tensor train embedding introduced in this work and \cite{daas2021randomized} adds an output sketch dimension to the standard tensor train, and restricts the tensor train rank to be the sketch size $m$.
This is different from the recent work by Rakhshan and Rabusseau~\cite{rakhshan2020tensorized}, where they construct an embedding consisting of $m$ independent tensor trains, each one with a tensor train rank of $R$. A sketch size upper bound of $m=\bigTheta{1/\epsilon^2\cdot(1+2/R)^N\log^{2N}(1/\delta)}$ is derived for that embedding to be $(\epsilon,\delta)$-accurate. However, this bound has an exponential dependence on $N$.

\section{A sketching algorithm with efficient computational cost and sketch size}
\label{subsec:summary_efficient_ebmedding}

We find Gaussian tensor network embeddings $\Ge$ that both have efficient sketch size and yield efficient computational cost.
We are given a specific data tensor network $\Gd$ that implicitly represents a matrix $\mat{M}\in\R^{s_1s_2\ldots s_N \times t}$, and want to sketch the row dimension of the matrix. We assume that size of each dimension to be sketched, $s_i$ for $i\in[N]$, is greater than the sketch size $m$, and each one of these dimensions is adjacent to only one tensor.
The goal is to find a Gaussian embedding $\Ge$ satisfying the following properties.

  \begin{wrapfigure}{r}{0.3\textwidth}
\vspace{-2mm}
\centering
\includegraphics[width=.3\textwidth, keepaspectratio]{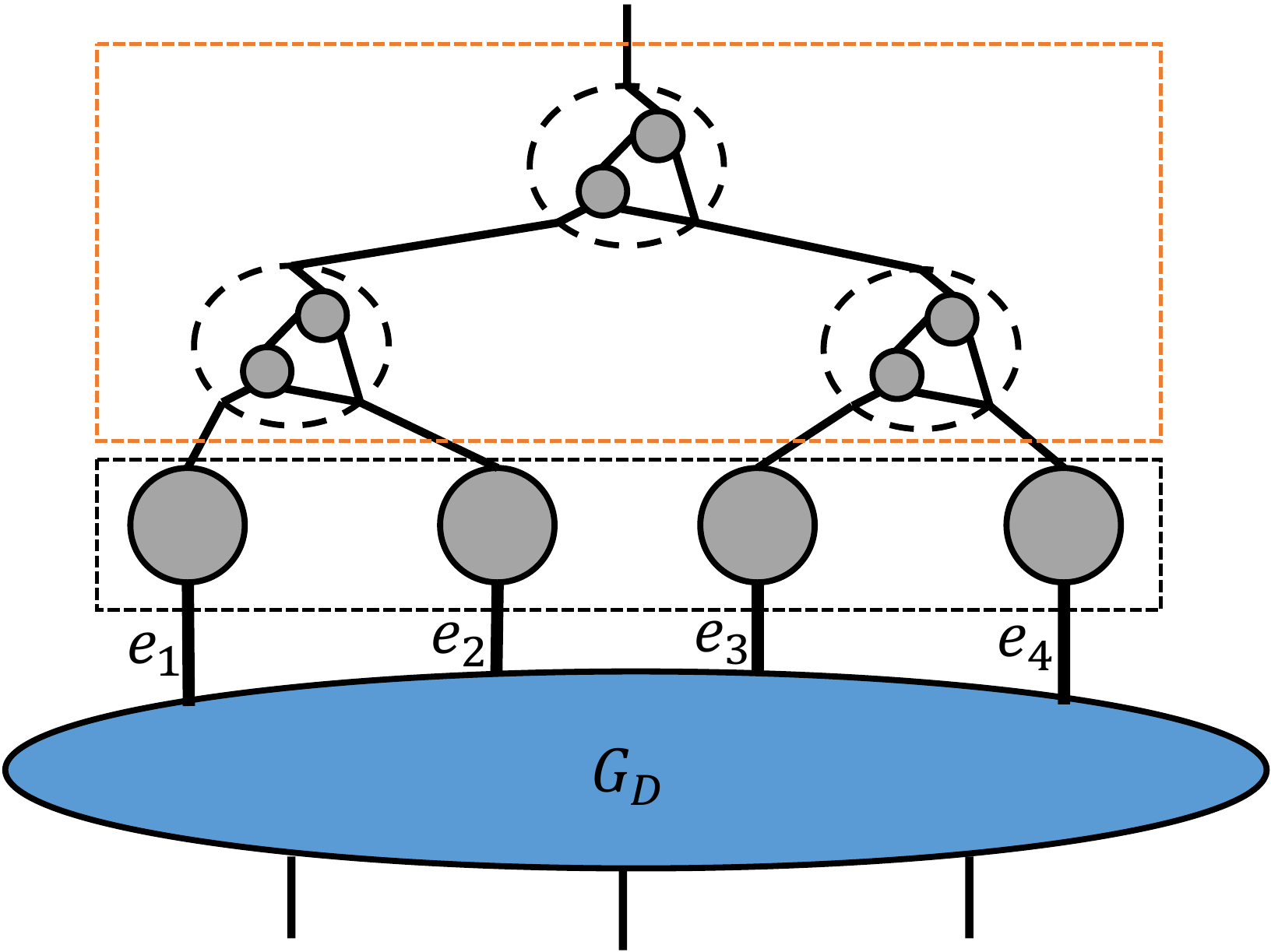}
\vspace{-4mm}
  \caption{Illustration of the embedding. The black box includes the Kronecker product embedding and the orange box includes the embedding containing a binary tree of small tensor networks.
  }
  \label{fig:ovw_embedding}
	\vspace{-5mm}
  \end{wrapfigure}

\begin{itemize}[leftmargin=*, itemsep=0pt]
    \item $\Ge\in \mathcal{G}^{(\epsilon,\delta)}$, where $\mathcal{G}^{(\epsilon,\delta)}$ contains all embeddings not only satisfying the ($\epsilon, \delta$)-accurate sufficient condition in \cref{thm:subgaussian_embedding}, but also only have one output sketch dimension ($|E_1|=1$) with size $m=\bigTheta{\Ne\log(1/\delta)/\epsilon^2}$.
    This guarantees that the embedding is accurate and the output sketch size is linear w.r.t. the number of vertices in $\Ge$. 
    Note that although the data can be a hypergraph, the embeddings considered in $\mathcal{G}^{(\epsilon,\delta)}$ are defined on graphs. 
    \item To fully contract the tensor network system $({\Gd, \Ge})$, this embedding yields a contraction tree with the \textit{optimal asymptotic contraction cost} under a fixed data contraction tree. 
    The data contraction tree constraint is useful for application of sketching to alternating optimization algorithms, as we will discuss in \cref{sec:app}.
    This can be written as an optimization problem below,
    \begin{align}\label{eq:cost_opt_unconstrained}
            \min_{\Ge} \min_{T_B} C_a(T_B({\Gd, \Ge})), 
           \quad \text{s.t. } 
        \Ge\in \mathcal{G}^{(\epsilon,\delta)}, 
            \quad
            T_0({\Gd})\subset T_B({\Gd, \Ge}),
    \end{align}
    where $T_B(\Gd,\Ge)$ denotes a contraction tree of the tensor network $(\Gd,\Ge)$,  $C_a$ denotes the asymptotic computational cost, and $T_0({\Gd})\subset T_B({\Gd, \Ge})$ means the contraction tree $T_B$ is constrained on $T_0$ (detailed definition in \cref{def:contrained_contraction_tree}). 
\end{itemize}

\begin{definition}[Constrained contraction tree]\label{def:contrained_contraction_tree}
Given $G=(\Ge,\Gd)$ and a contraction tree $T_0$ of $\Gd$, the contraction tree $T_B$ for $G$ is constrained on $T_0$ if for each contraction $(A,B)\in T_0$, there must exist one contraction $(\hat{A},\hat{B})\in T_B$, such that $\hat{A}\cap \Vd = A$ and $\hat{B}\cap \Vd = B$. 
\end{definition}

\paragraph{Algorithm}

We propose an algorithm to sketch tensor network data with an embedding containing two parts, a Kronecker product embedding and an embedding containing a binary tree of small tensor networks. The embedding is illustrated in \cref{fig:ovw_embedding}. 
The Kronecker product embedding consists of $N$ Gaussian random matrices and 
is used to reduce the weight of each edge in $\bar{E}$, the set of edges to be sketched.
The binary tree structured embedding consists of $N-1$ \textit{small tensor networks}, each represented by one binary tree vertex. 
Each small tensor network is used to effectively sketch the contraction of pairs of tensors adjacent to edges in $\bar{E}$.
The embedding with a binary tree structure may not be a binary tree tensor network, since each tree vertex is not restricted to represent one tensor.
The binary tree is chosen to be consistent with the dimension tree of the data contraction tree $T_0$, which is a directed binary tree showing the way edges
 in $\bar{E}$ are merged onto the same tensor in $T_0$. The detailed definition of dimension tree is in \cref{subsec:graph_notation}.

We first introduce some notation before presenting the algorithm. Consider a given input data tensor network $\Gd = (\Vd, \Ed, w)$ and its given data contraction tree, $T_0$. 
Below we let $\bar{E} = \{e_1,e_2,\ldots, e_N\}$ to denote the edges to be sketched. Let $\Nd = |\Vd|$.
Based on the definition we have $N \leq \Nd$ and $T_0$ contains $\Nd - 1$ contractions. Let one contraction path of $T_0$, which is a topological sort of the contractions in $T_0$, be expressed as
\begin{equation}\label{eq:exp_T0}
    \left\{(U_1,V_1), \ldots, (U_{\Nd-1}, V_{\Nd-1})\right\}.
\end{equation}
The $\Nd-1$ contractions can be categorized into $N+2$ sets, $\st{D}(e_1),\ldots, \st{D}(e_N), \st{S}, \st{I}$, as follows,
\begin{itemize}[leftmargin=*, itemsep=0pt]
\item Consider contractions $(U_i,V_i)$ such that both $U_i$ and $V_i$ are adjacent to edges in $\bar{E}$. $\st{S}$ contains all contractions with this property.
\item Consider contractions $(U_i, V_i)$ such that the only edge in $\bar{E}$ that is adjacent to the contraction output is $e_j$, $\bar{E}(U_i\cup V_i) = \{e_j\}$. 
We let  $\st{D}(e_j)$ contains contractions with this property.
When $\stD{j}$ is not empty,
we let $X(e_j)\subset V$ represent the sub network contracted by $\stD{j}$. When $\stD{j}$ is empty, we let $X(e_j)=v_j$, where $v_j$ is the vertex in the data graph adjacent to $e_j$. 
\item 
The remaining contractions in the contraction tree include $(U_i,V_i)$ such that both $U_i$ and $V_i$ are not adjacent to $\bar{E}$, and contractions where $U_i$ or $V_i$ is adjacent to at least two edges in $\bar{E}$, and the other one is not adjacent to any edge in $\bar{E}$. 
We let $\st{I}$ contain these contractions.
\end{itemize}

  \begin{wrapfigure}{r}{0.51\textwidth}
	\vspace{-5mm}
    \begin{minipage}{0.51\textwidth}
\begin{algorithm}[H]
\caption{Sketching algorithm}
\begin{algorithmic}[1]\label{alg:contract_sketch}
\STATE{\textbf{Input: } Input data tensor network $\Gd$, data contraction tree $T_0$ expressed in \eqref{eq:exp_T0}}
\FOR{each $e_i\in \bar{E}$}
\STATE{\Comment{Sketch with Kronecker product embedding}
\label{line:sketch_kronecker}
}
\STATE{$W\leftarrow \text{contract and sketch }X(e_j)$}
  \STATE{Replace the contraction output of $X(e_j)$ by $W$ in $T_0$}
\ENDFOR
\FOR{each contraction $(U_i, V_i)$ in $\st{S}\cup \st{I}$}{
   \IF{ $i\in \st{S}$}
\STATE{\Comment{Sketch with binary tree embedding}
\label{line:sketch_binary}
}
   \STATE{$W_i\leftarrow \text{contract and sketch }(U_i, V_i)$ (detailed in \cref{subsec:binary}) \label{line:sketch_s}}
   \ELSE
   \STATE{$W_i\leftarrow \text{contract}(U_i,V_i)$}
   \ENDIF
   \STATE{Replace the contraction output of $(U_i,V_i)$ by $W_i$ in $T_0$}
}
\ENDFOR 
\RETURN $W_{\Nd-1}$
\end{algorithmic}
\end{algorithm}
    \end{minipage}
	\vspace{-5mm}
  \end{wrapfigure}

The sketching algorithm is shown in \cref{alg:contract_sketch}, and the details  are as follows,
\begin{itemize}[leftmargin=*, itemsep=0pt]
    \item One matrix in the Kronecker product embedding  is used to sketch the sub data network $\stX{j}$, which guarantees that two sketch dimensions to be merged onto one tensor will both have size $\Theta(N\log(1/\delta)/\epsilon^2)$.
    For the case where $\stD{j} = \emptyset$, we directly sketch $X(e_j) = v_j$ using an embedding matrix. For the case where $\stD{j} \neq \emptyset$, we 
    select $k(e_j)\in \stD{j}$ and apply the sketching matrix during the contraction $(U_{k(e_j)}, V_{k(e_j)})$. The value of $k(e_j)$ is selected via an exhaustive search over all $|\stD{j}|$ contractions, so that sketching $\stX{j}$ has the lowest asymptotic cost. 
    \item One small tensor network (denoted as $Z_i$) represented by a binary tree vertex in the binary tree structured embedding  is used to sketch the contraction $(U_i,V_i)$ when $i\in \st{S}$, which means that both $U_i$ and $V_i$ are adjacent to $\bar{E}$.
    Let $\hat{U}_i,\hat{V}_i$ denote the sketched $U_i$ and $V_i$ formed in previous contractions in the sketching contraction tree $T_B$, such that $\hat{U}_i\cap \Vd = U_i$ and $\hat{V}_i\cap \Vd = V_i$,
    the structure of $Z_i$ is determined so that the asymptotic cost to sketch $(\hat{U}_i,\hat{V}_i)$ is minimized 
    under the constraint that $Z_i$ is in $\mathcal{G}^{(\epsilon/\sqrt{N},\delta)}$, so that it satisfies the $(\epsilon/\sqrt{N},\delta)$-accurate sufficient condition and only has one output dimension. 
In \cref{subsec:binary}, we provide an algorithm to construct $Z_i$ containing 2 tensors, so that the output sketch size of $Z_i$ is $\Theta(N\log(1/\delta)/\epsilon^2)$.
\end{itemize}

\paragraph{Analysis of the algorithm}

The embedding constructed during \cref{alg:contract_sketch} contains $\Theta(N)$ vertices, and the output sketch size is $m=\Theta(N\log(1/\delta)/\epsilon^2)$. Therefore, the  sketching result both has low sketch size and is $(\epsilon,\delta)$-accurate.
Below we discuss the optimality of \cref{alg:contract_sketch} in terms of the sketching asymptotic computational cost. 
We first discuss the case when each vertex in the data tensor network is adjacent to an edge in $\bar{E}$. 

\begin{theorem}\label{thm:optimal}
For data tensor networks where each vertex is adjacent to an edge in $\bar{E}$, the asymptotic cost of \cref{alg:contract_sketch} is optimal w.r.t. the optimization problem in \eqref{eq:cost_opt_unconstrained}.
\end{theorem}

We show the detailed proof of the theorem above in \cref{subsec:sketch_uniform}. Therefore, \cref{alg:contract_sketch} is efficient in sketching multiple widely used tensor network data, including tensor train, Kronecker product, and Khatri-Rao product. As we will discuss in \cref{sec:app}, \cref{alg:contract_sketch} can be used to design efficient sketching-based ALS algorithm for CP tensor decomposition.

Note that the embedding in \cref{alg:contract_sketch} may not be a tree embedding. As we will show in \cref{sec:exp}, for cases including sketching a Kronecker product data,  \cref{alg:contract_sketch} is more efficient than sketching with tree embeddings. On the other hand, for some data tensor networks, sketching with a tree embedding also yields the optimal asymptotic cost, which we will show in \cref{thm:tree_optimal}. 

For general input data where each data vertex may not adjacent to an edge in $\bar{E}$, \cref{alg:contract_sketch} may not yield the optimal sketching asymptotic cost, but is within a factor of at most $O(\sqrt{m})$ from the cost lower bound. Below we show the theorem, and the detailed proof is in \cref{subsec:sketch_general}.

\begin{theorem}\label{thm:approx_factor}
For any data tensor network $\Gd$, the asymptotic cost of \cref{alg:contract_sketch} (denoted as $c$) satisfy
$
c = O\left(\sqrt{m}\cdot c_{\text{opt}}
\right)
$, where $c_{\text{opt}}$ is the optimal asymptotic computational cost for the optimization problem \eqref{eq:cost_opt_unconstrained} and $m = \Theta(N\log(1/\delta)/\epsilon^2)$. 
When $\Gd$ is a graph, $
c = O\left(m^{0.375}\cdot c_{\text{opt}}
\right)
$.
\end{theorem}

\paragraph{Efficiency of tree tensor network embedding}
We discuss cases where tree tensor network embeddings can be optimal  w.r.t. the optimization problem in \eqref{eq:cost_opt_unconstrained}. Tree embeddings, in particular the tensor train embedding, have been widely discussed and used in prior work~\cite{rakhshan2020tensorized,daas2021randomized,batselier2018computing}. We design an algorithm to sketch with tree embeddings. 
The algorithm is similar to \cref{alg:contract_sketch}, and the only difference is that for each contraction $(U_i,V_i)$ with $i\in \st{S}$, 
such that both $U_i$ and $V_i$ are adjacent to edges in $\bar{E}$, we sketch it with one tensor rather than a small network. Below, we present the optimality of the algorithm in terms of sketching asymptotic cost.

\begin{theorem}\label{thm:tree_optimal}
Consider $\Gd$ with each vertex adjacent to an edge to be sketched and its given contraction tree  $T_0$. If each contraction in $T_0$ contracts dimensions with size being at least the sketch size, then sketching with tree embedding would yield the optimal asymptotic cost for \eqref{eq:cost_opt_unconstrained}.
\end{theorem}

We present the proof of \cref{thm:tree_optimal} in \cref{sec:appendix_tree}. As we will show in \cref{sec:exp}, for tensor network data with relatively large contracted dimension sizes such that the condition in \cref{thm:tree_optimal} is satisfied, sketching with tree embedding yields a similar performance as \cref{alg:contract_sketch}. However, for data where the condition in \cref{thm:tree_optimal} is not satisfied, \cref{alg:contract_sketch} is more efficient.
For example, when the data is a vector with a Kronecker product structure, sketching with \cref{alg:contract_sketch} yields a cost of $\Theta(
\sum_{j= 1}^{N} s_jm + 
Nm^{2.5}
)$ and sketching with a tree embedding yields a cost of  $\Theta(
\sum_{j= 1}^{N} s_jm + 
Nm^{3}
)$. We present the detailed analysis in \cref{subsec:appendix_costanalysis} and \cref{sec:appendix_tree}.

\section{Applications}\label{sec:app}

\paragraph{Alternating least squares for CP decomposition} 

On top of \cref{alg:contract_sketch}, we propose a new sketching-based ALS algorithm for CP tensor decomposition.
Throughout analysis we assume the input tensor is dense, and has order $N$ and size $s \times \cdots \times s$, and the CP rank is $R$.
The goal of CP decomposition is to minimize the objective function,
$  f(\mat{A}_1, \ldots, \mat{A}_N) = 
  \left\|
  \tsr{X} - \sum_{r=1}^{R} \mat{A}_1(:,r)\circ \cdots \circ \mat{A}_N(:,r)
  \right\|_F^2,$
where $\mat{A}_i\in\R^{s\times R}$ for $i\in[N]$ are called factor matrices, and $\tsr{X}$ denotes the input tensor. 
In each iteration of ALS, $N$ subproblems are solved sequentially, and the $i$th subproblem can be formulated as
$    \mat{A}_i = \arg\min_{\mat{A}}\left\|L_i\mat{A}{}^T - R_i\right\|^2_F, $
where $L_i= \mat{A}_1 \odot \cdots \odot  \mat{A}_{i-1}  \odot  \mat{A}_{i+1}\odot \cdots \odot \mat{A}_N$ consists of a chain of Khatri-Rao products, and
$R_i=\mat{X}_{(i)}^T$ is the transpose of $i$th matricization of $\tsr{X}$.

Multiple sketching-based randomized algorithms are proposed to accelerate each subproblem in CP-ALS~\cite{battaglino2018practical,larsen2020practical,malik2021more}. The sketched problem can be formulated as
$\mat{A}_i = \argmin{\mat{A}} \left\|S_iL_iA^T - S_iR_i\right\|_F^2,$
where $S_i$ is an embedding. The goal is to design $S_i$ such that the sketched subproblem can be solved efficiently and accurately.  In \cref{tab:alscompare}, we summarize two state-of-the-art sketching methods for CP-ALS. Larsen and Kolda~\cite{larsen2020practical} propose a method that sketches the subproblem based on (approximate) leverage score sampling (LSS), but both the per-iteration computational cost and the sketch size sufficient for $(\epsilon,\delta)$-accurate solution has an exponential dependence on $N$, which is inefficient for decomposing high order tensors. Malik~\cite{malik2021more} proposes a method called recursive leverage score sampling for CP-ALS, where the embedding contains two parts, $S_i = S_{i,1}S_{i,2}$, and $S_{i,2}$ is an embedding with a binary tree structure proposed in~\cite{ahle2020oblivious} with sketch size 
$\Theta(NR^2/\delta)$, and $S_{i,1}$ performs approximate leverage score sampling on $S_{i,2}L_i$ with sketch size $\Tilde{\Theta}(NR/\epsilon^2)$. This sketching method has a better dependence on $R$ in terms of per-iteration cost. 
For both algorithms, the preparation cost shown in \cref{tab:alscompare} denotes the cost to go over all elements in the tensor and initialize factor matrices using randomized range finder. As is shown in \cite{larsen2020practical,ma2021fast}, randomized range finder based initialization is critical for achieving accurate CP decomposition with sampling-based sketched ALS.

\begin{table}[!ht]
\small
  \begin{center}
    \renewcommand{\arraystretch}{1.2}
    {
    \begin{tabular}{l|l|l|l}
\hline
      CP-ALS algorithm  & Per-iteration cost & Sketch size ($m$) & Prep cost\\ \hline
      Standard ALS  & $\Theta(s^NR)$ & / & / \\ \hline
      LSS~\cite{larsen2020practical}  & 
      $\Tilde{\Theta}(N(R^{N+1} + sR^N)/\epsilon^2)$
      & 
      $\Tilde{\Theta}(R^{N-1}/\epsilon^2)$ 
      & $\Theta(s^N)$ \\ \hline
      Recursive LSS \cite{malik2021more}&  $\Tilde{\Theta}(N^2(R^4 + NsR^3/\epsilon)/\delta)$  & 
      $\Theta(NR^2/\delta)$ and $\Tilde{\Theta}(R/(\epsilon\delta))$ & $\Theta(s^N)$\\ \hline
      \cref{alg:contract_sketch} & 
      $
      \Tilde{\Theta}(N^2(N^{1.5}
       {\color{red} R^{3.5}}
      /\epsilon^3 + s
      {\color{red} R^2}
      )/\epsilon^2)
      $
      & $\Tilde{\Theta}(NR/\epsilon^2)$
      & $\Theta(s^Nm)$   \\ \hline
    \end{tabular}
    }
    \renewcommand{\arraystretch}{1}
  \end{center}
\caption{Comparison of asymptotic algorithmic complexity between standard CP-ALS, CP-ALS with leverage score sampling (LSS), CP-ALS with recursive leverage score sampling (recursive LSS), and sketching CP-ALS with \cref{alg:contract_sketch}.
The  third column
shows the sketch size sufficient for the sketched linear least squares to be $(1+\epsilon)$-accurate with probability at least $1-\delta$. By using $\Tilde{\Theta}$, we neglect logarithmic factors, including $\log(R)$ and $\log(1/\delta)$. 
}
\vspace{-7mm}
\label{tab:alscompare}
\end{table}

We propose a new sketching algorithm for CP-ALS based on \cref{alg:contract_sketch}. Each $S_i$ is generated on top of the data tensor network
$L_i$ and its given data contraction tree $T_i$, with the sketch size being $m=\Theta(NR\log(1/\delta)/\epsilon^2) = \Tilde{\Theta}(NR/\epsilon^2)$. The contraction trees $T_i$ for $i\in [N]$ are chosen in a fixed alternating order, such that the resulting embeddings $S_i$ for $i\in [N]$ have common parts and allow
reusing contraction intermediates. We leave the detailed analysis in \cref{subsec:cp_contractiontree}.

The ALS per-iteration cost is $\Theta(N(m^{2.5}R+smR)) = \Tilde{\Theta}(N^2(N^{1.5}R^{3.5}/\epsilon^3 + sR^2)/\epsilon^2)$. We present the detailed sketching algorithm and its cost analysis in \cref{subsec:cp_algorithm}.
When performing a low-rank CP decomposition with $s\gg R^{1.5}$ and $\epsilon$ is not too small so that $\epsilon = \Theta(1)$\footnote{As is shown in \cite{ma2021fast}, in practice, setting $\epsilon$ to be 0.1-0.2 will result in accurate sketched least squares with relative
residual norm error less than 0.05.}, the per-iteration cost is dominated by the term $\Tilde{\Theta}(N^2 sR^2/\epsilon^2)$, which is $\Theta(NR\epsilon/\delta)=\Omega(NR)$ times better than the per-iteration cost of the recursive LSS algorithm. For another case of a high-rank CP decomposition with $R\gg s$, which happens when one wants a high-accuracy CP decomposition of high order tensors, the per-iteration cost of our sketched CP-ALS algorithm is dominated by the term $\Tilde{\Theta}(N^{3.5}R^{3.5}/\epsilon^5)$, and the cost ratio between this algorithm and the recursive LSS algorithm is $\Tilde{\Theta}(N^{1.5}\delta/(\epsilon^5R^{0.5}))$. For this case, our algorithm is only preferable when $N^{1.5}\delta/\epsilon^5$ is not too large compared to $R^{0.5}$.

Although our proposed sketching algorithm yields better per-iteration asymptotic cost in multiple regimes compared to existing leverage score based sketching algorithms,
some preparation computations are needed to sketch right-hand-sides $S_iR_i$ for $i\in [N]$ before ALS iterations, and this cost is non-negligible. 
On the other hand, this algorithm has better parallelism, since it involves a sequence of matrix multiplications rather than sampling the matrix. 
We leave the detailed experimental comparison of computational efficiency of different sketching techniques for future work.

\paragraph{Tensor train rounding}
Given a tensor train, \textit{tensor train rounding} finds a tensor train with a lower rank to approximate the original representation.
Throughout analysis we assume the tensor train has order $N$ with the output dimension sizes equal $s$, the tensor train rank is $R<s$, and the goal is to round the rank to $r<R$. 
The standard tensor train rounding algorithm~\cite{oseledets2011tensor} consists of a right-to-left sweep of QR decompositions of the input tensor train (also called orthogonalization), and another left-to-right truncated singular value decompositions (SVD) sweep to perform rank reduction. 
The orthogonalization step is the bottleneck of the rounding algorithm, and costs $\Theta(NsR^3)$. 
Recently, \cite{daas2021randomized} has introduced a randomized rounding algorithm called ``\textit{Randomize-then-Orthogonalize}". Let $\mat{X}$ denote a matricization of  the tensor train data with all except one dimension at the end grouped into the row, the algorithm first sketches $X$ with 
a tensor train embedding $S$, then performs a sequence of truncated SVDs on top of $SX$.
The sketch size $m$ of $S$ is $r$ plus some constant, and is assumed to be smaller than $R$.
The bottleneck is to compute $SX$, which costs $\Theta(NsR^2m)$. 
 
We can recast the problem as finding an embedding satisfying the linearization sufficient condition with sketch size $m$, such that the asymptotic cost of computing $SX$ is optimal given the data contraction tree that contracts the tensor train from one end to another.
Our analysis (detailed in \cref{sec:tt}) shows that the sketching cost for the problem is lower bounded by $\Omega(NsR^2m)$, thus the sketching algorithm in \cite{daas2021randomized} attains the asymptotic cost lower bound and is efficient. 
Note that sketching with \cref{alg:contract_sketch} yields the same asymptotic cost, despite using a different embedding.

%% file: exp.tex
We compare our proposed embeddings with embeddings discussed in the reference.
The experiments are used to justify the theoretical analysis in \cref{thm:optimal} and \cref{thm:tree_optimal}.
We test the sketching performance 
on tensor train inputs and Kronecker product inputs.
Our experiments are carried out on an Intel Core i7 2.9 GHz Quad-Core machine using NumPy~\cite{oliphant2006numpy} routines in Python.

  \begin{wrapfigure}{r}{0.71\textwidth}
\centering
\includegraphics[width=0.35\textwidth, keepaspectratio]{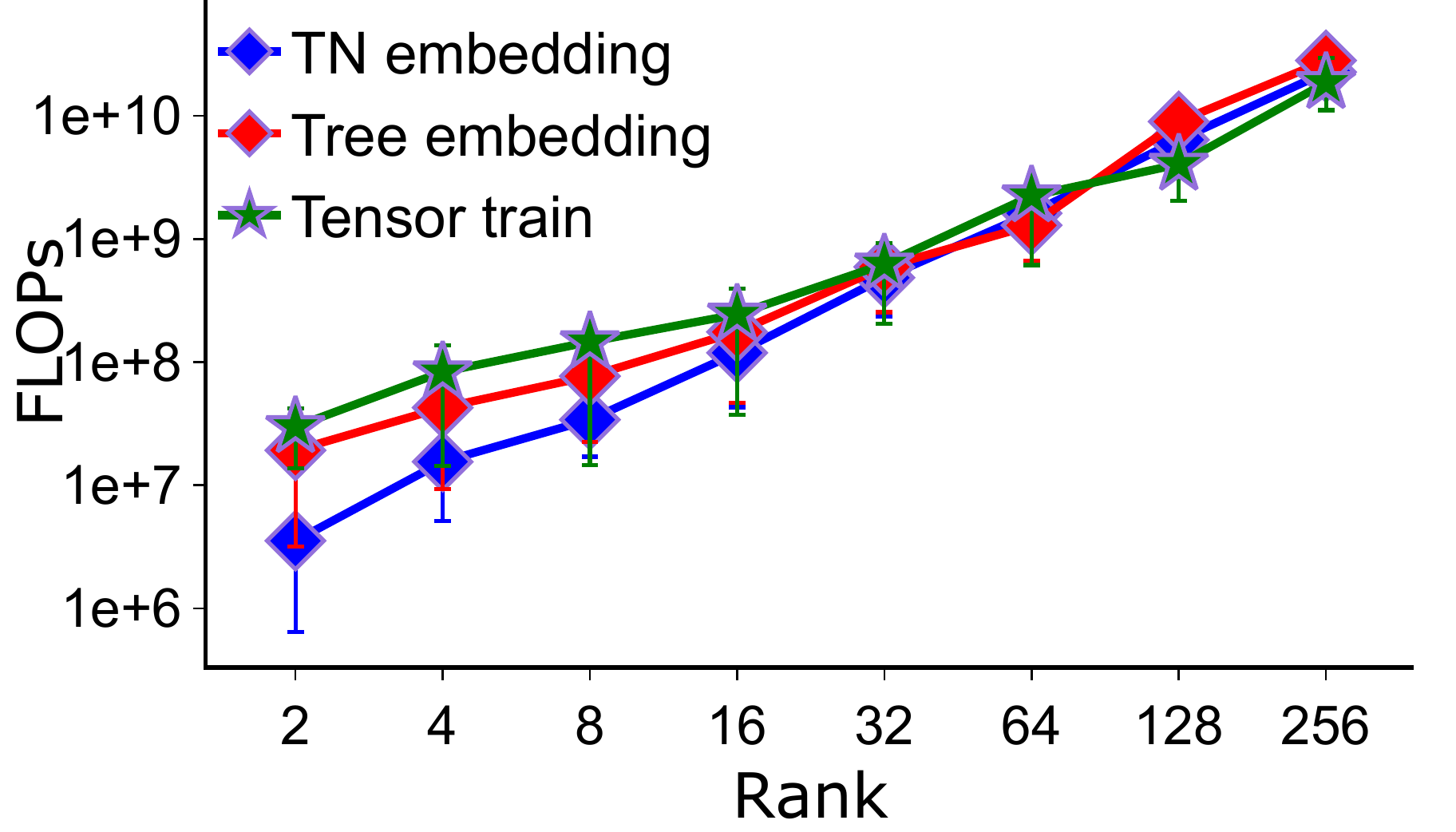}
\includegraphics[width=0.35\textwidth, keepaspectratio]{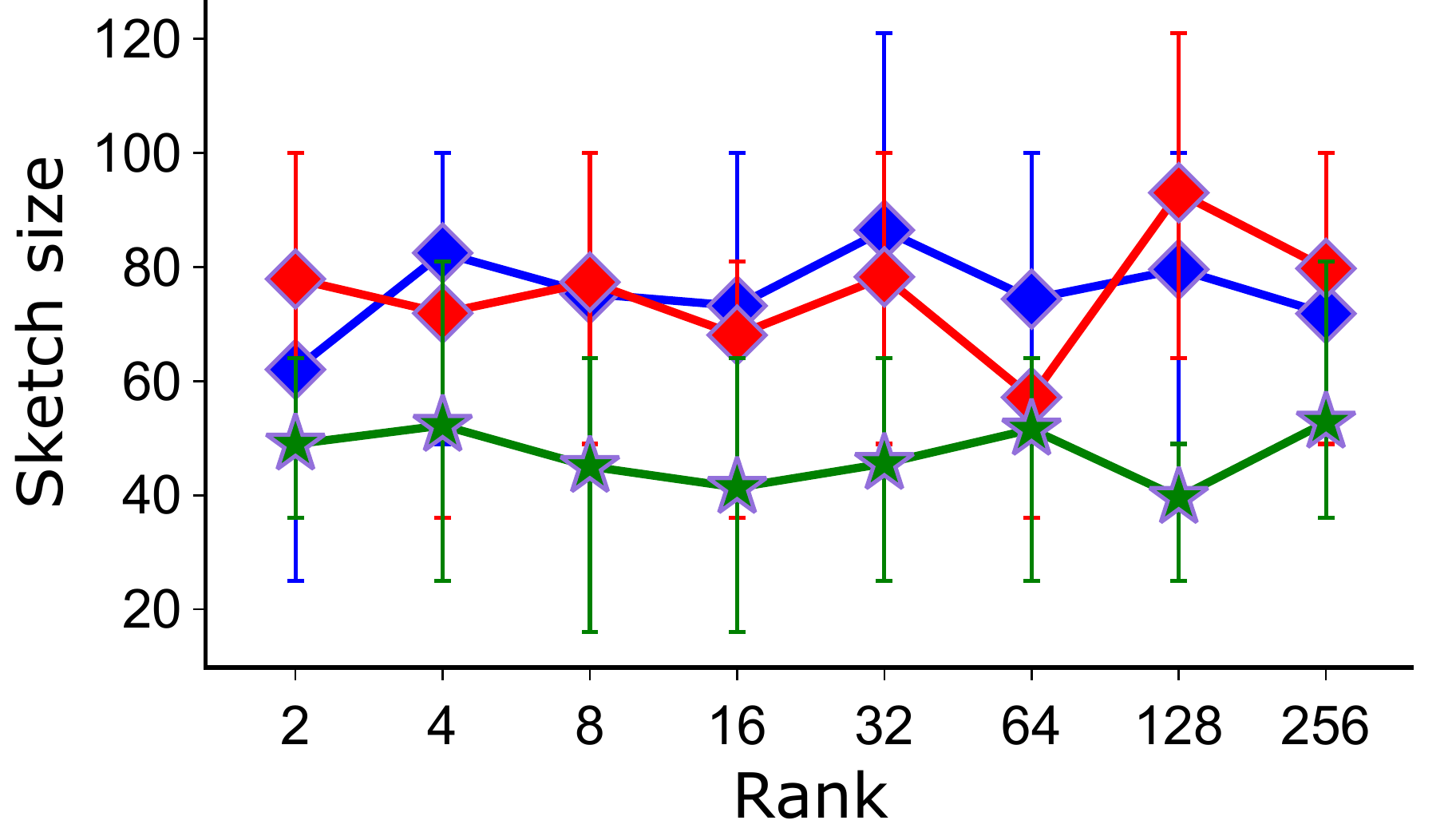}
	\vspace{-3mm}
  \caption{Results for sketching tensor train inputs. Each point denotes the mean value across 25 experiments, and each error bar shows the 25th-75th quartiles.
  }
  \label{fig:exp1}
	\vspace{-2mm}
  \end{wrapfigure}

We compare the performance of general tensor network embedding used in \cref{alg:contract_sketch} (called TN embedding), tree embedding discussed in \cref{thm:tree_optimal}, and tensor train embedding~\cite{daas2021randomized} in sketching tensor train input data in \cref{fig:exp1}. The input tensor train data has order 6, and the dimension size is 500. We test the sketching performance under different tensor train ranks. 
For a given rank, we randomly generate 25 different inputs, with each element in each tensor being an i.i.d. variable uniformly distributed within $[0, 1]$.
For each input $x$ and a specific embedding structure, we calculate the relative sketching error twice under different sketch sizes, and record the smallest sketch size such that both of its relative sketching errors are within 0.2, $\frac{\|Sx\|_2}{\|x\|_2}\leq 0.2$.
We also calculate the number of floating point operations (FLOPs) for computing $Sx$ under the smallest sketch size based on the classical dense matrix multiplication algorithm. As can be seen, tree and tensor train embeddings are as efficient as TN embedding in terms of number of FLOPs under relatively high tensor train rank (when rank is at least 32), but are less efficient than TN embedding when the tensor train rank is lower than 32. The results are consistent with the theoretical analysis in \cref{thm:tree_optimal}, which shows that tree embeddings yield the optimal asymptotic cost when the input tensor train rank is at least the output sketch size, but the asymptotic cost is not optimal when the tensor train rank is low.

  \begin{wrapfigure}{r}{0.71\textwidth}
\vspace{-2mm}
\centering
\includegraphics[width=0.35\textwidth, keepaspectratio]{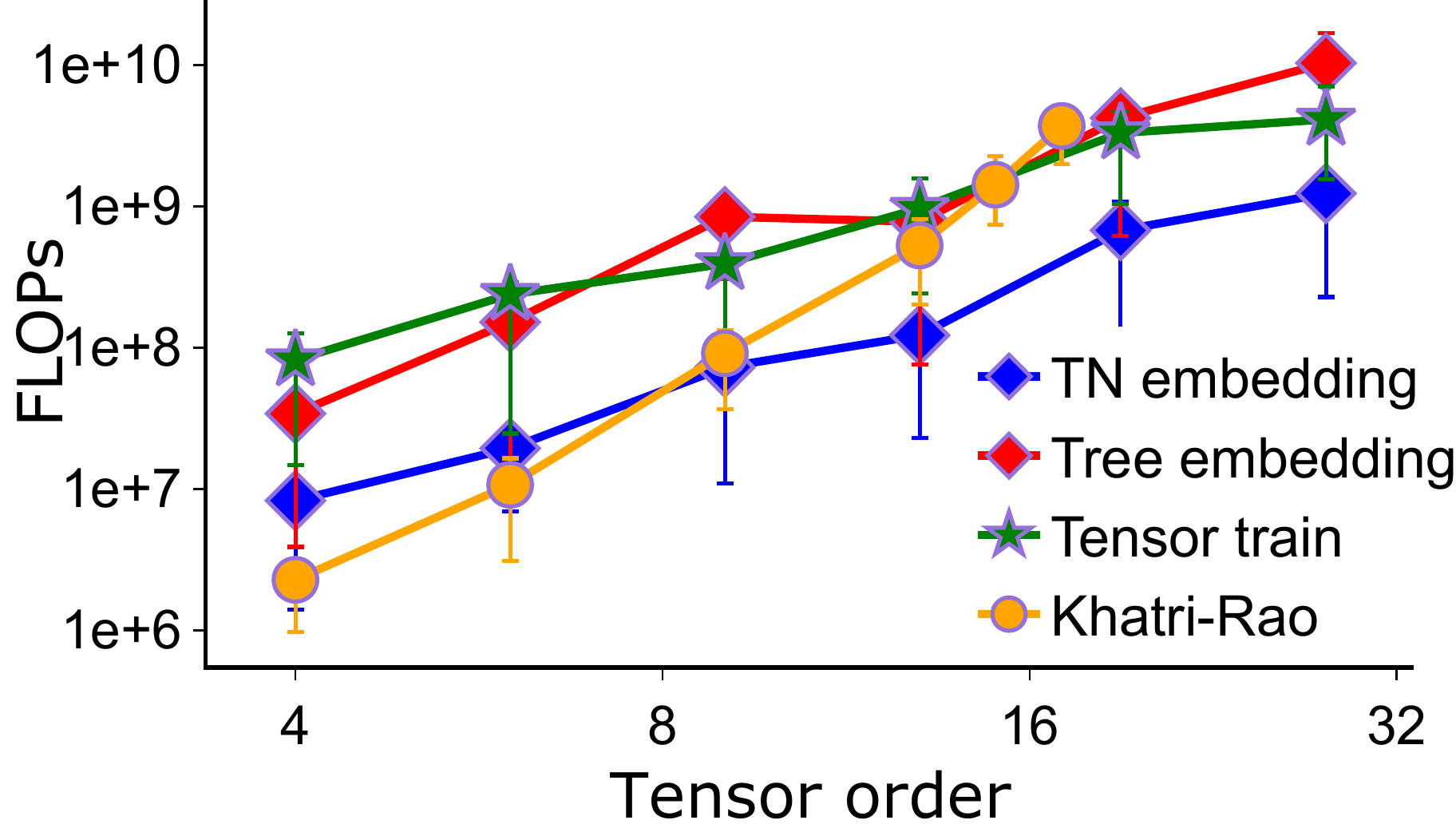}
\includegraphics[width=0.35\textwidth, keepaspectratio]{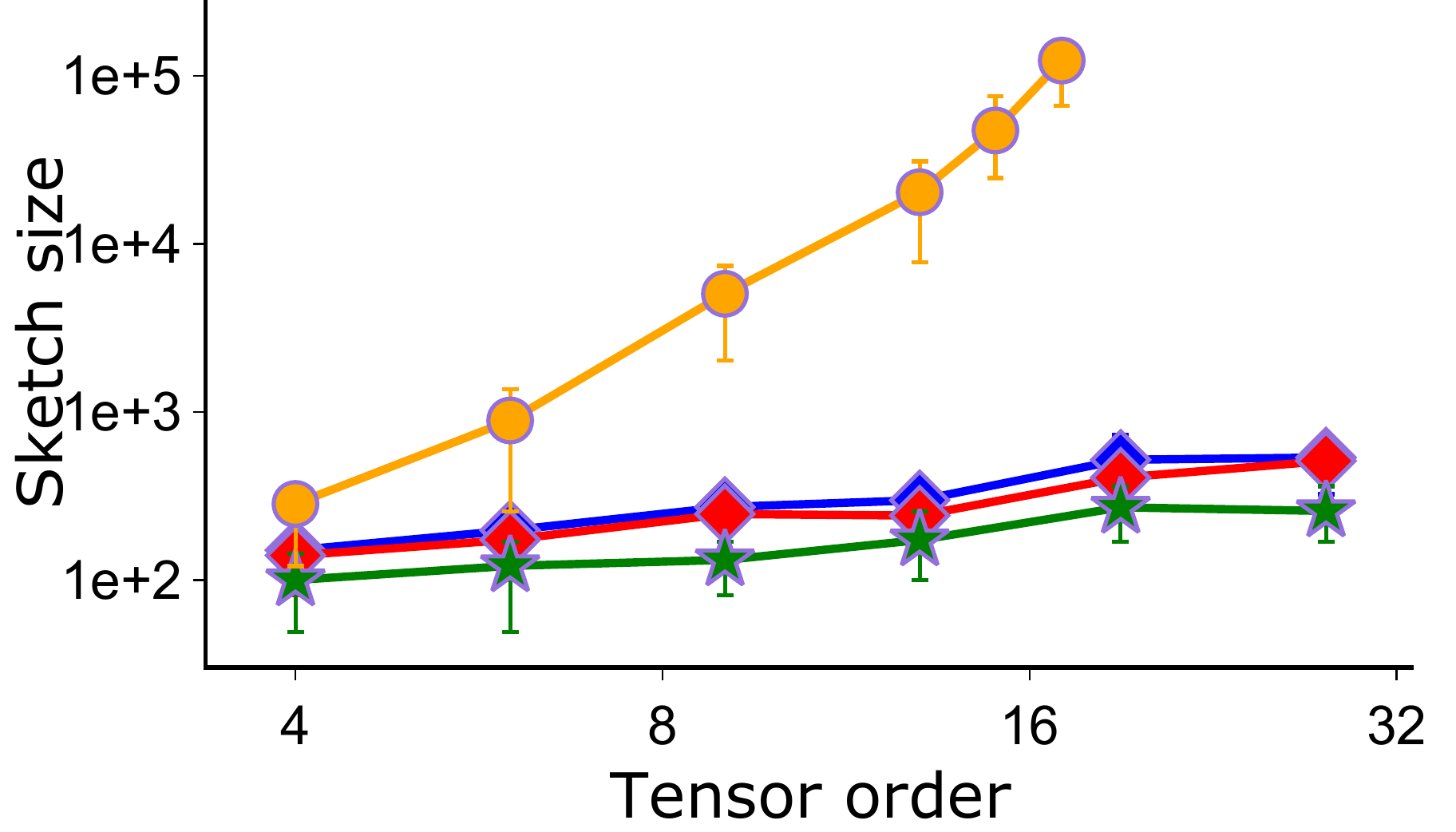}
	\vspace{-3mm}
  \caption{Results for sketching Kronecker product inputs.
  }
  \label{fig:exp2}
	\vspace{-2mm}
  \end{wrapfigure}

We also compare the performance of  TN, tree, tensor train and Khatri-Rao product embeddings in sketching Kronecker product inputs in \cref{fig:exp2}. Each dimension size of the Kronecker product input is fixed to be 1000, and we test the sketching performance under different tensor orders. 
For each input $x$ and a specific embedding structure, we record the smallest sketch size such that its relative sketching error is within 0.1.
As can be seen, compared to Khatri-Rao product embedding, the sketch size of 
TN, tree and tensor train embeddings all increase slowly with the increase of tensor order, consistent with the theoretical analysis that these embeddings have efficient sketch size. 
In addition, the cost in FLOPs of TN embedding is smaller than tree and tensor train embeddings. This is consistent with the analysis in \cref{thm:optimal} and its following discussions, showing that TN embedding yields the optimal asymptotic cost for Kronecker product inputs, but tree and tensor train embeddings do not.

%% file: conclu.tex
We provide detailed analysis of general tensor network embeddings. 
For input data such that each dimension to be sketched has size greater than the sketch size, we provide an algorithm to efficiently sketch such data using Gaussian embeddings that can be linearized into a sequence of sketching matrices and have low sketch size.
Our sketching method is then used to design  state-of-the-art sketching algorithms for CP tensor decomposition and tensor train rounding.
We leave the analysis for more general embeddings for future work, including those with each tensor representing
fast sketching techniques, such as Countsketch and fast JL transform using fast Fourier transform, and those containing structures cannot be linearized, such as Khatri-Rao product embedding.
It would also be of interest to look at other tensor-related applications that could benefit from tensor network embedding, including tensor ring decomposition and simulation of quantum circuits.

%% file: background.tex
\subsection{Tensor algebra and tensor diagram notation}
\label{sec:notations}

Our analysis makes use of tensor algebra 
for tensor operations~\cite{kolda2009tensor}.
Vectors are denoted with lowercase Roman letters (e.g., $\vcr{v}$), matrices are denoted with uppercase Roman letters (e.g., $\mat{M}$), and tensors are denoted with calligraphic font (e.g., $\tsr{T}$). 
An order $N$ tensor corresponds to an $N$-dimensional array. 
For an order $N$ tensor $\tsr{T}\in \R^{s_1\times\cdots\times s_N}$, the size of $i$th dimension is $s_i$.
The $i$th column of the matrix $\mat{M}$ is denoted by $\mat{M}(:,i)$, and the $i$th row is denoted by $\mat{M}(i,:)$.
Subscripts are used to label different vectors, matrices and tensors (e.g. $\tsr{T}_1$ and $\tsr{T}_2$ are unrelated tensors). 
The Kronecker product of two vectors/matrices is denoted with $\otimes$, and the outer product of two or more vectors is denoted with $\circ$.
For matrices $\mat{A}\in \mathbb{R}^{m\times k}$ and $\mat{B}\in \mathbb{R}^{n\times k}$, their Khatri-Rao product results in a matrix of size $(mn)\times k$ defined by
$
  \mat{A}\odot \mat{B} = [\mat{A}(:,1)\otimes \mat{B}(:,1),\ldots, \mat{A}(:,k)\otimes \mat{B}(:,k)] .
$
Matricization is the process of unfolding a tensor into a matrix. The dimension-$n$ matricized version of $\tsr{T}$ is denoted by $\mat{T}_{(n)}\in \mathbb{R}^{s_n\times K}$ where $K=\prod_{m=1,m\neq n}^N s_m$. 

We introduce the graph representation for tensors, which is also called tensor diagram~\cite{bridgeman2017hand}. 
A tensor is represented by a vertex with hyperedges adjacent to it, each corresponding to a tensor dimension.
A matrix $\mat{M}$ and an order four tensor $\tsr{T}$ are represented as follows,
\begin{align*}
   \mat{M}
    \implies  
	\diagramsized{0.25}{
		\draw (0:0) -- (0:2);
		\draw (0:0) -- (0:-2);
		\draw[tengrey] circle (1);	
	}
\quad \quad \quad
	\tsr{T}
	 \implies  
	\diagramsized{0.25}{
		\draw (0:0) -- (90:2);
		\draw (0:0) -- (-90:2);
		\draw (0:0) -- (30-90:2);
		\draw (0:0) -- (-30-90:2);
		\draw[tengrey] circle (1);	
	}.
\end{align*}
The Kronecker product of two matrices $\mat{A}$ and $\mat{B}$ can be expressed as
\begin{align*}
\diagram{
\def\r{0.5}
\coordinate (c1) at (0, 0);
\coordinate (c2) at (2, 0);
\draw[tengrey] (c1) circle (\r) node {$A$};
\draw[shift=(c1)] ( 90:\r) -- ( 90:4*\r);
\draw[shift=(c1)] (270:\r) -- (270:4*\r);
\draw[tengrey] (c2) circle (\r) node {$B$};
\draw[shift=(c2)] ( 90:\r) -- ( 90:4*\r);
\draw[shift=(c2)] (270:\r) -- (270:4*\r);
\draw[dashed] ($0.5*(c1)+0.5*(c2)$) ellipse (2.25cm and 1cm);
}
=
\diagram{
\def\r{0.5}
\coordinate (c1) at (0, 0);
\coordinate (c2) at (2, 0);
\draw[shift=(c1)] ( 90:\r) -- ( 90:4*\r);
\draw[shift=(c1)] (270:\r) -- (270:4*\r);
\draw[shift=(c2)] ( 90:\r) -- ( 9 0:4*\r);
\draw[shift=(c2)] (270:\r) -- (270:4*\r);
\draw[tengrey] ($0.5*(c1)+0.5*(c2)$) ellipse (2.25cm and 1cm) node {$A\otimes B$};
}.
\end{align*}
Connecting two edges means two tensor dimensions are contracted or summed over. One example is shown in \cref{fig:notation2}.
\begin{figure}[!ht]
\centering
\includegraphics[width=.4\textwidth, keepaspectratio]{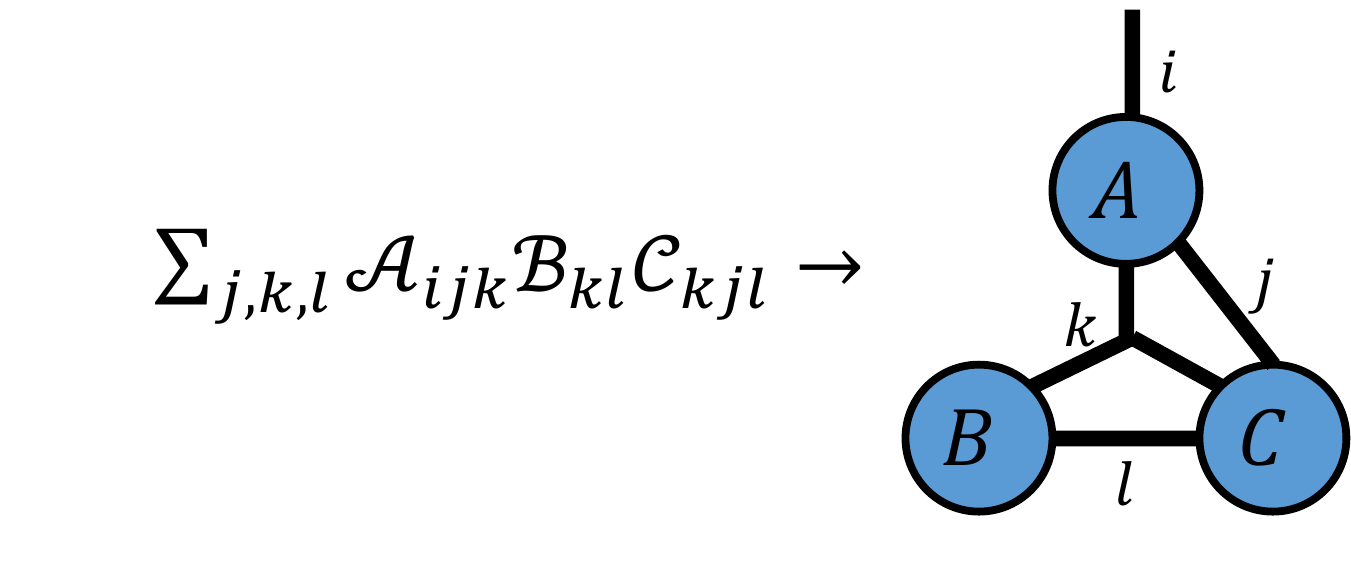}
  \caption{An example of tensor diagram notation.
  }
  \label{fig:notation2}
\end{figure}

\subsection{Background on sketching}\label{subsec:background_sketching}

In this section, we introduce definitions for sketching used throughout the paper. 
\begin{definition}[Gaussian embedding]\label{def:subgaussian}
A matrix $\mat{S} = \frac{1}{\sqrt{m}}\mat{M} \in \R^{m \times n}$ is a Gaussian embedding if each element of $\mat{M}$ is a normalized Gaussian random variable, $\mat{M}(i,j)\sim N(0,1)$. 
\end{definition}

One key property we would like the tensor network embedding to satisfy is the 
($\epsilon$, $\delta$)-accurate property. To achieve this, one central property each tensor in the tensor network embedding needs to satisfy is the Johnson-Lindenstrauss (JL) moment property. 
The JL moment property captures a bound on the moments of the
difference between the  vector Euclidean norm and the norm after sketching.
We introduce both definitions below.

\begin{definition}[($\epsilon$, $\delta$)-accurate embedding]\label{def:embedding}
A random matrix $S\in\R^{m\times n}$ has the $(\epsilon, \delta)$-accurate embedding property if for every $\vcr{x} \in \R^n$ with $\|\vcr{x}\|_2 = 1$,
 \[
 \Prob_{\mat{S}} \left(
 \left|\left\|\mat{Sx}\right\|_2^2 - 1\right| > \epsilon
 \right) < \delta.
 \]
 \end{definition}

\begin{definition}[($\epsilon$, $\delta$, $p$)-JL moment~\cite{kane2011almost,kane2014sparser}]\label{def:jl_moment}
A random matrix $S\in\R^{m\times n}$ has the $(\epsilon, \delta, p)$-JL moment property if for every $\vcr{x} \in \R^n$ with $\|\vcr{x}\|_2 = 1$,
\[
\E_{\mat{S}} \left|\left\|\mat{Sx}\right\|_2^2 - 1\right|^p < \epsilon^p \delta \quad \text{and} \quad \E\left[\left\|\mat{Sx}\right\|_2^2\right] = 1.
\]
 \end{definition}
 
\begin{definition}[Strong ($\epsilon$, $\delta$)-JL moment~\cite{kane2011almost,ahle2020oblivious}]\label{def:strong_jl_moment}
A random matrix $S\in\R^{m\times n}$ has the strong $(\epsilon, \delta)$-JL moment property if for every $\vcr{x} \in \R^n$ with $\|\vcr{x}\|_2 = 1$, and every integer $p\in [2, \log(1/\delta)]$,
\begin{equation}\label{eq:strong_jlmoment}
  \E_{\mat{S}} \left|\left\|\mat{Sx}\right\|_2^2 - 1\right|^p < \left(\frac{\epsilon}{e}\right)^p\left(\frac{p}{\log(1/\delta)}\right)^{p/2}
\end{equation}
and $\E\left[\left\|\mat{Sx}\right\|_2^2\right] = 1$.
\end{definition}

Note that the strong $(\epsilon, \delta)$-JL moment property directly reveals the $(\epsilon, \delta, \log(1/\delta))$-JL moment property, since letting $p = \log(1/\delta)$, \eqref{eq:strong_jlmoment} becomes
\[
\E_{\mat{S}} \left|\left\|\mat{Sx}\right\|_2^2 - 1\right|^{\log(1/\delta)} < \left(\frac{\epsilon}{e}\right)^{\log(1/\delta)} = \epsilon^p\delta.
\]
Both the strong  ($\epsilon$, $\delta$)-JL moment property
and the  ($\epsilon$, $\delta$, $p$)-JL moment property
directly imply ($\epsilon$, $\delta$)-accurate embedding via Markov's inequality, 
 \[
 \Prob_{\mat{S}} \left(
 \left|\left\|\mat{Sx}\right\|_2^2 - 1\right| > \epsilon
 \right) < \frac{\E\left|\left\|\mat{Sx}\right\|_2^2 - 1\right|^p}{\epsilon^p} < \delta.
 \]

The lemmas below show that Gaussian embeddings can be used to construct embeddings with the JL moment property. 
\begin{lemma}[Strong JL moment of Gaussian embeddings~\cite{kane2011almost}]\label{lem:subgaussian_strongjl}
Gaussian embeddings with $m=\Omega(\log(1/\delta)/\epsilon^2)$ satisfy the $(\epsilon,\delta)$-strong JL moment property.  
\end{lemma}

Below we review the composition rules of JL moment properties introduced in \cite{ahle2020oblivious}, which are used to prove the ($\epsilon,\delta$)-accurate sufficient condition in \cref{thm:subgaussian_embedding}.

\begin{lemma}[JL moment with Kronecker product]\label{lem:kronecker_jl}
If a matrix $\mat{S}$ has the $(\epsilon, \delta, p)$-JL moment property, then  the matrix $\mat{M} = \mat{I}_i \otimes \mat{S} \otimes \mat{I}_{j}$ also has the $(\epsilon, \delta, p)$-JL moment property for identity matrices $\mat{I}_i$ and $\mat{I}_j$ with any size. This relation also holds for the strong $(\epsilon, \delta)$-JL moment property.
\end{lemma}

\begin{lemma}[Strong JL moment with matrix product]\label{lem:composition_strongjl}
There exists a universal constant $L$, such that for any constants $\epsilon, \delta \in [0, 1]$ and any integer $k$, if $\mat{M}_1 \in \R^{d_2\times d_1}, \cdots, \mat{M}_k \in \R^{d_{k+1}\times d_k}$ are independent random matrices, each having the strong
$\left(\frac{\epsilon}{L\sqrt{k}}, \delta\right)$-JL moment property, then the product matrix $\mat{M} = \mat{M}_k \cdots \mat{M}_1$ satisfies the strong $(\epsilon, \delta)$-JL moment property.
\end{lemma}

%% file: linearization.tex
In this section, we introduce definitions and basic properties of tensor network embeddings. These properties will be used in \cref{sec:alg_detail} and \cref{sec:lowerbound} for detailed computational cost analysis. The notation defined in the main text is summarized in \cref{tab:notations}, which is also used in later analysis.

\begin{table}[!ht]
  \begin{center}
    \renewcommand{\arraystretch}{1.6}
    {
    \begin{tabular}{l|l}
\hline
      Notations  & Meanings \\ \hline
      $S, S_i$  & Embedding matrix \\ \hline
      $m$  & Sketch size
       \\ \hline
      $\Ge=(\Ve,\Ee,w)$ &  Embedding tensor network  \\ \hline
      $\Gd=(\Vd,\Ed,w)$ &  Input data tensor network  \\ \hline
      $\bar{E} = \{e_1,\ldots, e_N\}$ & 
     Set of edges to be sketched
         \\ \hline
      $s_i$ & 
     Size of $e_i$ in $\bar{E}$
         \\ \hline
      $T_0$ & 
     Given data contraction tree
         \\ \hline
    $\st{D}(e_1),\ldots, \st{D}(e_N), \st{S}, \st{I}$ & 
    Subsets of contractions in $T_0$ \\ \hline
    $X(e_i)$  & Sub network contracted by $\st{D}(e_i)$ \\ \hline
    \end{tabular}
    }
    \renewcommand{\arraystretch}{1}
  \end{center}
\caption{Notations used throughout the paper.
}
\label{tab:notations}
\end{table}

\subsection{Graph notation 
for tensor network and tensor contraction
}\label{subsec:graph_notation}

We use undirected hypergraphs to represent tensor networks.
For a given hypergraph $G=(V,E,w)$, $V$ represents the vertex set, $E$ represents the set of hyperedges, and $w$ is a function such that $w(e)$ is the natural logarithm of the tensor dimension size represented by the hyperedge $e\in E$. 
We use $E(u,v)$ to denote 
the set of hyperedges adjacent to both $u$ and $v$, which includes the edge $(u,v)$ and hyperedges adjacent to $u,v$.
We use $E(A,B)$ to denote the set of hyperedges connecting two subsets $A,B$ of $V$ with $A\cap B = \emptyset$. 
We use $E(A, *)$ to denote all uncontracted edges only adjacent to $A$, $E(A,*) = \{(u)\in E: u\in A\}$. we illustrate $E(A,B), E(A,*)$ in \cref{fig:E_AB_example}.
For any set $A\subseteq V$, we let 
\begin{align}\label{eq:ES}
    E(A) &= E(A, V\setminus A) \cup E(A, *).
\end{align}
A tensor network implicitly represents a tensor with a set of (small) tensors and a specific contraction pattern.
 We use $G[A]=(A,E_A,w)$ to denote a sub tensor network defined on $A\subseteq V$, where $E_A$ contains all hyperedges in $E$ adjacent to any $v\in A$.

\begin{figure}
\centering
\includegraphics[width=.35\textwidth, keepaspectratio]{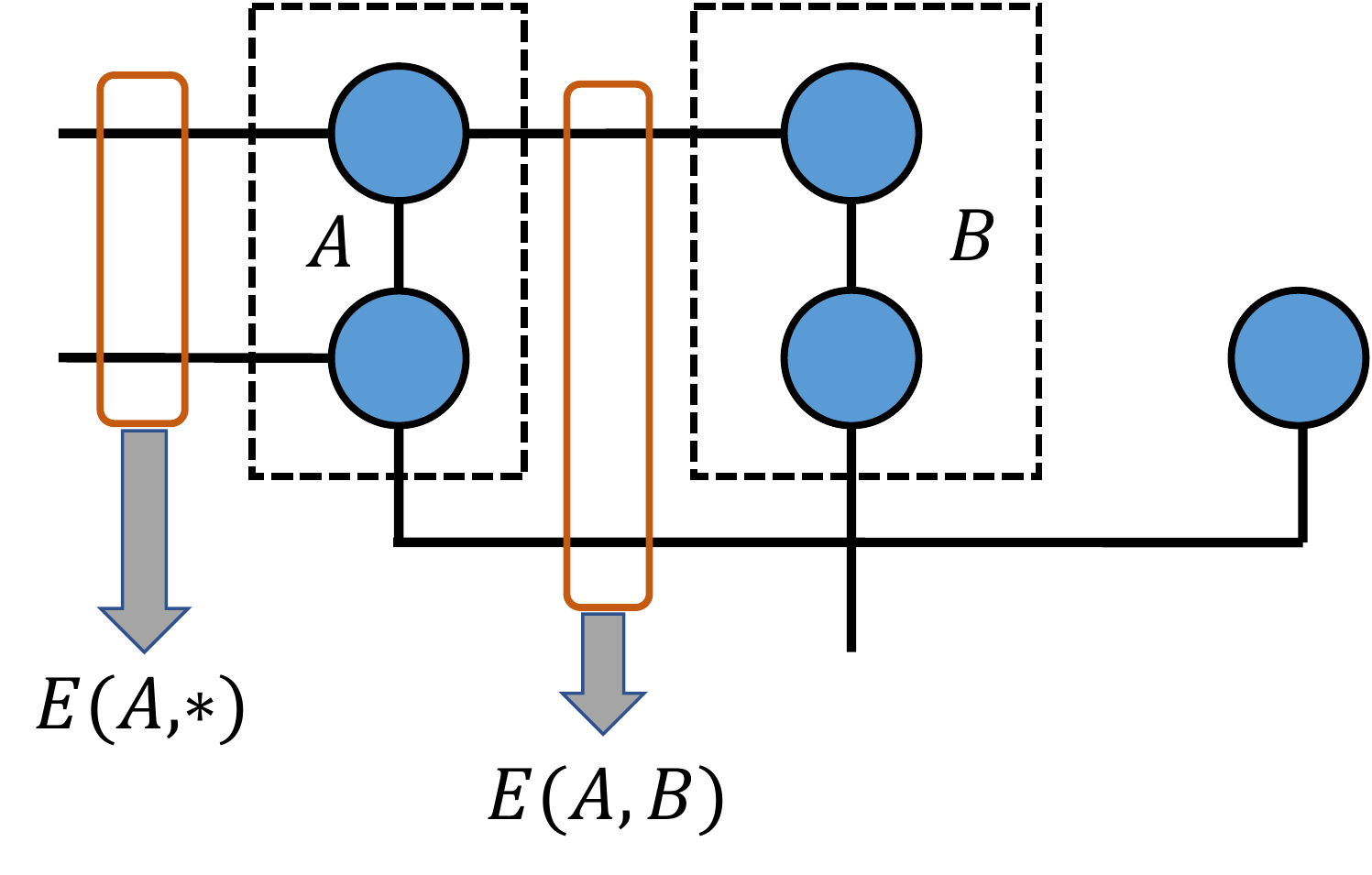}
  \caption{An example of $E(A,*)$ and $E(A,B)$, where both $A,B$ are subset of vertices.
  }
  \label{fig:E_AB_example}
\end{figure} 

Our analysis also use directed graphs to represent tensor network linearizations. We use $E(u,v)$ to denote the edge from $u$ to $v$, and similarly use $E(A,B)$ to denote the set of edges from $A$ to $B$.

When representing the contraction tree, we use $(v_1, v_2)$ to denote the contraction of $v_1,v_2$. This notation is also used to represent multiple contractions. For example, we use $(((v_1,v_4),(v_2,v_5)),v_3)$ to represent the contraction tree shown in \cref{fig:example_threetrees}.
The computational cost of a contraction tree is the summation of each contraction's cost. 
In the discussion throughout the paper, we assume that all tensors in the network are dense. Therefore,
the contraction of two general dense tensors $\tsr{A}$ and $\tsr{B}$, represented as vertices $v_a$ and $v_b$ in $G=(V,E,w)$, can be cast as a matrix multiplication, and the overall asymptotic cost is 
\[\bigTheta{ \exp\left(w(E(v_a)) + w(E(v_b)) - w(E(v_a,v_b))\right)}
\]
with classical matrix multiplication algorithms.
In general, contracting tensor networks with arbitrary structure is  \#P-hard~\cite{damm2002complexity,valiant1979complexity}.

\begin{figure}[!ht]
\centering

\subfloat[]{\includegraphics[width=0.2\textwidth, keepaspectratio]{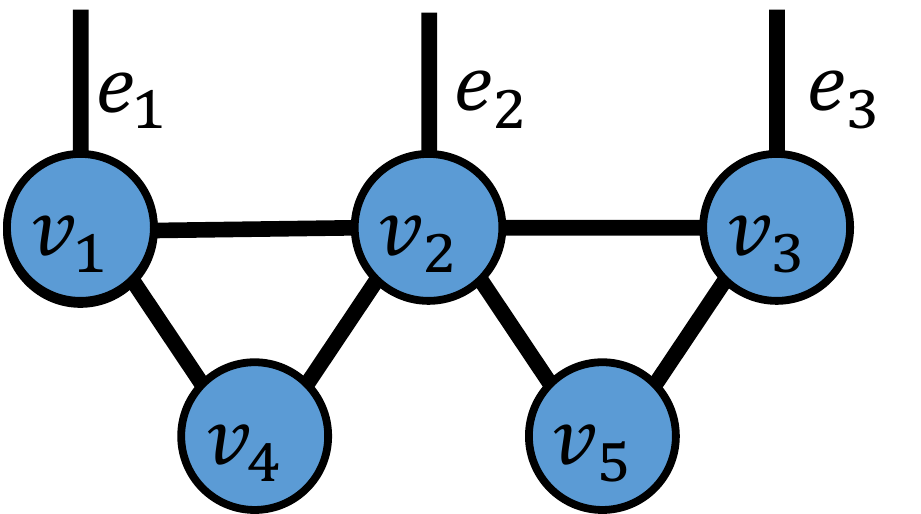}\hspace{6mm}}
\subfloat[]{\hspace{6mm}\includegraphics[width=0.25\textwidth, keepaspectratio]{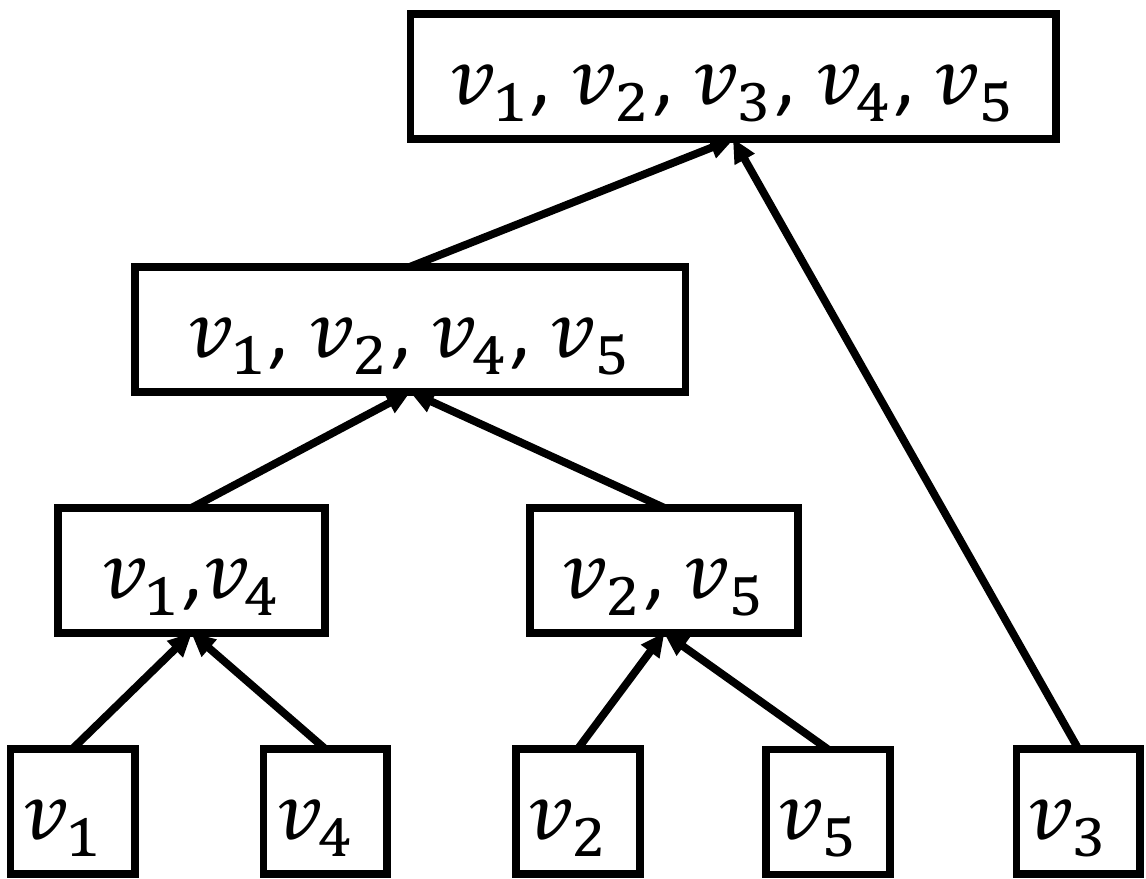}\hspace{6mm}}
\subfloat[]{\hspace{6mm}\includegraphics[width=0.18\textwidth, keepaspectratio]{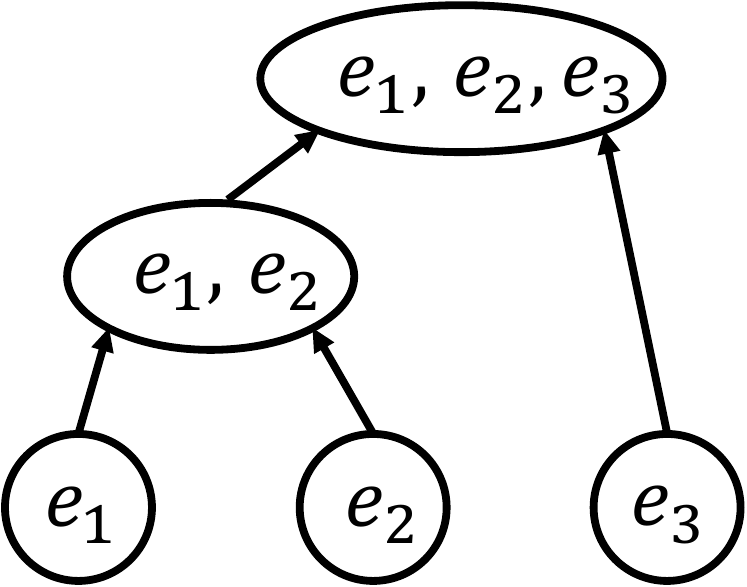}}

\caption{Example of a tensor network, its contraction tree and the corresponding dimension tree. 
}
\label{fig:example_threetrees}
\end{figure}

For a given data $\Gd$ and its given contraction tree $T_0$, its dimension tree is a directed binary tree showing the way edges
 in $\bar{E}$ are merged onto the same tensor. Each vertex in the dimension tree is a subset $E'\subseteq \bar{E}$, and for any two vertices $E_1',E_2'$ of the dimension tree with the same parent, there is a contraction in $T_0$ such that the two input tensors are incident to $E_1',E_2'$, respectively. One example is shown in \cref{fig:example_threetrees}.

\subsection{Definitions used in the analysis of tensor network embedding}
In this section, we introduce definitions that will be used in later analysis. For a (hyper)graph $G = (V, E, w)$ and two subsets of $V$ denoted as $A,B$, we define 
$\cut_G(A,B) = \sum_{e\in E(A,B)}w(e)$.
Similarly, we define $\cut_G(A,*) = \sum_{e\in E(A,*)}w(e)$, and define $\cut_G(A) = \sum_{e\in E(A)}w(e)$, where $E(A)$ is expressed in \eqref{eq:ES}.
When $G$ is a directed hypergraph, $\cut_G(A,B)$ denotes the sum of the weights of edges from $A$ to $B$.
When $G$ is an undirected graph, $\cut_G(A,B)$ denotes the sum of the weights of hyperedges connecting $A$ and $B$. 

For two tensors represented by two subsets $A,B\subset V$ and $A\cap B = \emptyset$, 
the logarithm of the contraction cost
between a tensor represented by $A$ and a tensor represented by $B$,
$(A, B)$, is
\[
\cost_G(A, B)  = \cut_G(A) + \cut_G(B) - \cut_G(A, B).
\]
Note that the function $\cost$ is only defined on undirected hypergraphs.

 Consider a given input data $\Gd= (\Vd,\Ed,w)$ and an embedding $\Ge= (\Ve,\Ee,w)$. Below we let $V = \Ve\cup \Vd$, $E = \Ee\cup \Ed$, and $G = (V,E,w)$ denote the hypergraph including both the embedding and the input data.
We use $\GL = (V, \Ee, w)$ to denote the graph including $V$ and all edges in the embedding, and use $\GR = (V, E\setminus \Ee, w)$.
Note that in this work we focus on the case where $\GL$ is a graph, and $\GR$ can be a general hypergraph.
We illustrate $G, \Gd, \Ge, \GL, \GR$ in \cref{fig:directed_linearization}.
For any $A,B\subset V$ and $A\cap B = \emptyset$, we have 
\begin{equation}\label{eq:cut1}
    \cut_G(A) = \cut_\GL(A) + \cut_\GR(A),
\end{equation}
and 
\begin{equation}\label{eq:cut2}
\cut_G(A,B) = \cut_\GL(A,B) + \cut_\GR(A,B).
\end{equation}
Based on \eqref{eq:cut1} and \eqref{eq:cut2}, we have 
\[
\cost_G(A, B)  = \cost_\GL(A, B)  + \cost_\GR(A, B).
\]

Our analysis of tensor network embedding is based on the linearization of the tensor network graph. Linearization casts an undirected graph into a \textit{directed acyclic graph} (DAG).
We define linearization formally below, then specify linearizations of the data and embedding graphs that our analysis considers.

\begin{definition}[Linearization DAG]\label{def:directed_linearization}
A linearization of the undirected graph $G=(V,E,w)$ is defined by the DAG $G' = (V,E',w)$ induced by a given choice of vertex ordering in $V$. For each contracted edge in $E$, $E'$ contains an same-weight edge  directing towards the higher indexed vertex.
For each uncontracted edge in $E$, $E'$ contains an edge with the same weight that is directed outward from the vertex it is adjacent to. 
\end{definition}

Based on \cref{def:directed_linearization}, we define the \textit{sketching linearization DAG}, $\GA = (V, \EA, w)$, as a DAG defined on top of the graph $L=(V,\Ee, w)$, which includes all vertices in both the embedding and the data and all embedding edges. For a given vertex ordering of embedding vertices, $\GA$ is the linearization of $\GL$ based on the ordering with all data vertices being ordered ahead of embedding vertices.

As discussed in \cref{subsec:summery_sketchsize}, for a given sketching linearization,
the sketching accuracy of each tensor $\tsr{A}_i$ at $v_i$ is dependent on the row size of its matricization $\mat{A}_i$, which is the
weighted size of the edge set adjacent to $v_i$ containing all uncontracted edges and contracted edges also adjacent to $v_j$ with $j > i$, which is called \textit{effective sketch dimension} of $v_i$ throughout the paper.
Based on the definition, when $v\in \Ve$, $\cut_{\GA}(v)$ equals the effective sketch dimension size of $v$. When $v\in \Vd$, $\cut_{\GA}(v)$ represents the size of the sketch dimension adjacent to $v$.
We look at embeddings $\Ge$ not only satisfying the ($\epsilon, \delta$)-accurate sufficient condition in \cref{thm:subgaussian_embedding}, but also only have one output sketch dimension ($|E_1|=1$) with the output sketch size $m=\bigTheta{\Ne\log(1/\delta)/\epsilon^2}$.
For each one of these embeddings, there must exist a linearization $\GA$ such that for all $v\in \Ve$, we have \begin{equation}\label{eq:cut_bound_sufficient}
    \cut_{\GA}(v)= \bigOmega{\log(m)}.
\end{equation}

\begin{figure}[!ht]
\centering

\subfloat[$G = (V, \Ee\cup \Ed, w)$]
{\includegraphics[width=0.23\textwidth, keepaspectratio]{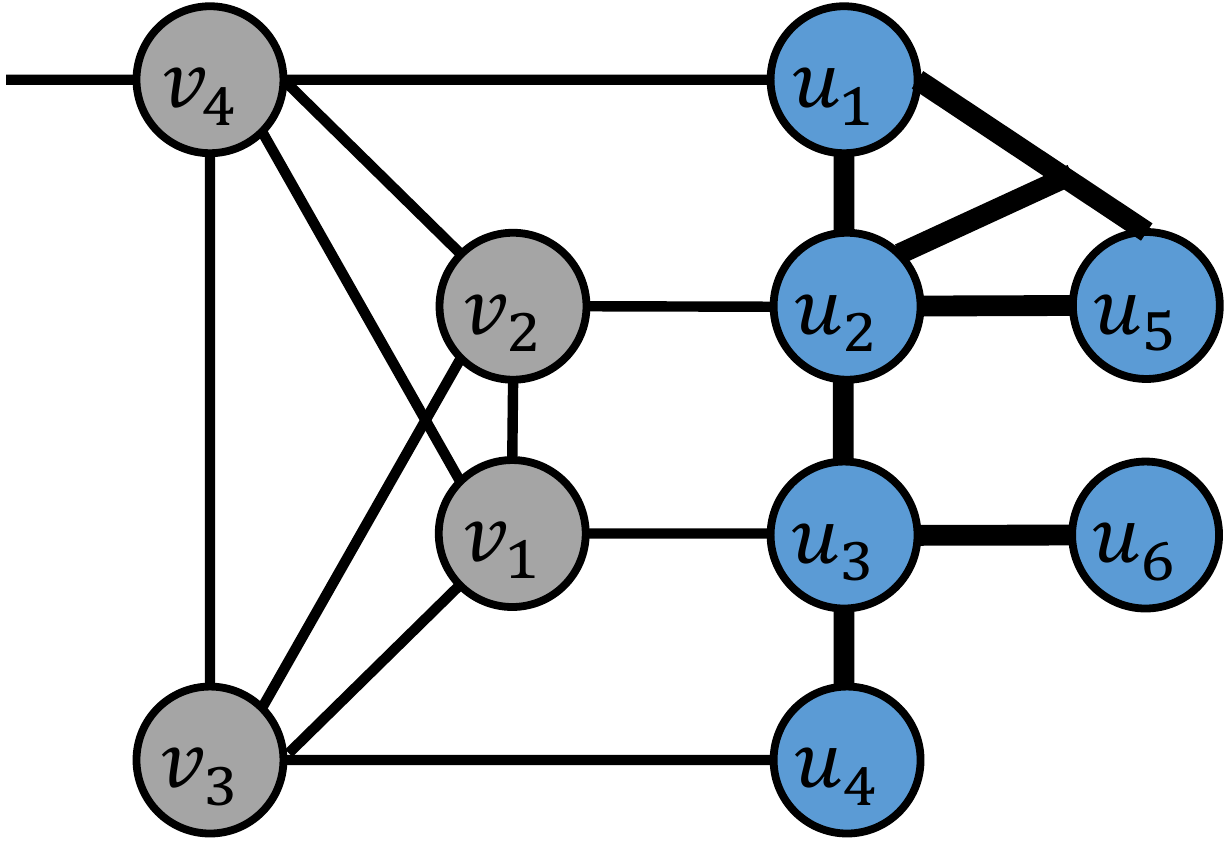}\label{subfig:G}\hspace{12mm}}
\subfloat[$\Ge = (\Ve, \Ee, w)$]
{\includegraphics[width=0.15\textwidth, keepaspectratio]{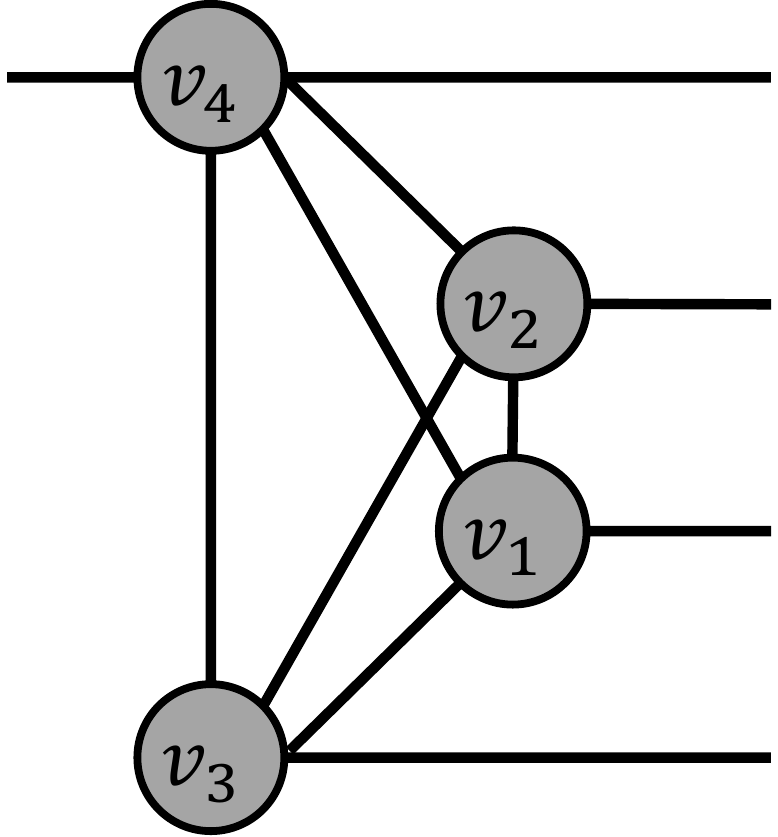}\label{subfig:\Ge}\hspace{12mm}}
\subfloat[$\Gd = (\Vd, \Ed, w)$]
{\includegraphics[width=0.13\textwidth, keepaspectratio]{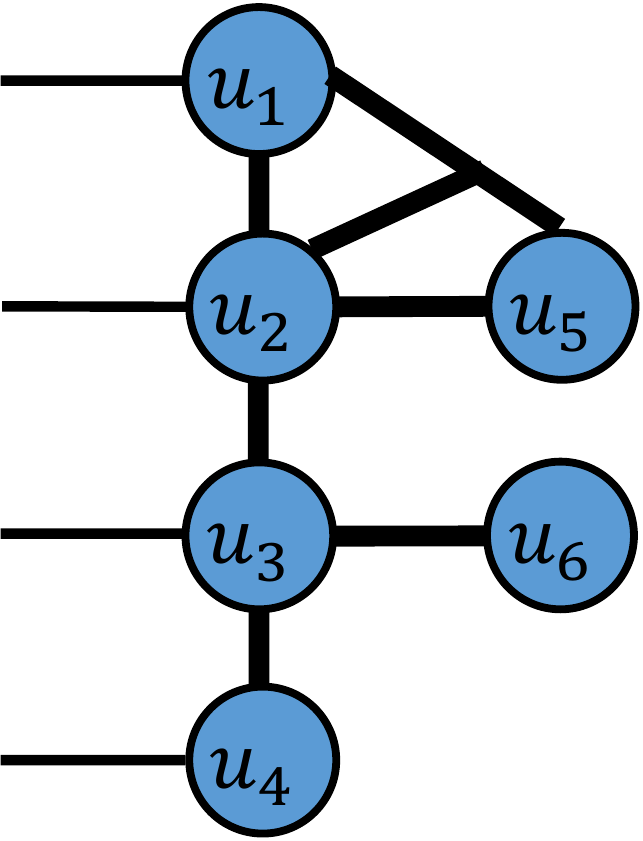}\label{subfig:\Gd}\hspace{12mm}}

\subfloat[$L = (V, \Ee, w)$]
{\includegraphics[width=0.23\textwidth, keepaspectratio]{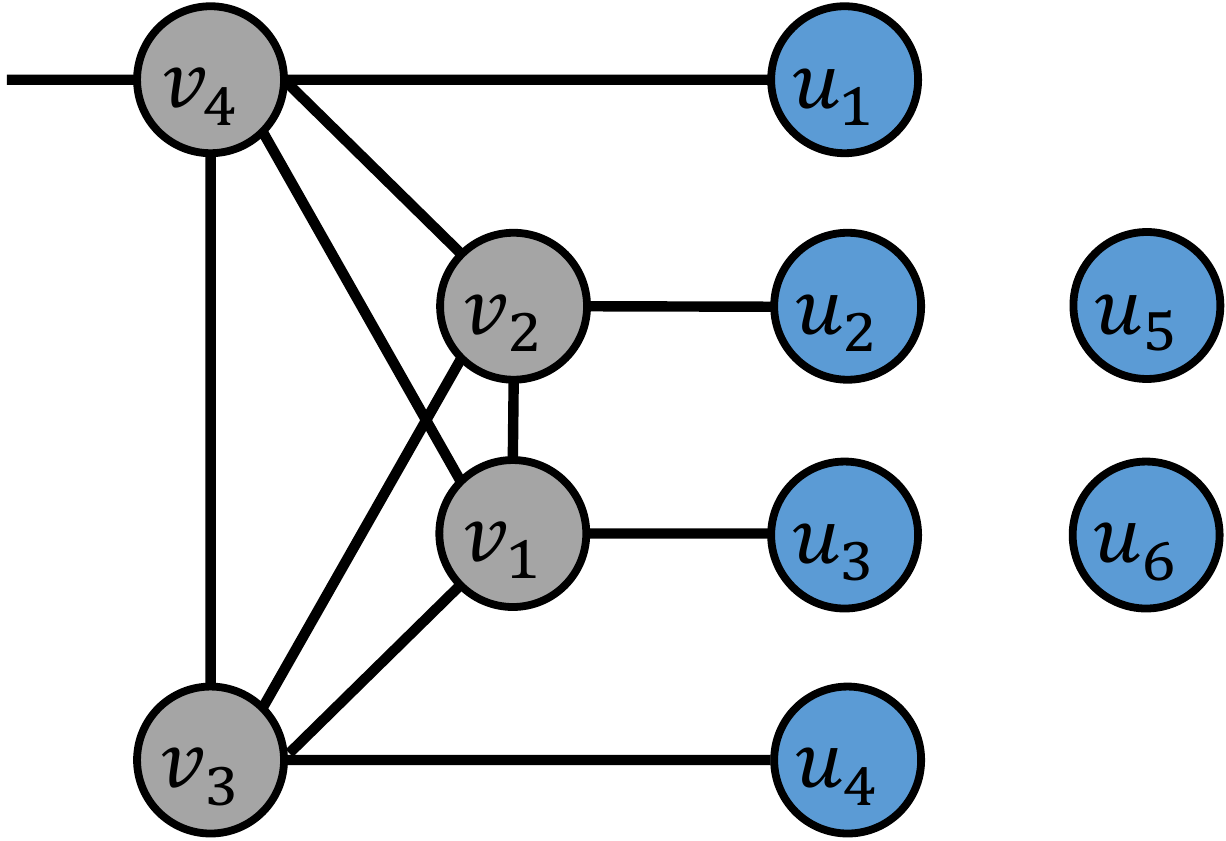}\label{subfig:L}\hspace{10mm}}
\subfloat[$R = (V, E \setminus \Ee, w)$]
{\includegraphics[width=0.2\textwidth, keepaspectratio]{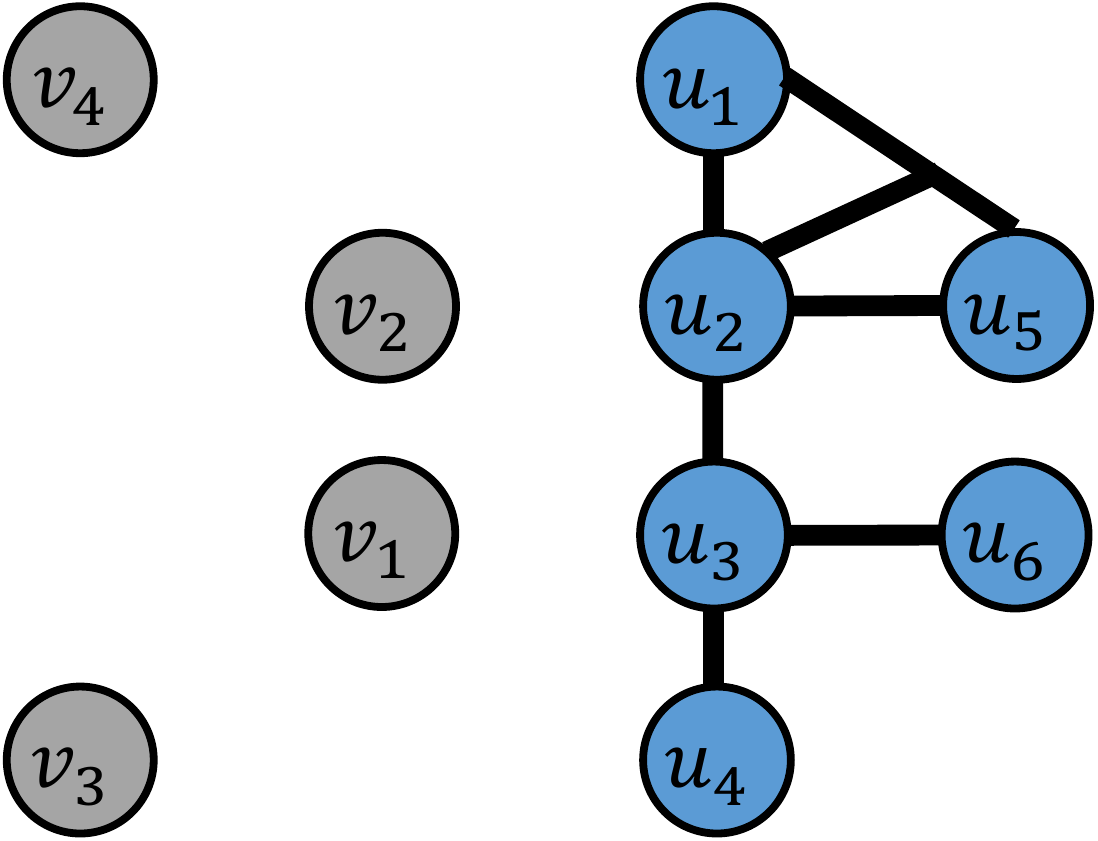}\label{subfig:R}\hspace{10mm}}
\subfloat[$\GA = (V, \EA, w)$]
{\includegraphics[width=0.23\textwidth, keepaspectratio]{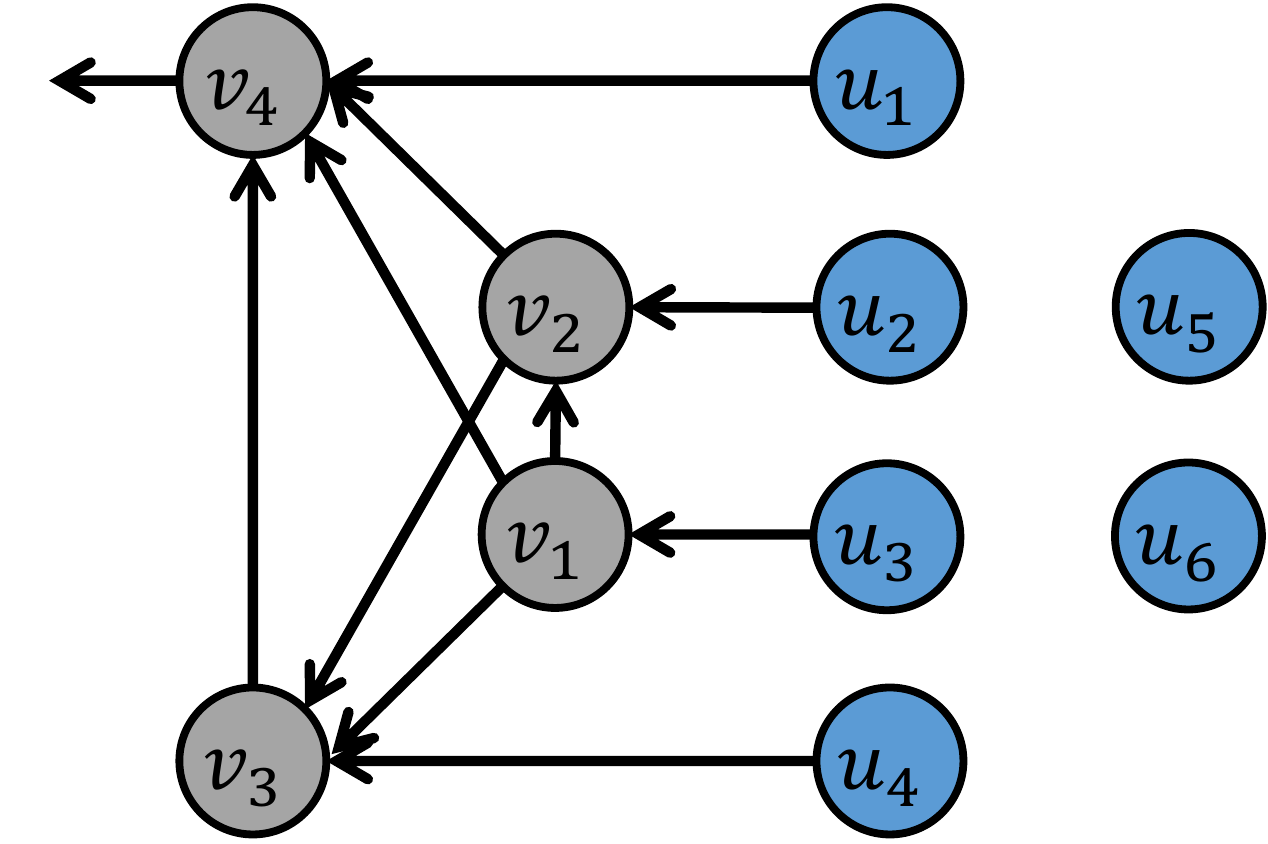}\label{subfig:D}\hspace{10mm}}

\caption{Illustration of graphs and hypergraphs used throughout the paper.}
    \label{fig:directed_linearization}
\end{figure}

\subsection{Properties of tensor network embedding}

We now derive properties that are used in the sketching computational cost analysis. In \cref{lem:graph_relation}, we show relations between cuts in the graph $\GL$ and cuts in the graph $\GA$. In \cref{lem:cost_linearization}, we show relations between costs in the graph $\GL$ and cuts in the graph $\GA$. \cref{lem:cost_linearization} along with cut lower bounds \eqref{eq:cut_bound_sufficient} is used to derive lower bounds for $\cost_{\GL}$ and $\cost_{G}$ in \cref{sec:lowerbound}.

\begin{lemma}\label{lem:graph_relation}
Consider an embedding $\Ge=(\Ve,\Ee,w)$ and a data tensor network $\Gd=(\Vd,\Ed,w)$, and
a given sketching linearization $\GA = (V,\EA,w)$, where $V = \Ve\cup \Vd$. For any $A,B\subset V$ and $A\cap B = \emptyset$, the following relations hold,
\begin{equation}\label{eq:linearization_p1}
  \cut_{\GL}(A) = \cut_{\GA}(A) +
  \cut_{\GA}(V\setminus A,A),
\end{equation}
\begin{equation}\label{eq:linearization_p2}
\cut_{\GL}(A,B) = \cut_{\GA}(A,B) + \cut_{\GA}(B,A),
\end{equation}
\begin{align}\label{eq:linearization_p3}
\cut_{\GA}(A\cup B) &= \cut_{\GA}(A) + \cut_{\GA}(B) - \cut_{\GA}(A,B) - \cut_{\GA}(B,A) \nonumber\\
& = \cut_{\GA}(A) + \cut_{\GA}(B) - \cut_{\GL}(A,B) .
\end{align}
\end{lemma}
\begin{proof}
\eqref{eq:linearization_p1} and \eqref{eq:linearization_p2} hold directly based on the definition of the linearization DAG. For \eqref{eq:linearization_p3}, based on \eqref{eq:ES}, we have 
\begin{align}\label{eq:graph_relation_1}
    \cut_{\GA}(A\cup B) &= \cut_{\GA}\left(A\cup B, V\setminus(A\cup B)\right) + \cut_{\GA}(A\cup B, *) \nonumber \\
    & = \cut_{\GA}\left(A, V\setminus(A\cup B)\right) + \cut_{\GA}\left( B, V\setminus(A\cup B)\right) + \cut_{\GA}(A\cup B, *) \nonumber\\
    & = \cut_{\GA}\left(A, V\setminus A\right) - \cut_{\GA}\left(A,B\right)+\cut_{\GA}\left( B, V\setminus B\right) -\cut_{\GA}\left( B, A\right) \nonumber \\ &+ \cut_{\GA}(A, *)+ \cut_{\GA}(B, *) \nonumber\\
    & = \cut_{\GA}(A) + \cut_{\GA}(B) - \cut_{\GA}(A,B) - \cut_{\GA}(B,A).
\end{align}
Note that the second and third equalities in \eqref{eq:graph_relation_1} hold since $A$ and $B$ are disjoint sets.
This finishes the proof.
\end{proof}

\begin{lemma}\label{lem:cost_linearization}
Consider any data $\Gd=(\Vd,\Ed,w)$ and embedding $\Ge=(\Ve,\Ee,w)$, and a sketching linearization $\GA = (V,\EA, w)$, where $V= \Vd\cup \Ve$. For any two subsets $A,B\in V$ such that $A\cap B = \emptyset$, the contraction of two tensors that are the contraction outputs of  $A$ and $B$ has a logarithm cost of 
\[
\cost_\GL(A, B) = \cut_{\GA}(A) + \cut_{\GA}(B) + \cut_{\GA}(V\setminus(A\cup B), A\cup B).
\]
\end{lemma}
\begin{proof}
Based on \cref{lem:graph_relation}, we have
\begin{equation}\label{eq:cost_ey1}
    \cut_{\GL}(A) \overset{\eqref{eq:linearization_p1}}{=} \cut_{\GA}(V\setminus A,A) +  \cut_{\GA}(A) ,
\end{equation}
\begin{equation}\label{eq:cost_ey2}
    \cut_{\GL}(B) \overset{\eqref{eq:linearization_p1}}{=} \cut_{\GA}(V\setminus B,B) +  \cut_{\GA}(B) .
\end{equation}
Based on \eqref{eq:linearization_p2}, we have 
\begin{align}\label{eq:cost_ey1ey2}
    &\cut_{\GA}(V\setminus A,A) + \cut_{\GA}(V\setminus B,B)  - \cut_{\GL}(A,B) \nonumber\\
    &= \cut_{\GA}(V\setminus A,A) + \cut_{\GA}(V\setminus B,B)  - \cut_{\GA}(A,B) - \cut_{\GA}(B,A)\nonumber\\
    &= \cut_{\GA}(V\setminus (A\cup B),A) + \cut_{\GA}(V\setminus (A\cup B),B)\nonumber\\
    &= \cut_{\GA}(V\setminus(A\cup B), A\cup B).
\end{align}
Based on \eqref{eq:cost_ey1},\eqref{eq:cost_ey2}, \eqref{eq:cost_ey1ey2}, we have 
\begin{align*}
    \cost_{\GL}(A, B) 
    &= \cut_{\GL}(A) + \cut_{\GL}(B) - \cut_{\GL}(A,B)
    \\&= \cut_{\GA}(A) + \cut_{\GA}(B) + \cut_{\GA}(V\setminus(A\cup B), A\cup B).
\end{align*}
This finishes the proof.
\end{proof}

\begin{lemma}\label{lem:elu}
Consider 
any data $\Gd=(\Vd,\Ed,w)$ and an embedding $\Ge=(\Ve,\Ee,w)$, and a sketching linearization $\GA=(V,\EA,w)$ such that the embedding is $(\epsilon,\delta)$-accurate. Then for any  $U\subseteq V$ such that there exists $v\in U$ and $\cut_{\GA}(v)\geq \log(m)$,
we have $\cut_{\GA}(U)\geq \log(m)$. 
\end{lemma}
\begin{proof}
When $U$ is a subset of the data vertices, $U\subseteq \Vd$, this holds directly since 
\[
\cut_{\GA}(U) = \sum_{u\in U}\cut_{\GA}(u) \geq 
\cut_{\GA}(v)\geq \log(m). 
\]
Next we consider the case where $U\cap \Ve \neq \emptyset$. Let $A = U\cap \Ve$ and $B = U\cap \Vd$. 
Based on the definition of DAG, there is no directed cycle in the subgraph $\GA[A]$.
Therefore, there exists one vertex $s\in A$, such that $\cut_{\GA}(s, A\setminus \{s\}) = 0$.
 Based on \cref{lem:graph_relation}, we have 
\begin{align*}
    \cut_{\GA}(A) &\overset{\eqref{eq:linearization_p3}}{=}
\cut_{\GA}(s) + \cut_{\GA}(A\setminus \{s\})
- \cut_{\GA}(s,A\setminus \{s\}) - \cut_{\GA}(A\setminus \{s\},s) \\
 &\geq 
 \cut_{\GA}(s) - \cut_{\GA}(s,A\setminus \{s\}) \\
&= \cut_{\GA}(s) \overset{\eqref{eq:cut_bound_sufficient}}{\geq} \log(m),
\end{align*}
In addition, we have $\cut_{\GA}(A,B) = 0$ since $A\subseteq \Ve$ and $B\subseteq \Vd$. Thus we have 
\begin{align*}
\cut_{\GA}(U) = 
\cut_{\GA}(A\cup B) &\overset{\eqref{eq:linearization_p3}}{=} \cut_{\GA}(A) + \cut_{\GA}(B) - \cut_{\GA}(A,B) - \cut_{\GA}(B,A) \\
&= \cut_{\GA}(A) + \cut_{\GA}(B) - \cut_{\GA}(B,A) \\
&\geq \cut_{\GA}(A)  \geq \log(m).
\end{align*}
This finishes the proof.
\end{proof}

%% file: algorithm.tex
\section{Computationally-efficient sketching algorithm}\label{sec:alg_detail}

In this section, we introduce the detail of the computationally-efficient sketching algorithm in \cref{alg:contract_sketch}. Consider a given data tensor network $\Gd = (\Vd, \Ed, w)$ and a given data contraction tree, $T_0$. Also let $\Nd = |\Vd|$, and let $\bar{E}\subseteq \Ed$ denote the set of edges to be sketched, and $N = |\bar{E}|$. 
Below we let $\bar{E} = \{e_1,e_2,\ldots, e_N\}$, and let each $e_i$ has weight $\log(s_i)>\log(m)$.
Based on the definition we have $N \leq \Nd$.
Let one  contraction path representing $T_0$ be expressed as a sequence of $\Nd-1$ contractions, 
\begin{equation}
    \left\{(U_1,V_1), \ldots, (U_{\Nd-1}, V_{\Nd-1})\right\}.
\end{equation}
Above we use $(U_i,V_i)$ to represent the contraction of two intermediate tensors represented by two subset of vertices $U_i,V_i\subset \Vd$. Below we let 
\begin{align}\label{eq:label}
    &a_i = \exp\left(\cut_{\GR}(U_i)-\cut_{\GR}(U_i,V_i)\right), \nonumber \\ 
    &c_i = \exp\left(\cut_{\GR}(V_i)-\cut_{\GR}(U_i,V_i)\right), \nonumber \\
    &d_i = \exp\left(\cut_{\GR}(U_i\cup V_i)\right)/(a_ic_i), \nonumber \\ 
    &b_i = \exp\left(\cut_{\GR}(U_i,V_i)\right)/d_i.
\end{align}
Note that $d_i$ represents the size of uncontracted dimensions adjacent to both $U_i$ and $V_i$, and $b_i$ represents the size of contracted dimensions between $U_i$ and $V_i$. We also have $\cost_R(U_i, V_i) = \log(a_ib_ic_id_i)$, and  $\cut_R(U_i\cup V_i)= \log(a_ic_id_i)$. $a_i,b_i,c_i,d_i$ are visualized in \cref{fig:sketch_cases}.

\subsection{Sketching with the embedding containing a binary tree of small tensor networks}\label{subsec:binary}

We now present the details of applying the embedding containing a binary tree of small tensor networks. 
In \cref{subsec:summary_efficient_ebmedding}, we define $\st{S}$ as the set containing contractions $(U_i,V_i)$ such that both $U_i$ and $V_i$ are adjacent to edges in $\bar{E}$. 
For each contraction $i\in \st{S}$,
one small embedding tensor network (denoted as $Z_i$) is applied to the contraction. 
Let $\hat{U}_i,\hat{V}_i$ denote 
   the sketched $U_i$ and $V_i$ formed in previous contractions in the sketching contraction tree $T_B$, such that $\hat{U}_i\cap \Vd = U_i$ and $\hat{V}_i\cap \Vd = V_i$.
The structure of $Z_i$ is determined so that the asymptotic cost to sketch $(\hat{U}_i,\hat{V}_i)$ is minimized, under the constraint that $Z_i$ is $(\epsilon/\sqrt{N},\delta)$-accurate and only has one output sketch dimension.

The structure of $Z_i$ is illustrated in \cref{fig:sketch_cases}. For the case $a_i\leq c_i$, the structure is shown in  \cref{subfig:sketch_case1}, and sketching is performed via the contraction sequence of 
contracting $\hat{U}_i$ and $v_1$ first, then with $\hat{V}_i$, and then with $v_2$ (also denoted as a contraction sequence of $(((\hat{U}_i, v_1), \hat{V}_i), v_2)$). 
For the case $a_i> c_i$, the structure of $Z_i$ is shown in \cref{subfig:sketch_case2}, and the sketching is performed via the contraction sequence of $(((\hat{V}_i, v_2), \hat{U}_i), v_1)$. With this algorithm, sketching $(\hat{U}_i,\hat{V}_i)$ yields a computational cost proportional to
\begin{equation}\label{eq:yi}
   y_i = a_ib_ic_id_im^2 + m^2d_i\sqrt{a_ib_ic_im}\cdot\min(\sqrt{a_i},\sqrt{c_i}).
\end{equation}
We show in \cref{lem:lowerbound_gen} that the asymptotic cost lower bound to sketch $(U_i, V_i)$ is also $\Omega(y_i)$.

\begin{figure}[!ht]
\centering

\subfloat[]{\includegraphics[width=0.28\textwidth, keepaspectratio]{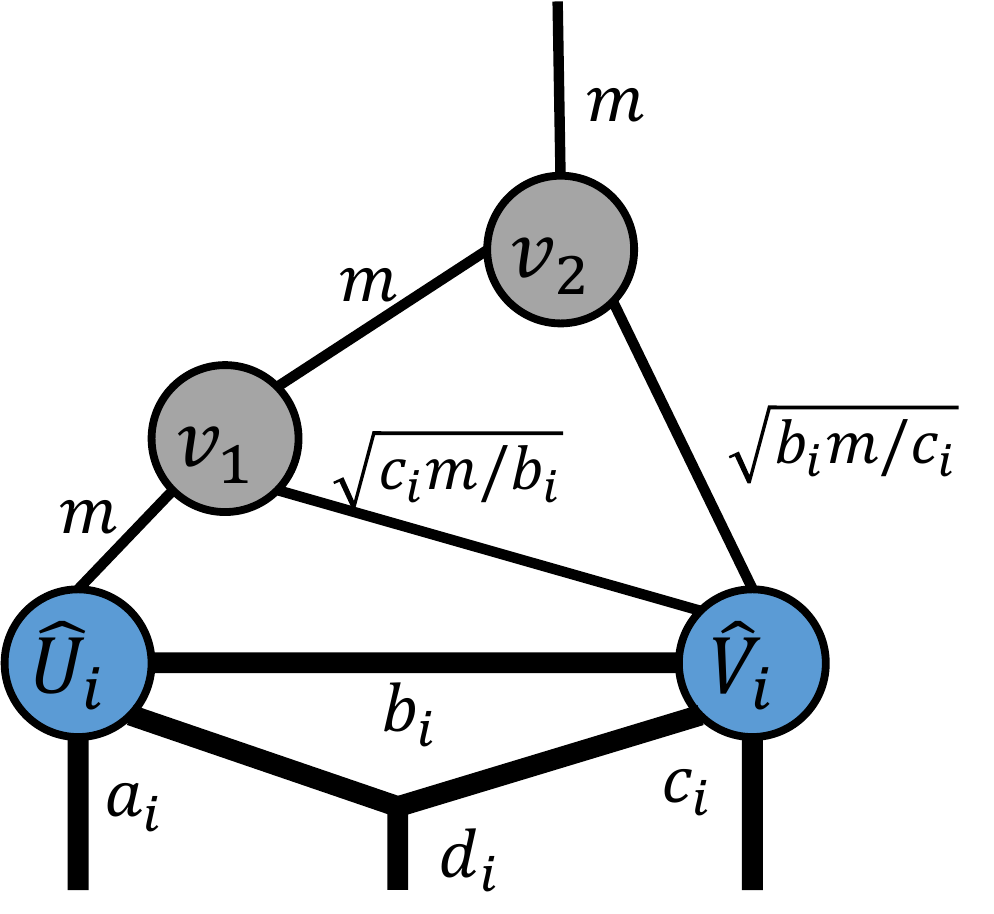}\label{subfig:sketch_case1}}
\subfloat[]{\includegraphics[width=0.28\textwidth, keepaspectratio]{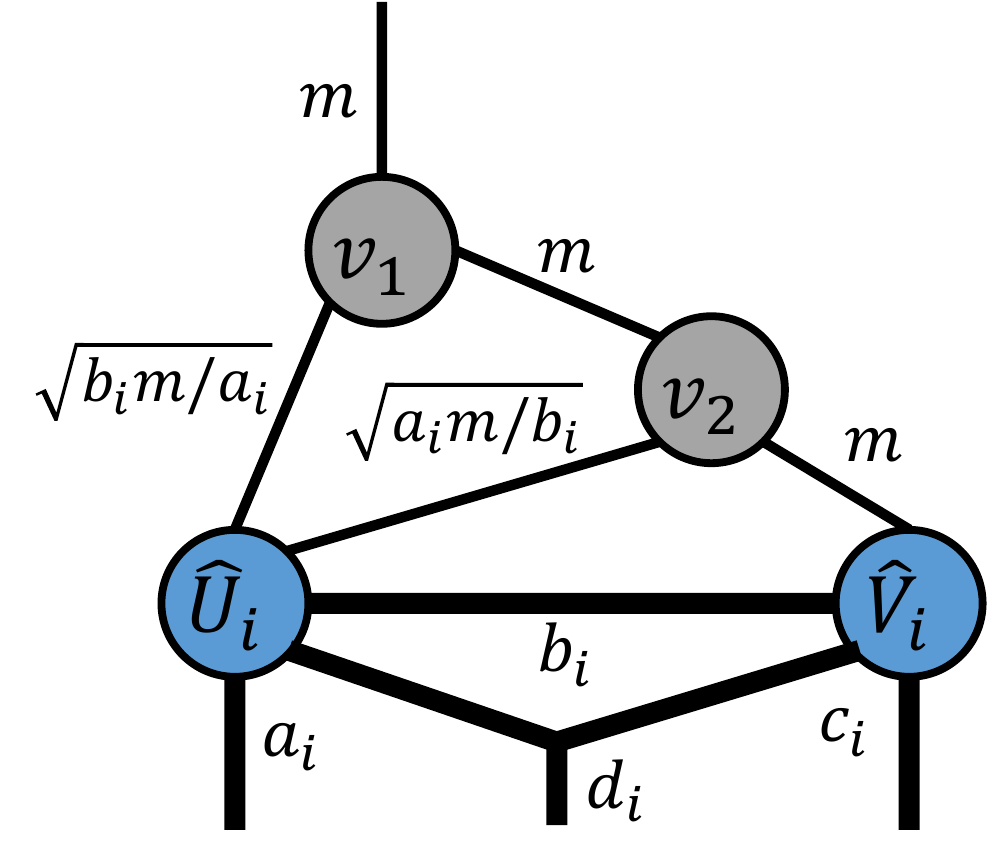}\label{subfig:sketch_case2}}

\caption{Illustration of the small network in the binary tree structured embedding. For each edge $e$, we show the dimension size of that edge (exponential in $w(e)$).
}
\label{fig:sketch_cases}
\end{figure}

\subsection{Computational cost analysis}\label{subsec:appendix_costanalysis}

We provide the computational cost analysis of \cref{alg:contract_sketch} in this section.

\begin{theorem}\label{thm:complexity}
  \cref{alg:contract_sketch} has an asymptotic computational cost of 
\begin{equation}\label{eq:cost}
  \Theta\left(
\sum_{j=1}^{N} \stt{j} + 
\sum_{i\in \st{S}}y_i + 
\sum_{i\in \st{I}}z_i
\right),  
\end{equation}
where $\stt{j}$ is the optimal asymptotic cost to sketch the sub tensor network $X(e_j)$ (defined in \cref{tab:notations}) with a matrix in the Kronecker product embedding, 
$y_i$ is expressed in \eqref{eq:yi}, 
and
\begin{equation}\label{eq:zi}
   z_i = a_ib_ic_id_i\cdot \min\left(\exp(\cut_{\GA}(U_i\cup V_i)), m\right),
\end{equation}
where $a_i,b_i,c_i,d_i$ are expressed in \eqref{eq:label}.
\end{theorem}
\begin{proof}
The terms $\sum_{j\in N} \stt{j} + 
\sum_{i\in \st{S}}y_i$ can be easily verified based on the analysis in \cref{subsec:summary_efficient_ebmedding} and \cref{subsec:binary}. 

Consider the contractions in $\st{I}$, which  include $(U_i,V_i)$ such that both $U_i$ and $V_i$ are not adjacent to $\bar{E}$, and contractions where $U_i$ or $V_i$ is adjacent to at least two edges in $\bar{E}$. 
The first type of contractions in $\st{I}$ would have a cost of $\Theta(a_ib_ic_id_i)$, and not be affected by previous sketching steps.
For the second type, application of the Kronecker product and binary tree embeddings to $U_i$ and $V_i$ would reduce all adjacent edges in $\bar{E}$ to a single dimension of size $m$. Consequently, the contraction cost 
would be $\Theta(a_ib_ic_id_i\cdot m)$. Summarizing both cases prove the cost in \eqref{eq:zi}.
The cost in \eqref{eq:cost} follows from combining the terms $\sum_{j\in N} \stt{j} + 
\sum_{i\in \st{S}}y_i$ and $\sum_{i\in \st{I}}z_i$.
\end{proof}

For the special case where each vertex in the data tensor network is adjacent to an edge to be sketched, we have $\stD{j}=  \emptyset$ for all $j\in [N]$ and $\st{I} = \emptyset$, thus all the contractions are in the set $\st{S}$. 
Therefore, sketching each $e_j$ has an asymptotic cost of
$\Theta(\stt{j}) = \Theta(\exp(\cut_G(v_j))\cdot m)$, where $v_j$ is the vertex in the data graph adjacent to $e_j$, and
\cref{thm:complexity} implies that the sketching cost would be 
\begin{align}\label{eq:cost_uniform}
  \Theta\left(
\sum_{j=1}^{N} \stt{j} + 
\sum_{i\in \st{S}}y_i + 
\sum_{i\in \st{I}}z_i
\right) 
 =   \Theta\left(
\sum_{j=1}^{N} \exp(\cut_G(v_j))\cdot m + 
\sum_{j=1}^{N-1}y_i
\right) .
\end{align}
As we will show in \cref{thm:lowerbound_uniform}, this cost matches the asymptotic cost lower bound, when the embedding satisfies the $(\epsilon,\delta)$-accurate sufficient condition and only has one output sketch dimension.

When
the data has a Kronecker product structure, we have $\cut_G(v_j) = w(e_j) = \log(s_j)$, 
and $a_i,b_i,c_i,d_i=1$ for all $i\in \{1,\ldots, N-1\}$ for all contraction trees. Therefore,
\[
y_i =  a_ib_ic_id_im^2 + m^2d_i\sqrt{a_ib_ic_im}\cdot\min(\sqrt{a_i},\sqrt{c_i}) = m^2 + m^{2.5},
\]
and the sketching cost is
\begin{align}\label{eq:sketchkronecker_general}
\Theta\left(
\sum_{j= 1}^{N} s_jm + 
Nm^{2.5}
\right).
\end{align}
As we will show in \cref{sec:appendix_tree}, sketching with tree tensor network embeddings yield an asymptotic cost of $\Theta\left(
\sum_{j= 1}^{N} s_jm + 
Nm^{3}
\right).$ Therefore, \cref{alg:contract_sketch} is more efficient to sketch Kronecker product input data.

%% file: lower_bound.tex
\section{Lower bound analysis}\label{sec:lowerbound}

In this section, we discuss the \textit{asymptotic} computational lower bound for sketching with embeddings satisfying the $(\epsilon,\delta)$-accurate sufficient condition and only have one output sketch edge. In \cref{subsec:sketch_uniform}, we discuss the case where the data has uniform sketch dimensions. In this case, each vertex in the data tensor network is adjacent to an edge to be sketched. In \cref{subsec:sketch_general}, we discuss the sketching computational lower bound for a more general case, when the data tensor network can have arbitrary graph structure, and vertices not adjacent to sketch edges are allowed. For both cases, we assume that 
the size of each dimension to be sketched 
is greater than the sketch size.

\subsection{Sketching data with uniform sketch dimensions}\label{subsec:sketch_uniform}

We now discuss the sketching asymptotic cost lower bound when the data $\Gd=(\Vd,\Ed,w)$  has uniform sketch dimensions, where each $v\in \Vd$ is adjacent to an edge to be sketched with size lower bounded by the target sketch size, $m$. 
We have $N =|\bar{E}| = |\Vd|$, and we let the size of each $e_i\in \bar{E}$ be denoted $s_i>m$.
We let $V=\Ve\cup \Vd$ denote the set of all vertices in both the data and the embedding.
Below, we show the main theorem using lemmas and notations introduced in \cref{sec:linearization}.

\begin{theorem}\label{thm:lowerbound_uniform}
  For any embedding $\Ge$ satisfying the $(\epsilon,\delta)$-accurate sufficient condition  and only has one output sketch dimension, and any contraction tree $T_B$ of $(\Gd,\Ge)$ constrained on the data contraction tree $T_0$ expressed in \eqref{eq:exp_T0}, the sketching asymptotic cost is lower bounded by 
\begin{equation}\label{eq:lowerbound_uniform}
  \Omega\left(
\sum_{j=1}^{N} \exp(\cut_G(v_j))\cdot m + 
\sum_{j=1}^{N-1}y_i
\right),  
\end{equation}
 where $m=\Omega(N\log(1/\delta)/\epsilon^2)$ represents the embedding sketch size, $v_j$ is the vertex in $\Vd$ adjacent to $e_j$, and $\exp(\cut_G(v_j))$ denotes the size of the tensor at $v_j$, and
$y_i$ is expressed in \eqref{eq:yi}. \end{theorem}
We present the proof of \cref{thm:lowerbound_uniform} at the end of \cref{subsec:sketch_uniform}. Note that the first term in \eqref{eq:lowerbound_uniform}, $\sum_{v\in \Vd} \exp\left(\cut_G(v)\right)\cdot m$, is a term independent of the data contraction tree, while the second term is dependent of the data contraction tree. 

\begin{proof}[Proof of \cref{thm:optimal}]
The asymptotic cost of of \cref{alg:contract_sketch} in \eqref{eq:cost_uniform} matches the lower bound shown in \cref{thm:lowerbound_uniform}, thus proving the statement.
\end{proof}

\cref{thm:lowerbound_uniform} also yields an asymptotic lower bound for sketching data with a Kronecker product structure. We state the results below.

\begin{corollary}\label{cor:kronecker}
Consider an input data $\Gd$ representing a vector with a Kronecker product structure and each $v_j$ for $j\in [N]$ is adjacent to an edge to be sketched with size $s_j$.
For any embedding $\Ge$ satisfying the $(\epsilon,\delta)$-accurate sufficient condition with only one output sketch dimension
and any contraction tree $T_B$ of $(\Gd,\Ge)$, the asymptotic cost must  be lower bounded by 
\[\bigOmega{\sum_{j= 1}^{N} s_jm + Nm^{2.5}},
\]
where $m=\Omega(N\log(1/\delta)/\epsilon^2)$.
\end{corollary}

Below, we present some lemmas needed to prove \cref{thm:lowerbound_uniform}. 

\begin{lemma}\label{lem:w_cut}
Consider 
an $(\epsilon,\delta)$-accurate embedding $\Ge = (\Ve, \Ee, w)$ with a sketching linearization $\GA=(V,\EA,w)$. Then for any subset of the embedding and data graph vertex set, $W\subseteq V$, we have  $\cut_{\GA}(W)\geq \log(m)$.
\end{lemma}
\begin{proof}
Since each vertex in the data graph is adjacent to an edge to be sketched, and the edge dimension size is greater than $m$, we have $\cut_{\GA}(w)\geq \log(m)$ for all $w\in \Vd$. Since the embedding satisfies the $(\epsilon,\delta)$-accurate sufficient condition, we have $\cut_{\GA}(w)\geq \log(m)$ for all $w\in \Ve$. Therefore, $\cut_{\GA}(w)\geq \log(m)$ for all $w\in V$. Based on \cref{lem:elu}, $\cut_{\GA}(W)\geq \log(m)$ for all $W\subseteq V$.
\end{proof}

\begin{lemma}\label{lem:sketch_one_intermediate_lower}
Consider 
an $(\epsilon,\delta)$-accurate embedding $\Ge$ with a sketching linearization $\GA=(V,\EA,w)$. Consider any contraction tree $T_B$ for $(\Gd,\Ge)$.
If there exists a contraction output of $U\subset V$ formed in $T_B$ and $\cut_{\GA}(U)> \log(m)$, then the asymptotic cost for the contraction tree $T_B$ must be lower bounded by $\bigOmega{\exp\left(\cut_R(U) + \cut_{\GA}(U)\right)\cdot m}$.
\end{lemma}
\begin{proof}
Since $\cut_{\GA}(U) > \log(m)$, there must exist a contraction $(U,W)\in T_B$ with $W$ containing some vertices in $V\setminus U$. 
Based on \cref{lem:w_cut}, $\cut_{\GA}(W)\geq \log(m)$.
Based on \cref{lem:cost_linearization}, we have
\begin{align*}
    \cost_G(U, W)
    &= \cost_R(U, W) + \cost_L(U, W) \\
    &= \cost_R(U, W) + \cut_{\GA}(U) + \cut_{\GA}(W) + \cut_{\GA}(V\setminus(U\cup W), U\cup W).
\end{align*}
Further, since $\cost_R(U, W)\geq \cut_R(U)$,
\begin{align*}
    \cost_G(U, W) & \geq \cut_R(U) + \cut_{\GA}(U) + \cut_{\GA}(W) \\
    & \geq \cut_R(U) + \cut_{\GA}(U) + \log(m).
\end{align*}
This proves the lemma since the contraction cost is $\bigTheta{\exp(\cost_G(U, W))}$.
\end{proof}

In \cref{lem:lowerbound_gen}, we show that when the data contraction tree $T_0$  contains the contraction $(U_i,V_i)$, then any contraction tree $T_B$ of $(\Gd,\Ge)$ that is constrained on $T_0$ will yield a contraction cost of $\Omega(y_i)$. To show that, we first discuss the case where $T_B$ also contains the contraction $(U_i,V_i)$ in \cref{lem:lowerbound_gen_case1}. 
The more general case where the contraction  $(U_i,V_i)$ need not be in $T_B$ is discussed in \cref{lem:lowerbound_gen}.

\begin{lemma} \label{lem:lowerbound_gen_case1}
Consider 
a specific contraction tree $T_0$ for $\Gd$, where the contraction $(U_i,V_i)$ is in $T_0$. For any embedding $\Ge$ satisfying the $(\epsilon,\delta)$-accurate sufficient condition
with only one output sketch dimension
and any contraction tree $T_B$ of $(\Gd,\Ge)$ constrained on $T_0$, if $(U_i, V_i)$ is also in $T_B$, the sketching asymptotic cost must be lower bounded by 
$\bigOmega{a_ib_ic_id_im^2 + a_ic_id_im^3}$,
where $a_i,b_i,c_i,d_i$ are defined in \eqref{eq:label}.
\end{lemma}
\begin{proof}
Consider any sketching linearization $\GA=(V,\EA,w)$ such that the embedding satisfies the $(\epsilon,\delta)$-accurate sufficient condition with only one output sketch dimension.
Based on \cref{lem:w_cut}, we have $\cut_{\GA}(U_i), \cut_{\GA}(V_i) \geq \log(m)$.
Based on \cref{lem:cost_linearization}, 
we have 
\begin{align*}
    \cost_G(U_i, V_i)
    &= \cost_R(U_i, V_i) + \cost_L(U_i, V_i) \\
    &= \cost_R(U_i, V_i) + \cut_{\GA}(U_i) + \cut_{\GA}(V_i) + \cut_{\GA}(V\setminus(U_i\cup V_i), U_i\cup V_i) \\
    & \geq \cost_R(U_i, V_i) + \cut_{\GA}(V_i) + \cut_{\GA}(V_i) \\
    & \geq \log(a_ib_ic_id_i)+ 2\log(m).
\end{align*}
Thus
this contraction has a cost of
$\bigOmega{ a_ib_ic_id_i\cdot m^2}.$
In addition, since
$U_i,V_i$ are subsets of the data vertices, $\cut_L(U_i,V_i)=0$. Therefore, based on \eqref{eq:linearization_p3},
\[\cut_{\GA}(U_i\cup V_i) = \cut_{\GA}(U_i) + \cut_{\GA}(V_i) \geq 2\log(m).\]
Based on \cref{lem:sketch_one_intermediate_lower}, the cost needed to sketch $U_i\cup V_i$ is $\bigOmega{a_ic_id_im^3}$. Thus
the overall asymptotic cost is lower bounded by 
$\bigOmega{a_ib_ic_id_im^2 + a_ic_id_im^3}$.
This finishes the proof. 
\end{proof}

\begin{lemma} \label{lem:lowerbound_gen}
Consider 
a specific contraction tree $T_0$ for $\Gd$, where the contraction $(U_i,V_i)$ is in $T_0$. For any embedding $\Ge$ satisfying the $(\epsilon,\delta)$-accurate sufficient condition
with only one output sketch dimension
and any contraction tree $T_B$ of $(\Gd,\Ge)$ constrained on $T_0$, the sketching asymptotic cost must be lower bounded by 
\begin{equation}\label{eq:lowerbound_gen_statement}
   \Omega(y_i) =  \Omega\left(a_ib_ic_id_im^2 + m^2d_i\sqrt{a_ib_ic_im}\cdot\min(\sqrt{a_i},\sqrt{c_i})\right),
\end{equation}
where $a_i,b_i,c_i,d_i$ are defined in \eqref{eq:label}, and $y_i$ is defined in \eqref{eq:yi}.
\end{lemma}
\begin{proof}
Consider any sketching linearization $\GA=(V,\EA,w)$ such that the embedding satisfies the $(\epsilon,\delta)$-accurate sufficient condition with only one output sketch dimension.
We first consider the case where the contraction $(U_i,V_i)$ exists in $T_B$. Based on \cref{lem:lowerbound_gen_case1}, 
the overall asymptotic cost is lower bounded by 
$\bigOmega{a_ib_ic_id_im^2 + a_ic_id_im^3}$.
Since \[a_ib_ic_id_im^2 + a_ic_id_im^3 = a_ic_id_im^2(b_i+m)\geq 2a_ic_id_im^2\sqrt{b_im},\] 
the overall asymptotic cost is lower bounded by 
\[
\bigOmega{a_ib_ic_id_im^2 + a_ic_id_im^2\sqrt{b_im}}
= \bigOmega{a_ib_ic_id_im^2 + a_im^2d_i\sqrt{b_ic_im}},
\]
and hence it satisfies \eqref{eq:lowerbound_gen_statement}. 

We next consider the other case where the contraction $(U_i,V_i)$ is not performed directly in $T_B$. 
Since $T_B$ is constrained on $T_0$, there must exist a contraction $(\hat{U}_i, \hat{V}_i)\in T_B$ with either $\hat{U}_i$ or $\hat{V}_i$ containing embedding vertices, and
$\hat{U}_i\cap \Vd = U_i$, $\hat{V}_i\cap \Vd = V_i$. Let $x$ be the last embedding vertex (based on the linearization order) applied in $T_B$ to $\hat{U}_i\cup \hat{V}_i$, so that 
$\cut_{\GA}(x,(\hat{U}_i\cup \hat{V}_i)\setminus\{x\}) = 0.$
For the case where $x\in \hat{U}_i\setminus U_i$, we show below that 
the sketching asymptotic cost is lower bounded by
\begin{equation}\label{eq:lowerbound_xU}
    \Omega\left(
a_ib_ic_id_im^2 + a_im^2d_i\sqrt{b_ic_im}
\right).
\end{equation}
For the other case where $x\in \hat{V}_i\setminus V_i$, 
we have the cost is lower bounded by $
\Omega\left(
a_ib_ic_id_im^2 + c_im^2d_i\sqrt{a_ib_im}
\right)
$
by symmetry.
Together, these two results prove the lemma.

\paragraph{Detailed proof of \eqref{eq:lowerbound_xU}}
Since $|\hat{U}_i|>|U_i|$, there must exist a contraction $(Y_1,Y_2)$, for which the output is $\hat{U}_i = Y_1\cup Y_2$. 
Based on \cref{lem:cost_linearization},
we have 
\begin{align*}
    \cost_G(Y_1, Y_2)
    &= \cost_R(Y_1, Y_2) + \cost_L(Y_1, Y_2) \\
    &= \cost_R(Y_1, Y_2) + \cut_{\GA}(Y_1) + \cut_{\GA}(Y_2) + \cut_{\GA}(V\setminus(Y_1\cup Y_2), Y_1\cup Y_2) \\
    & \geq \cost_R(Y_1, Y_2) + \cut_{\GA}(Y_1) + \cut_{\GA}(Y_2) + \cut_{\GA}(\hat{V}_i, \hat{U}_i)\\
    & \geq \log(a_ib_id_i)+ 2\log(m) +  \cut_{\GA}(\hat{V}_i, \hat{U}_i).
\end{align*}
Thus, the cost of  the contraction $(Y_1,Y_2)$ is lower bounded by 
\begin{align}\label{eq:lowerbound_gen_cost1}
    \bigOmega{ a_ib_id_im^2\cdot \exp\left(\cut_{\GA}(\hat{V}_i, \hat{U}_i)\right)}.
\end{align}
In addition, since
\begin{align*}
    \cost_G(\hat{U}_i, \hat{V}_i)
    &= \cost_R(\hat{U}_i, \hat{V}_i) + \cut_{\GA}(\hat{U}_i) + \cut_{\GA}(\hat{V}_i) + \cut_{\GA}(V\setminus(\hat{U}_i\cup \hat{V}_i), \hat{U}_i\cup \hat{V}_i) \\
    & \geq \log(a_ib_ic_id_i)+ 2\log(m) ,
\end{align*}
the contraction $(\hat{U}_i,\hat{V}_i)$ yields a cost lower bounded by
\begin{align}\label{eq:lowerbound_gen_cost2}
\bigOmega{ a_ib_ic_id_i\cdot m^2}.
\end{align}
Combining \eqref{eq:lowerbound_gen_cost1} and \eqref{eq:lowerbound_gen_cost2}, we have that the contractions  $(Y_1,Y_2)$ and $(\hat{U}_i,\hat{V}_i)$ have a cost of 
\begin{align}\label{eq:YS}
    \bigOmega{ a_ib_id_im^2\cdot \exp\left(\cut_{\GA}(\hat{V}_i, \hat{U}_i)\right)
    + a_ib_ic_id_i\cdot m^2
    }.
\end{align}
When $\cut_{\GA}(\hat{V}_i,\hat{U}_i) = \log(m)$, \eqref{eq:YS} implies that
the overall asymptotic cost is lower bounded by 
$\Omega(a_ib_ic_id_im^2 + a_ib_id_im^3)$.
Since 
\[
a_ib_ic_id_im^2 + a_ib_id_im^3 = a_ib_id_im^2(c_i+m)\geq 2a_ib_id_im^2\sqrt{c_im},
\] 
the overall asymptotic cost is lower bounded by 
\[
\Omega\left(a_ib_ic_id_im^2 + a_ib_id_im^2\sqrt{c_im}\right)
= \Omega\left(a_ib_ic_id_im^2 + a_im^2d_i\sqrt{b_ic_im}\right)
.\]
When $\cut_{\GA}(\hat{V}_i,\hat{U}_i) < \log(m)$, 
based on \cref{lem:graph_relation}, the effective sketch dimensions of $\hat{U}_i\cup \hat{V}_i$ satisfy
\begin{align}\label{eq:lowerbound_gen_sketchsize}
 \cut_{\GA}(\hat{U}_i\cup \hat{V}_i) &=
 \left(\cut_{\GA}(\hat{U}_i) 
 - 
 \cut_{\GA}(\hat{U}_i,\hat{V}_i)\right)
 + \cut_{\GA}(\hat{V}_i)  - \cut_{\GA}(\hat{V}_i,\hat{U}_i) \nonumber \\
& \geq \cut_{\GA}(x) + \cut_{\GA}(\hat{V}_i) - \cut_{\GA}(\hat{V}_i,\hat{U}_i)
\nonumber\\
& \geq 2\log(m) - \cut_{\GA}(\hat{V}_i,\hat{U}_i),   
\end{align}
where the first inequality holds since 
\begin{align*}
\cut_{\GA}(\hat{U}_i) - \cut_{\GA}(\hat{U}_i,\hat{V}_i)&= \cut_{\GA}(\hat{U}_i,V\setminus(\hat{U}_i\cup \hat{V}_i)) + 
\cut_{\GA}(\hat{U}_i,*)
\\
& \geq \cut_{\GA}(x,V\setminus(\hat{U}_i\cup \hat{V}_i)) + 
\cut_{\GA}(x,*) =   \cut_{\GA}(x),
\end{align*}
and the second inequality in \eqref{eq:lowerbound_gen_sketchsize} holds since  $\cut_{\GA}(x), \cut_{\GA}(\hat{V}_i) \geq \log(m)$ based on \cref{lem:w_cut}.
Based on the condition $\cut_{\GA}(\hat{V}_i,\hat{U}_i) < \log(m)$ as well as  \eqref{eq:lowerbound_gen_sketchsize}, we have $\cut_{\GA}(\hat{U}_i\cup \hat{V}_i) > \log(m)$.

Based on \cref{lem:sketch_one_intermediate_lower}, since $\cut_{\GA}(\hat{U}_i\cup \hat{V}_i) > \log(m)$,
there must exist another contraction in $T_B$ to sketch $\hat{U}_i\cup \hat{V}_i$ with a cost of  \begin{align*}
    \bigOmega{\exp\left(\cut_R(\hat{V}_i\cup \hat{U}_i) + \cut_{\GA}(\hat{V}_i\cup \hat{U}_i)\right)\cdot m} &=
    \Omega\left(a_ic_id_i\cdot
\exp\left(\cut_{\GA}(\hat{V}_i\cup \hat{U}_i)\right)
\cdot m\right)\\
& \overset{\eqref{eq:lowerbound_gen_sketchsize}}{=} \Omega\left(a_ic_id_i\cdot \frac{m^3}{\exp(\cut_{\GA}(\hat{V}_i, \hat{U}_i))}\right).
\end{align*}
 Let $\alpha = \exp\left(\cut_{\GA}(\hat{V}_i, \hat{U}_i)\right)$, the asymptotic cost is then lower bounded by 
\[
\Omega\left(a_ib_ic_id_im^2 + a_ib_id_im^2\cdot \alpha + a_ic_id_im^3 \cdot \frac{1}{\alpha}\right) = 
\Omega\left(
a_ib_ic_id_im^2 + a_im^2d_i\sqrt{b_ic_im}
\right).
\]
This finishes the proof.

\end{proof}

\begin{proof}[Proof of \cref{thm:lowerbound_uniform}]
Based on \cref{lem:lowerbound_gen}, 
the cost of $\Omega(y_i)$ is needed to sketch the contraction $(U_i,V_i)$.
Since $T_0$ contains contractions $(U_i,V_i)$ for $i\in [N-1]$, the asymptotic cost of $T_B$ must be lower bounded by $\bigOmega{\sum_{i=1}^{N-1}y_i}$.
In addition, in the analysis of \cref{lem:lowerbound_gen}, at least one embedding vertex is needed to sketch each contraction $(U_i,V_i)$,
thus $\Ne = \Omega(N)$ and $m = \Omega(N\log(1/\delta)/\epsilon^2)$ for the lower bound $\bigOmega{\sum_{i=1}^{N-1}y_i}$ to hold.

In addition, each $v_j\in V_D$ for $j\in [N]$ is adjacent to $e_j$ and each $w(e_j)>\log(m)$. Based on \cref{lem:sketch_one_intermediate_lower}, the asymptotic cost must be lower bounded by 
\[\bigOmega{\sum_{j=1}^N\exp\left(\cut_R(v_j) + \cut_{\GA}(v_j)\right)\cdot m}
= \bigOmega{\sum_{j=1}^N\exp\left(\cut_G(v_j)\right)\cdot m}
.\]
The above holds since $v_j$ is a vertex in the data graph, thus $\cut_{\GA}(v_j) = \cut_{\GL}(v_j)$ and $\cut_R(v_j) + \cut_{\GA}(v_j) = \cut_G(v_j)$.  This finishes the proof.
\end{proof}

\subsection{Sketching general data}\label{subsec:sketch_general}

In this section, we look at general tensor network data $\Gd$, where each vertex in $\Gd$ can either be adjacent to an edge to be sketched with weight greater than $\log(m)$ or not adjacent to any edge to be sketched. Below we consider any data contraction tree $T_0$ containing $\stD{1},\ldots, \stD{N}, \st{S}, \st{I}$ defined in \cref{subsec:summary_efficient_ebmedding}.  We also let $\stX{j}\subset V$ represent the sub network contracted by $\stD{j}$. 
We present the asymptotic sketching lower bound in \cref{thm:lowerbound_gen}.

\begin{lemma}\label{lem:sketch_Sj}
Consider $G_D$ with a data contraction tree $T_0$ containing $\stD{j}$, which is a set containing contractions such that $e_j$ is the only data edge adjacent to the contraction output and in $\bar{E}$ (set of data edges to be sketched).  For any embedding $\Ge$ satisfying the $(\epsilon,\delta)$-accurate sufficient condition
with only one output sketch dimension
and any contraction tree $T_B$ of $(\Gd,\Ge)$ constrained on the data contraction tree $T_0$, the sketching asymptotic cost must be lower bounded by $\Omega(\stt{j})$, where $\stt{j}$ is the optimal asymptotic cost to sketch the sub tensor network $X(e_j)$ (defined in \cref{tab:notations}) with an adjacent matrix in the Kronecker product embedding.
\end{lemma}
\begin{proof}

When $\stD{j} = \emptyset$, $X(e_j) = \{v_j\}$, where $v_j$ is the vertex in the data graph adjacent to $e_j$. As is analyzed in the proof of \cref{thm:lowerbound_uniform}, the asymptotic cost must be lower bounded by $\bigOmega{\exp\left(\cut_G(v_j)\right)\cdot m}$, which equals the asymptotic cost to contract $v_j$ with the adjacent embedding matrix.

Now we discuss the case where $\stD{j} \neq \emptyset$.
We first consider the case where there is a contraction $(\stX{j}, W)$ in $T_B$.
We show that under this case, the cost is lower bounded by $\Omega(\stt{j})$. We then show that for the case where there is no contraction $(\stX{j}, W)$ in $T_B$, meaning that some sub network of $\stX{j}$ is sketched, the cost is also lower bounded by $\Omega(\stt{j})$. Summarizing both cases prove the lemma.

Consider the case where there exists a contraction $(\stX{j}, W)$ in $T_B$. 
Contracting $\stX{j}$ yields a cost of $\Omega(\sum_{i\in \stD{j}} a_ib_ic_id_i\cdot s_j)$.
Next we analyze the contraction cost of $(\stX{j}, W)$.
Since $\stX{j}$ is the contraction output of $\stD{j}$, $W$ must either contain embedding vertices, or contain some data vertex adjacent to edges in $\bar{E}$ (edges to be sketched). Therefore, $W$ contains some vertex $v$ with $\cut_{\GA}(v)\geq \log(m)$. Based on \cref{lem:elu}, 
we have $\cut_{\GA}(W) \geq \log(m)$.
Therefore, 
\begin{align*}
    \cost_\GL(\stX{j}, W) &= \cut_{\GA}(\stX{j}) + \cut_{\GA}(W) + \cut_{\GA}(V\setminus(\stX{j}\cup W), \stX{j}\cup W) \\
    & \geq \log(s_j) + \log(m).
\end{align*}
Let $l\in \stD{j}$ denote the last contraction in $\stD{j}$, then we have $\cut_R(\stX{j}) = \log(a_lc_ld_l)$. Thus, we have
\begin{align}\label{eq:SjW}
    \cost_G(\stX{j}, W) &= \cost_{\GL}(\stX{j}, W) + \cost_{\GR}(\stX{j}, W) \\
    & \geq \cost_{\GL}(\stX{j}, W) + \cut_R(\stX{j})
 \geq  \log(a_lc_ld_ls_jm) ,
\end{align}
making the cost lower bounded by $\Omega\left(
\sum_{i\in \stD{j}} a_ib_ic_id_i \cdot s_j + a_lc_ld_ls_jm
\right)$. Note that contracting $\stX{j}$ and sketching the contraction output with a matrix in the Kronecker product embedding yields a cost of $\Theta\left(
\sum_{i\in \stD{j}} a_ib_ic_id_i \cdot s_j + a_lc_ld_ls_jm
\right)$, which upper-bounds the value of $\stt{j}$ based on definition. Thus the sketching cost is lower bounded by $\Omega(\stt{j})$.

Below we analyze the case where there is no contraction $(\stX{j}, W)$ in $T_B$.
Without loss of generality, for each contraction $(U_i,V_i)$ with $i\in \stD{j}$, assume that $U_i$ is adjacent to $e_j$. 
When $\stX{j}$ is not formed in $T_B$,
there must exist $U_k$ with $k\in \stD{j}$, and a contraction $(U_k, X)$ with $X\subset \Ve$ is in $T_B$. 
All contractions before $k$ yield a cost of 
\begin{equation}\label{eq:sketch_Sj_1}
   \Omega\left(
\sum_{i\in \stD{j}, i<k} a_ib_ic_id_i \cdot s_j\right). 
\end{equation}
Similar to the analysis for the contraction $(\stX{j},W)$ in \eqref{eq:SjW}, the contraction $(U_k, X)$ yields a cost of 
\begin{equation}\label{eq:sketch_Sj_2}
\Omega(a_kb_kd_ks_jm).
\end{equation}
For other contractions in $T_0$, $(U_i,V_i)$ with $i\in \stD{j}, i\geq k$, there must exist some contractions $(\hat{U}_i, \hat{V}_i)$ in $T_B$ with $\hat{U}_i \cap \Vd = U_i$, $ \hat{V}_i\cap \Vd = V_i$, since $T_B$ is constrained on $T_0$. Therefore, we have
\begin{align}\label{eq:sketch_Sj_3}
    \cost_G(\hat{U}_i, \hat{V}_i) &= \cost_{\GR}(\hat{U}_i, \hat{V}_i) + \cost_{\GL}(\hat{U}_i, \hat{V}_i) = \cost_{\GR}({U}_i, {V}_i) + \cost_{\GL}(\hat{U}_i, \hat{V}_i) \nonumber \\
    & \geq \cost_{\GR}({U}_i, {V}_i) + \cut_{\GA}(\hat{U}_i)  \geq \log(a_ib_ic_id_im).
\end{align}
In the last inequality in \eqref{eq:sketch_Sj_3} we use the fact that there exists a vertex $v\in U_i \subseteq \hat{U}_i$ with $\cut_{\GA}(v) = \log(s_j) \geq \log(m)$, then based on \cref{lem:elu},  $\cut_{\GA}(\hat{U}_i) \geq \log(m)$.

Combining \eqref{eq:sketch_Sj_1}, \eqref{eq:sketch_Sj_2} and \eqref{eq:sketch_Sj_3}, we have the cost is lower bounded by
\begin{align*}
  \bigOmega{f(k)}  = \bigOmega{\sum_{i\in \stD{j}, i<k} a_ib_ic_id_i \cdot s_j + a_kb_kd_ks_jm + \sum_{i\in \stD{j}, i\geq k}a_ib_ic_id_i \cdot m},
\end{align*}
where $f(k)$ represents the asymptotic cost to contract $\stX{j}$ with an embedding matrix adjacent to $e_j$, when sketching is performed at $k$th contraction with $k\in \stD{j}$.
Based on the definition of $\stt{j}$, we have $f(k) = \Omega(\stt{j})$,
thus finishing the proof.
\end{proof}

\begin{lemma}\label{lem:m2.5}
Consider any data $G_D$. For any embedding $\Ge$ satisfying the $(\epsilon,\delta)$-accurate sufficient condition
with only one output sketch dimension
and any contraction tree $T_B$ of $(\Gd,\Ge)$, the sketching asymptotic cost must be lower bounded by $\Omega(Nm^{2.5})$, where $m=\Omega(N\log(1/\delta)/\epsilon^2)$.
\end{lemma}
\begin{proof}
Let $G_D'$ be the data with the same set of sketching edges ($\bar{E}$) as $\Gd$, but $G_D'$ is a Kronecker product data. For any given contraction tree $T_B$ of $(\Gd, \Ge)$, there must exist a contraction tree of $(G_D', \Ge)$ whose asymptotic cost is upper bounded by the cost of $T_B$. Therefore, the asymptotic cost lower bound to contract $(G_D', \Ge)$ must also be the asymptotic cost lower bound to contract $(\Gd, \Ge)$.
Based on  \cref{cor:kronecker}, the asymptotic cost of $T_B$ must be lower bounded by 
\[
\bigOmega{\sum_{j= 1}^{N} s_jm + Nm^{2.5}} = \Omega(Nm^{2.5}).
\]
\end{proof}

\begin{theorem}\label{thm:lowerbound_gen}
  For any embedding $\Ge$ satisfying the $(\epsilon,\delta)$-accurate sufficient condition and any contraction tree $T_B$ of $(\Gd,\Ge)$ constrained on the data contraction tree $T_0$ expressed in \eqref{eq:exp_T0}, the sketching asymptotic cost must be lower bounded by 
\begin{equation}\label{eq:lowerbound_gen}
  \Omega\left(
\sum_{j=1}^{N} \stt{j} + 
\sum_{i\in \st{S}}a_ib_ic_id_im^2 + Nm^{2.5} + 
\sum_{i\in \st{I}}z_i
\right) ,  
\end{equation}
where $m=\Omega(N\log(1/\delta)/\epsilon^2)$, $a_i,b_i,c_i,d_i$ are expressed in \eqref{eq:label}, 
$\stt{j}$ is the optimal asymptotic cost to sketch the sub tensor network $X(e_j)$ (the sub network contracted by $\stD{j}$, also defined in \cref{tab:notations}) with an adjacent matrix in the Kronecker product embedding,
and $z_i$ is expressed in \eqref{eq:zi}.
\end{theorem}
\begin{proof}
The term $\sum_{j=1}^N \stt{j}$ can be proven based on \cref{lem:sketch_Sj}, and the term $Nm^{2.5}$ with $m=\Omega(N\log(1/\delta)/\epsilon^2)$ can be proven based on \cref{lem:m2.5}. Below we show the asymptotic cost is also lower bounded by $\Omega(\sum_{i\in \st{S}}a_ib_ic_id_im^2 + \sum_{i \in \st{I}} z_i)$, thus finishing the proof.

For each contraction $(U_i,V_i)$ in $T_0$ with $i\in \st{S}\cup \st{I}$, there must exist a contraction $(\hat{U}_i, \hat{V}_i)$ in $T_B$, and $\hat{U}_i \cap \Vd = U_i$, $\hat{V}_i\cap\Vd = V_i$. For the case where $i\in \st{S}$, since both $U_i$ and $V_i$ contain edges to be sketched, 
we have $\cut_{\GA}(\hat{U}_i) \geq \log(m)$ and $\cut_{\GA}(\hat{V}_i) \geq \log(m)$ based on \cref{lem:elu}. Therefore, we have
\begin{align}
   \sum_{i\in \st{S}} \cost_G(\hat{U}_i, \hat{V}_i) &= \sum_{i\in \st{S}} \cost_{\GR}(\hat{U}_i, \hat{V}_i) + \cost_{\GL}(\hat{U}_i, \hat{V}_i) \nonumber \\
    &= \sum_{i\in \st{S}} \cost_{\GR}({U}_i, {V}_i) + \cost_{\GL}(\hat{U}_i, \hat{V}_i) \nonumber \\
    & \geq \sum_{i\in \st{S}} \cost_{\GR}({U}_i, {V}_i) + \cut_{\GA}(\hat{U}_i) + \cut_{\GA}(\hat{V}_i) \nonumber \\
    & \geq \sum_{i\in \st{S}} \log(a_ib_ic_id_im^2),
\end{align}
where the first inequality above holds based on \cref{lem:cost_linearization}.
This shows the cost is lower bounded by $\Omega(\sum_{i\in \st{S}}a_ib_ic_id_im^2)$. 

Now consider the case where $i\in \st{I}$. In this case, either $\cut_{\GA}(U_i \cup {V}_i) = 0$ or $\cut_{\GA}(U_i \cup {V}_i) \geq \log(m)$. When $\cut_{\GA}(U_i \cup {V}_i) = 0$, we have $\cut_{\GA}(\hat{U}_i \cup \hat{V}_i) \geq  \cut_{\GA}(U_i \cup {V}_i)$. When $\cut_{\GA}(U_i \cup {V}_i) \geq \log(m)$, 
based on \cref{lem:elu}, we have $\cut_{\GA}(\hat{U}_i \cup \hat{V}_i) \geq \log(m)$. To summarize, we have
\[
\cut_{\GA}(\hat{U}_i \cup \hat{V}_i) \geq  \min\left(\cut_{\GA}(U_i \cup {V}_i), \log(m)\right),
\]
thus
\begin{align}
   \sum_{i\in \st{I}} \cost_G(\hat{U}_i, \hat{V}_i) &= \sum_{i\in \st{I}} \cost_{\GR}(\hat{U}_i, \hat{V}_i) + \cost_{\GL}(\hat{U}_i, \hat{V}_i) \nonumber \\
    & \geq \sum_{i\in \st{I}} \cost_{\GR}({U}_i, {V}_i) + \cut_{\GA}(\hat{U}_i) + \cut_{\GA}(\hat{V}_i) \nonumber \\
    & \geq \sum_{i\in \st{I}} \cost_{\GR}({U}_i, {V}_i) + \cut_{\GA}(\hat{U}_i\cup\hat{V}_i) \nonumber \\
    & \geq \sum_{i\in \st{I}} \log(a_ib_ic_id_i)+ \min\left(\cut_{\GA}(U_i\cup V_i), \log(m)\right)\nonumber\\
    & = \sum_{i\in \st{I}}\log(z_i).
\end{align}
This shows the sketching cost is lower bounded by $\Omega\left(\sum_{i\in \st{I}}z_i\right)$, thus finishing the proof.
\end{proof}

\begin{proof}[Proof of \cref{thm:approx_factor}]
Based on \cref{thm:complexity}, the computational cost of \cref{alg:contract_sketch} is 
\[
\alpha = 
  \Theta\left(
\sum_{j=1}^{N} \stt{j} + 
\sum_{i\in \st{S}}y_i + 
\sum_{i\in \st{I}}z_i
\right).
\]
Let $\beta$ equals the expression in \eqref{eq:lowerbound_gen}. We have 
\begin{align*}
  \frac{\alpha}{\beta} &=  \frac{
\Theta\left(
\sum_{j=1}^{N} \stt{j} + 
\sum_{i\in \st{S}}y_i + 
\sum_{i\in \st{I}}z_i
\right)
}{
\Omega\left(
\sum_{j=1}^{N} \stt{j} + 
\sum_{i\in \st{S}}(a_ib_ic_id_im^2 + m^{2.5}) + 
\sum_{i\in \st{I}}z_i
\right)
}
= \bigO{\frac{
\sum_{i\in \st{S}}y_i
}{
\sum_{i\in \st{S}}(a_ib_ic_id_im^2 + m^{2.5})
}}   \\
& = \bigO{\max_{i\in \st{S}}\frac{
y_i
}{
a_ib_ic_id_im^2 + m^{2.5}
}} 
= \bigO{\max_{i\in \st{S}}\frac{
a_ib_ic_id_im^2 + m^2d_i\sqrt{a_ib_ic_im}\cdot\min(\sqrt{a_i},\sqrt{c_i})
}{
a_ib_ic_id_im^2 + m^{2.5}
}} \\
& = O(1) + \bigO{\max_{i\in \st{S}}\frac{ m^2d_i\sqrt{a_ib_ic_im}\cdot\min(\sqrt{a_i},\sqrt{c_i})
}{
a_ib_ic_id_im^2 + m^{2.5}
}}.
\end{align*}
Below we derive asymptotic upper bound of the term $\theta = \frac{ m^2d_i\sqrt{a_ib_ic_im}\cdot\min(\sqrt{a_i},\sqrt{c_i})
}{
a_ib_ic_id_im^2 + m^{2.5}
}$. We analyze the case below with $a_i\leq c_i$,
and the other case with $a_i > c_i$ can be analyzed in a similar way based on the symmetry of $a_i,c_i$ in $\theta$.

When $a_i\leq c_i$, we have $m^2d_i\sqrt{a_ib_ic_im}\cdot\min(\sqrt{a_i},\sqrt{c_i}) = a_im^2d_i\sqrt{b_ic_im}$. We consider two cases, one satisfies $\sqrt{b_ic_im} \leq b_ic_i$ and the other satisfies $\sqrt{b_ic_im} > b_ic_i$.

When $\sqrt{b_ic_im} \leq b_ic_i$, we have $\theta\leq 1$, thus $\frac{\alpha}{\beta} = O(1)$, thus satisfying the theorem statement. 

When $\sqrt{b_ic_im} > b_ic_i$, which means that ${m}> {b_ic_i}$, we have 
\begin{align*}
  \theta &= 
\frac{a_im^2d_i\sqrt{b_ic_im}}{a_ib_ic_id_im^2 + m^{2.5} } \leq 
\min\left(
\frac{\sqrt{m}}{\sqrt{b_ic_i}}, 
a_id_i\sqrt{b_ic_i}
\right)
\leq \sqrt{m},
\end{align*}
thus $
\frac{\alpha}{\beta} \leq O\left(\sqrt{m}\right).
$
In addition, when $\Gd$ is a graph, we have $d_i=1$ for all $i$. Therefore, 
\begin{align*}
    \theta &\leq \min\left(
\frac{\sqrt{m}}{\sqrt{b_ic_i}}, 
a_id_i\sqrt{b_ic_i}
\right) = \min\left(
\frac{\sqrt{m}}{\sqrt{b_ic_i}}, 
a_i\sqrt{b_ic_i}
\right) \\
&\leq \min\left(
\frac{\sqrt{m}}{\sqrt{b_ic_i}}, 
c_i\sqrt{b_ic_i}
\right) \leq 
\min\left(
\frac{\sqrt{m}}{(b_ic_i)^{1/2}}, 
(b_ic_i)^{3/2}
\right) \leq m^{0.375}.
\end{align*}
Therefore, in this case we have $
\frac{\alpha}{\beta} \leq O\left(m^{0.375}\right),
$
thus finishing the proof.
\end{proof}

%% file: tree_embedding.tex
\section{Analysis of tree tensor network embeddings}\label{sec:appendix_tree} 

In this section, we provide detailed analysis of sketching with tree embeddings. The algorithm to sketch with tree embedding is similar to \cref{alg:contract_sketch}, and the only difference is that for each contraction $(U_i,V_i)$ with $i\in \st{S}$, 
such that both $U_i$ and $V_i$ are adjacent to edges in $\bar{E}$, we sketch it with one embedding tensor $z_i$  rather than a small network. 
Let $\hat{U}_i,\hat{V}_i$ denote the sketched $U_i$ and $V_i$ formed in previous contractions in the sketching contraction tree $T_B$, such that $\hat{U}_i\cap \Vd = U_i$ and $\hat{V}_i\cap \Vd = V_i$, we sketch $(\hat{U}_i,\hat{V}_i)$ 
via the contraction path $((\hat{U}_i,\hat{V}_i), z_i)$. 
For the case where each vertex in the data tensor network is adjacent to an edge to be sketched, the sketching cost would be 
\begin{align}\label{eq:cost_uniform_tree}
\Theta\left(
\sum_{j=1}^{N} \exp(\cut_G(v_j))\cdot m + 
\sum_{j=1}^{N-1}(a_ib_ic_id_im^2 + a_ic_id_im^3)
\right) ,
\end{align}
where $v_j$ is the vertex in the data graph adjacent to $e_j$, $a_i,b_i,c_i,d_i$ are defined in \eqref{eq:label}, and
we replace the term $y_i = a_ib_ic_id_im^2 + m^2d_i\sqrt{a_ib_ic_im}\cdot\min(\sqrt{a_i},\sqrt{c_i})$ in \eqref{eq:cost_uniform} with $a_ib_ic_id_im^2 + a_ic_id_im^3$. 

\begin{proof}[Proof of \cref{thm:tree_optimal}]
Since each contraction in $T_0$ contracts dimensions with size being at least the sketch size $m$, we have $b_i \geq m$ for $i\in \st{S}$.
Therefore, 
\[
m^2d_i\sqrt{a_ib_ic_im}\cdot\min(\sqrt{a_i},\sqrt{c_i})\leq a_im^2d_i\sqrt{b_ic_im}\leq a_ib_id_im^2\sqrt{c_i}\leq a_ib_ic_id_im^2,
\]
and the asymptotic cost in \eqref{eq:cost_uniform} would be 
\begin{equation}\label{eq:tree_cost}
   \Theta\left(
\sum_{j=1}^{N} \exp(\cut_G(v_j))\cdot m + 
\sum_{j=1}^{N-1}a_ib_ic_id_im^2
\right) . 
\end{equation}
Based on \cref{thm:lowerbound_uniform}, \eqref{eq:tree_cost} matches the sketching asymptotic cost lower bound for this data. Since $a_ic_id_im^3\leq a_ib_ic_id_im^2$ so \eqref{eq:cost_uniform_tree} equals \eqref{eq:tree_cost}, 
sketching with tree embeddings also yield the optimal asymptotic cost.
\end{proof}

When
the data has a Kronecker product structure, 
sketching with tree tensor network embedding is less efficient compared to \cref{alg:contract_sketch}. 
As is shown in \eqref{eq:sketchkronecker_tree}, \cref{alg:contract_sketch} yields a cost of  $\Theta\left(
\sum_{j= 1}^{N} s_jm + 
Nm^{2.5}
\right)$
to sketch the Kronecker product data. However, for tree embeddings, the asymptotic cost
\eqref{eq:cost_uniform_tree} is equal to
\begin{align}\label{eq:sketchkronecker_tree}
\Theta\left(
\sum_{j= 1}^{N} s_jm + 
Nm^{3}
\right).
\end{align}

%% file: application.tex
\section{Computational cost analysis of sketched CP-ALS}\label{sec:cp}

In this section, we provide detailed computational cost analysis of the sketched CP-ALS algorithm based on \cref{alg:contract_sketch}.
We are given a tensor $\tsr{X}\in \R^{s\times \cdots \times s}$, and aim to decompose that into $N$ factor matrices, $A_i\in \R^{s\times R}$ for $i\in [N]$. 
Let $L_i= \mat{A}_1 \odot \cdots \odot  \mat{A}_{i-1}  \odot  \mat{A}_{i+1}\odot \cdots \odot \mat{A}_N$ and 
$R_i=\mat{X}_{(i)}^T$. In each iteration, we aim to update $A_i$ via solving a sketched linear least squares problem, 
$\mat{A}_i = \argmin{\mat{A}} \left\|S_iL_iA^T - S_iR_i\right\|_F^2,$
where $S_i$ is an embedding constructed based on \cref{alg:contract_sketch}. 

Below we first discuss the sketch size of $S_i$ sufficient to make each sketched least squares problem accurate. We then discuss the contraction trees of $L_i$, on top of which embedding structures are determined. We select contraction trees such that contraction intermediates can be reused across subproblems. Finally,
we present the detailed computational cost analysis of the sketched CP-ALS algorithm.

\subsection{Sketch size sufficient for accurate least squares subproblem}\label{subsec:cp_sketchsize}
Since the tensor network of $L_i$ contains $N$ output dimensions and $L_i$ contains $R$ columns, we show below that a sketch size of $\Theta(NR\log(1/\delta)/\epsilon^2)= \Tilde{\Theta}(NR/\epsilon^2)$ is sufficient for the least squares problem to be $(\epsilon,\delta)$-accurate.

\begin{theorem}\label{thm:mps_lsq}
Consider the sketched linear least squares problem
$\min_{\mat{A}} \left\|S_iL_iA^T - S_iR_i\right\|_F^2$. Let $\mat{S}_i$ be an embedding constructed based on \cref{alg:contract_sketch}, with the sketch size $m =\Theta(NR\log(1/\delta)/\epsilon^2)$, solving the sketched least squares problem gives us an $(1+\epsilon)$-accurate solution with probability at least $1-\delta$.
\end{theorem}
\begin{proof}
\cref{alg:contract_sketch} outputs an embedding with $\Theta(N)$ vertices. Based on \cref{thm:subgaussian_embedding}, a sketched size of $\Theta(NR\log(1/\delta)/\epsilon^2)$ will make the embedding $(\epsilon,\frac{\delta}{e^R})$-accurate.
Based on the $\epsilon$-net argument~\cite{woodruff2014sketching}, $\mat{S}$ is the $(\epsilon,\delta)$-accurate subspace embedding for a subspace with dimension $R$. Therefore, we can get an $(1+\epsilon)$-accurate solution with probability at least $1-\delta$ for the least squares problem.
 \end{proof}

\subsection{Data contraction trees and efficient embedding structures}\label{subsec:cp_contractiontree}

The structures of embeddings $S_1,\ldots,S_N$ also depend on the data contraction trees for $L_1,\ldots,L_N$. We denote the contraction tree of $L_i$ as $T_{i}$.
We construct $T_i$ for $i\in [N]$
such that resulting embeddings $S_1,\ldots,S_N$ have common parts, which yields more efficient sketching computational cost via reusing contraction intermediates.

Let the vertex $v_i$ represent the matrix $A_i$. We also let $V_L^{(i)} = \{v_1,\ldots, v_i\}$ denote the set of all first $i$ vertices, and let $V_R^{(i)} = \{v_i,\ldots, v_N\}$ denote the set of vertices from $v_i$ to $v_N$.
In addition, we let $\st{C}_L^{(i)}$ denote a contraction tree to fully contract $V_L^{(i)}$, from $v_1$ to $v_i$. Let $\st{C}_L^{(1)} = \emptyset$, we have for all $i\geq 1$, $\st{C}_L^{(i+1)} = \st{C}_L^{(i)} \cup \left\{(V_L^{(i)}, v_{i+1})\right\}$.
Similarly, we let $\st{C}_R^{(i)}$ denote a contraction tree to fully contract $V_R^{(i)}$, from $v_N$ to $v_i$. Let $\st{C}_R^{(N)} = \emptyset$, we have for all $i\leq N$, $\st{C}_R^{(i-1)} = \st{C}_R^{(i)} \cup \left\{(V_R^{(i)}, v_{i-1})\right\}$.

\begin{figure}[!ht]
\centering

\subfloat[]{\includegraphics[width=0.2\textwidth, keepaspectratio]{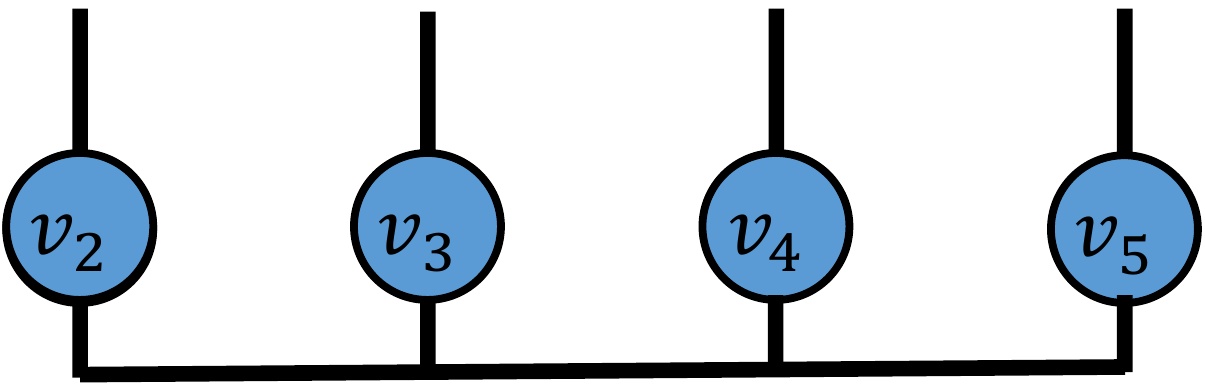}}
\subfloat[]{\includegraphics[width=0.2\textwidth, keepaspectratio]{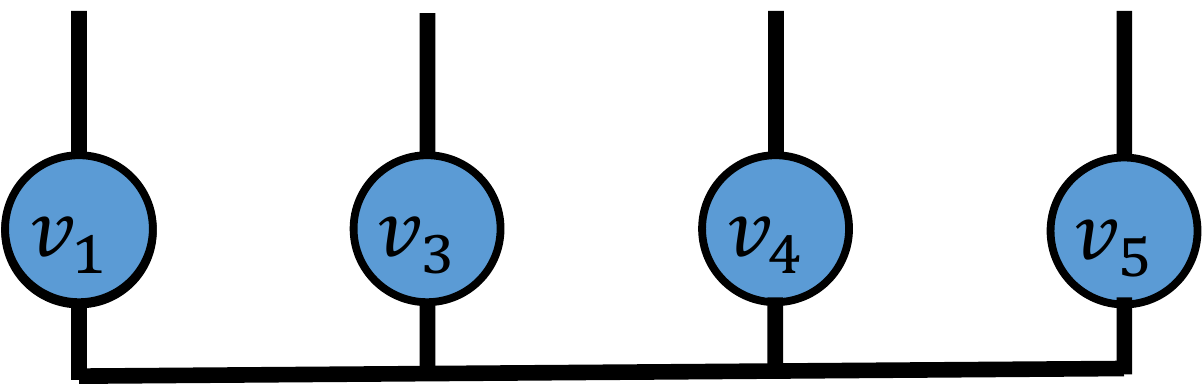}}
\subfloat[]{\includegraphics[width=0.2\textwidth, keepaspectratio]{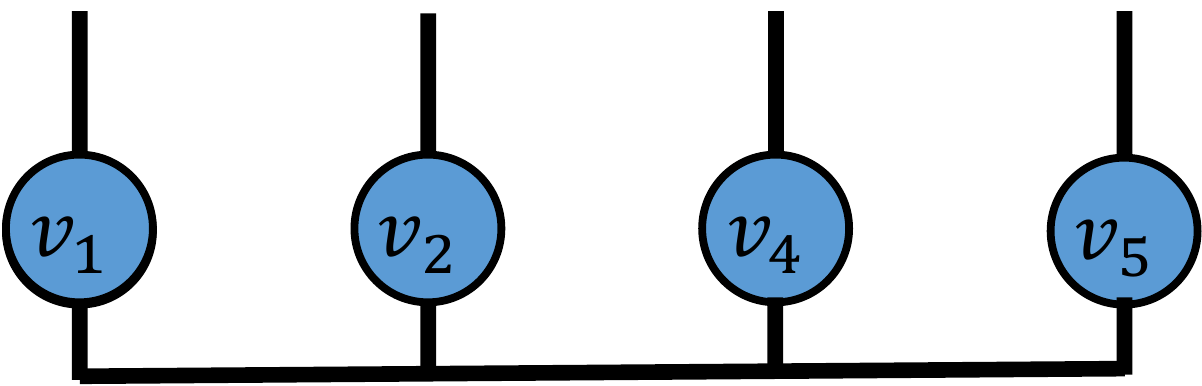}}
\subfloat[]{\includegraphics[width=0.2\textwidth, keepaspectratio]{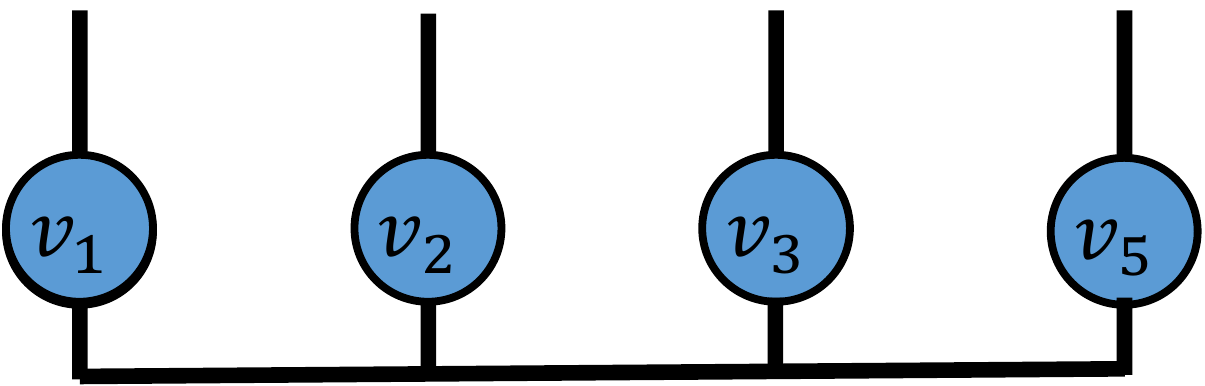}}
\subfloat[]{\includegraphics[width=0.2\textwidth, keepaspectratio]{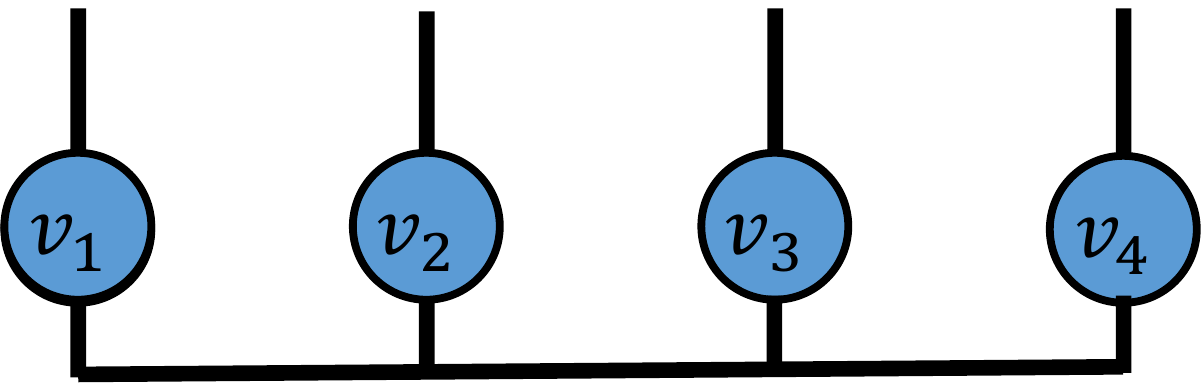}}

\subfloat[]{\includegraphics[width=0.2\textwidth, keepaspectratio]{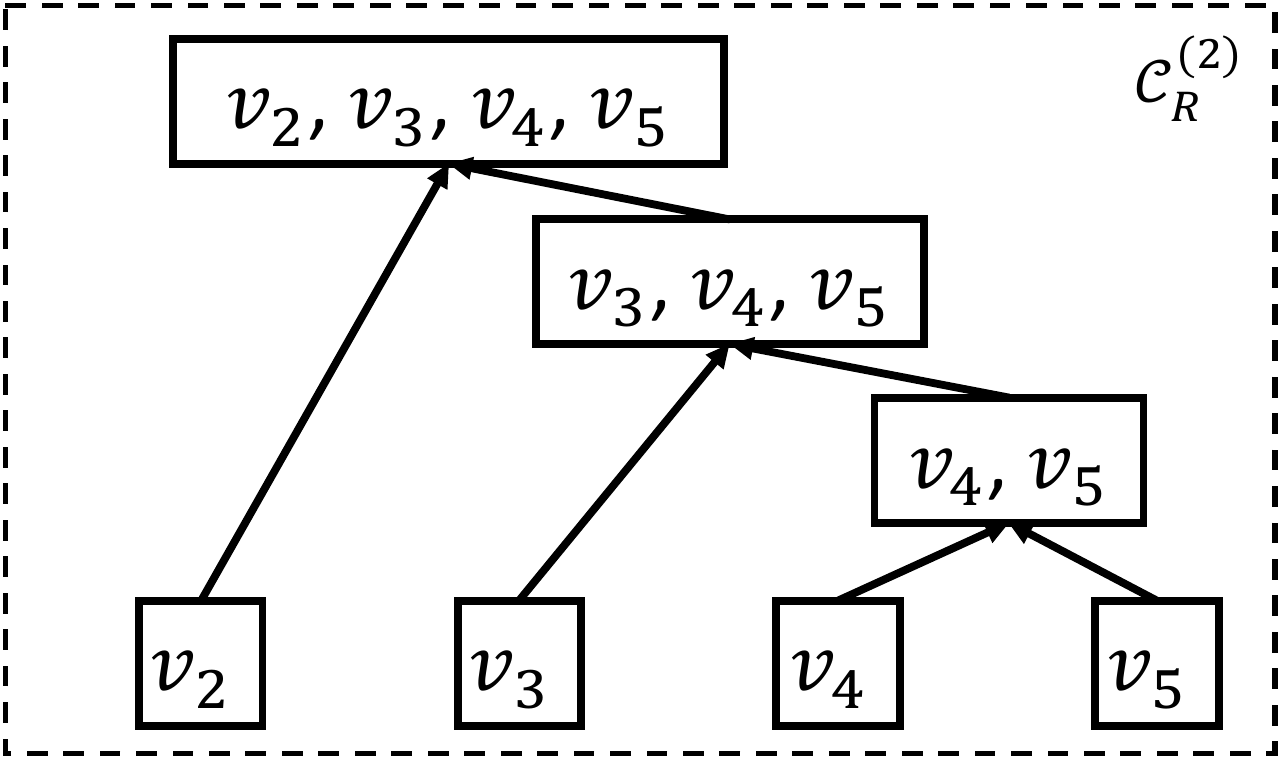}}
\subfloat[]{\includegraphics[width=0.2\textwidth, keepaspectratio]{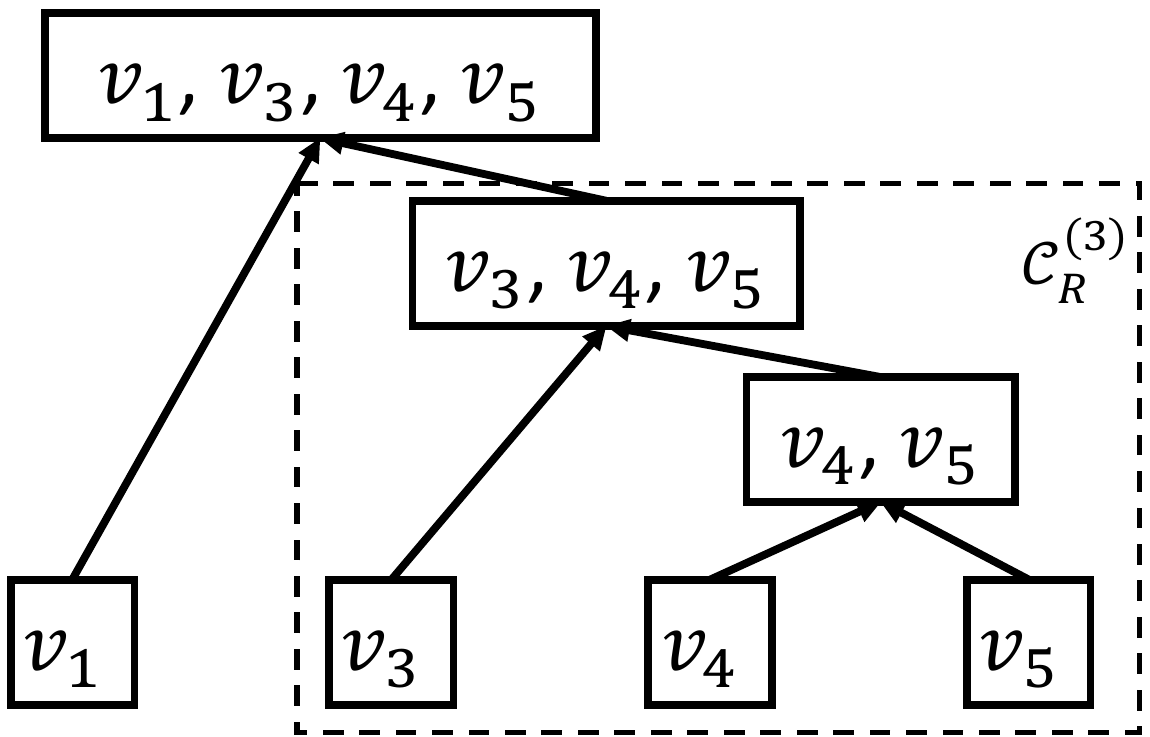}}
\subfloat[]{\includegraphics[width=0.2\textwidth, keepaspectratio]{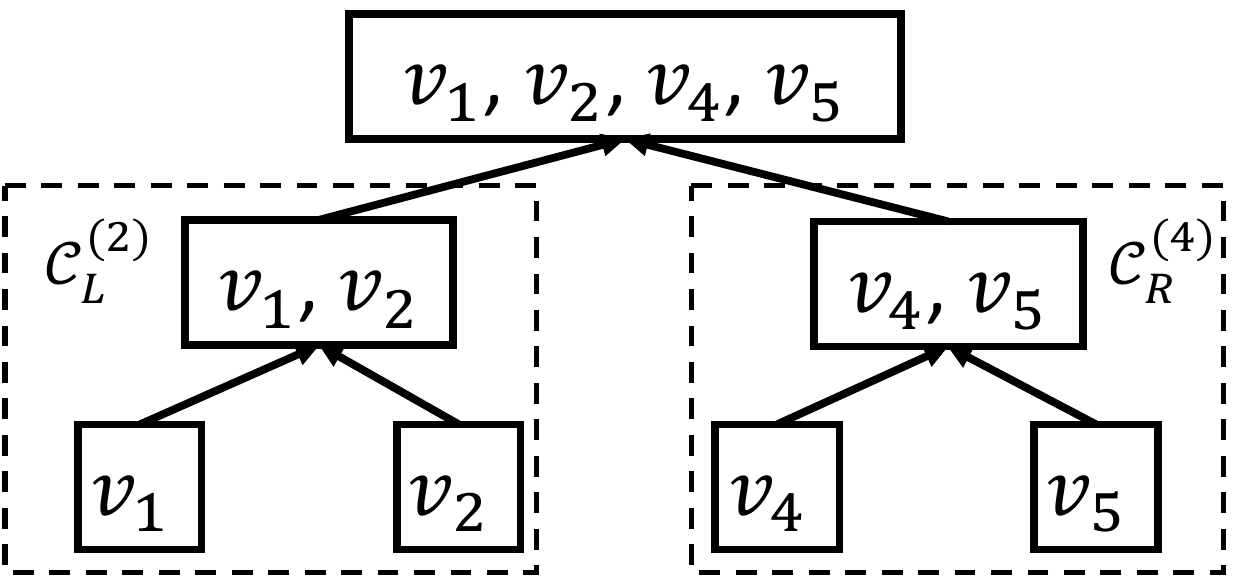}}
\subfloat[]{\includegraphics[width=0.2\textwidth, keepaspectratio]{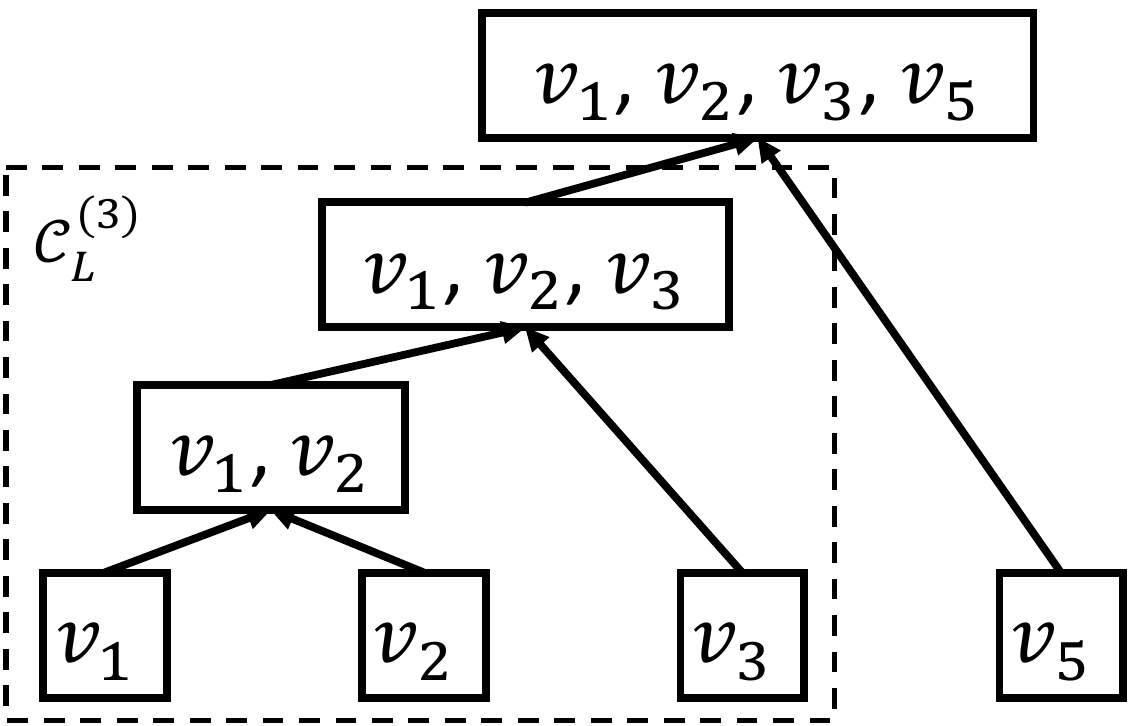}}
\subfloat[]{\includegraphics[width=0.2\textwidth, keepaspectratio]{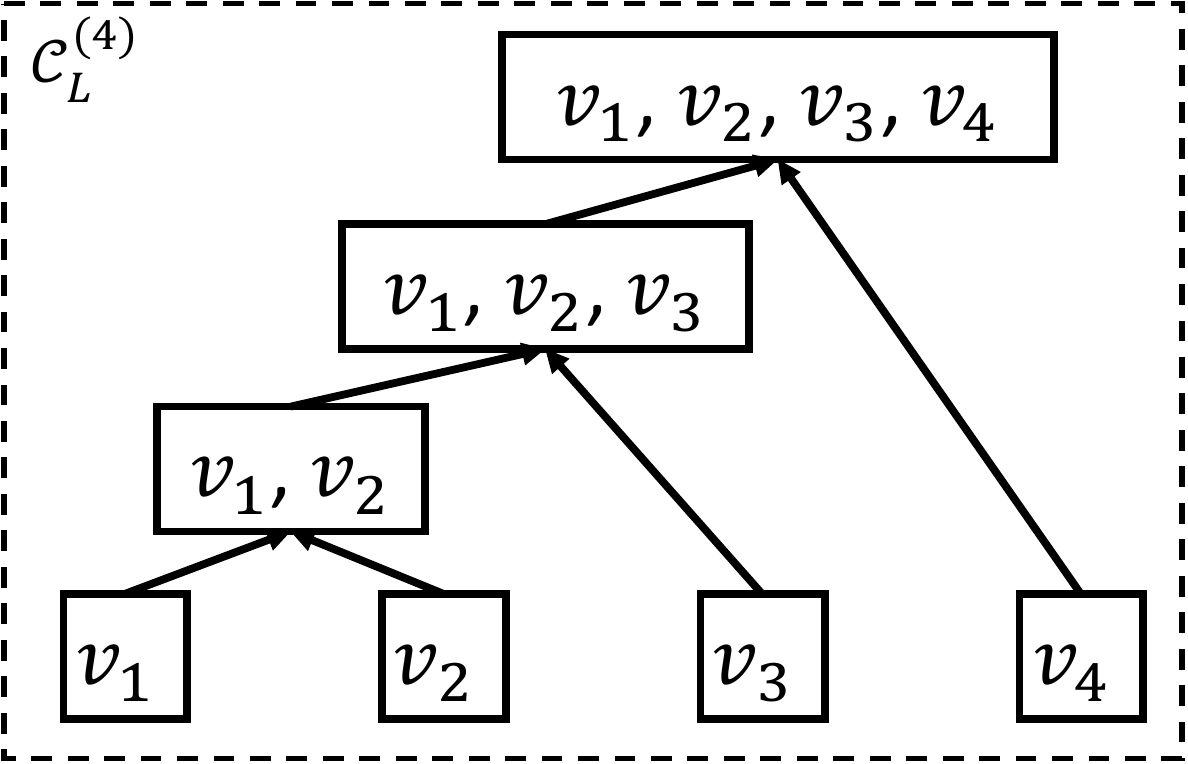}}

\caption{(a)-(e): Representations of $L_1,\ldots, L_5$ for the CP decomposition of an order 5 tensor. (f)-(j): Data dimension trees $T_1,\ldots, T_5$. 
}
\label{fig:cpd_contraction_trees}
\end{figure}

Note that the vertex set of the tensor network of $L_i$ is $V_L^{(i-1)}\cup V_R^{(i+1)}$. 
Each $T_i$ is constructed so that $V_L^{(i-1)},V_R^{(i+1)}$ are first contracted via the contraction trees $\st{C}_L^{(i-1)}$, $\st{C}_R^{(i+1)}$, respectively, then a contraction of $(V_L^{(i-1)},V_R^{(i+1)})$ is used to contract them into a single tensor. We illustrate $T_i$ for the CP decomposition of an order 5 tensor in \cref{fig:cpd_contraction_trees}.

These tree structures allow us to reuse contraction intermediates during sketching.
On top of $T_1$, sketching $L_1$ using \cref{alg:contract_sketch} yields a cost of $\Theta(N(smR + m^{2.5}R))$, where the term $\Theta(NsmR)$ comes from sketching with the Kronecker product embedding, and the term $\Theta(Nm^{2.5}R)$ comes from 
sketching each data contraction in $\st{C}_R^{(2)}$.
Since  $\st{C}_R^{(2)} = \st{C}_R^{(3)} \cup \left\{(V_R^{(3)}, v_{2})\right\}$, all contractions in $\st{C}_R^{(3)}$ are sketched, and we obtain
the sketching output of $V_R^{(3)}$, which is denoted as $\hat{V}_R^{(3)}$ below.

We use $\hat{V}_R^{(3)}$ formed during sketching $L_1$ to sketch $L_2$. Since $T_2$ contains contractions
\begin{align*}
   T_{2} &= \st{C}_R^{(3)}\cup
\st{C}_L^{(1)}\cup
\left\{(V_L^{(1)},V_R^{(3)})\right\}  
= \st{C}_R^{(3)}
\cup 
\left\{(v_{1},V_R^{(3)})\right\},
\end{align*}
through reusing $\hat{V}_R^{(3)}$, we only need to sketch $(v_{1},\hat{V}_R^{(3)})$ to compute $S_2L_2$, which 
only costs $\Theta(smR + m^{2.5}R)$.
Similarly, sketching each $L_i$ for $i\geq 2$ only costs $\Theta(smR + m^{2.5}R)$, thus making the overall cost of sketching $L_1,\ldots,L_N$ being $\Theta(N(smR + m^{2.5}R))$.

\subsection{Detailed algorithm and the overall computational cost}\label{subsec:cp_algorithm}

\begin{algorithm}[!ht]
\caption{\textbf{Sketched-ALS}: Sketched ALS for CP decomposition}\label{alg:sketchals}
\begin{algorithmic}[1]
\STATE{\textbf{Input: }Input tensor $\tsr{X}$, initializations $\mat{A}_1,\ldots,\mat{A}_N$,
maximum number of iterations $I_{\text{max}}$
}
\STATE{$G_{D}(L_i)\leftarrow $ structure of the data $L_i =  \mat{A}_1 \odot \cdots \odot  \mat{A}_{i-1}  \odot  \mat{A}_{i+1}\odot \cdots \odot \mat{A}_N$ for $i\in [N]$ }
\STATE{$T_{i}\leftarrow $ contraction tree of $G_{D}(L_i)$ for $i\in [N]$ constructed based on  \cref{subsec:cp_contractiontree} }
\STATE{Build tensor network embeddings $S_i$ on $G_{D}(L_i)$ and $T_i$ based on \cref{alg:contract_sketch} for $i\in[N]$\label{line:cp_buildS}}
\STATE{Compute $\hat{R}_i \leftarrow S_i\mat{X}_{(i)}^T$ for $i\in[N]$\label{line:cp_rhs}}
\FOR{{$t\in [I_{\text{max}}] $}} {
	\FOR{{$i\in [N] $}\label{line:cpiteration_start}} {
   \STATE{Compute $\hat{L}_i\leftarrow S_iL_i$\label{line:cp_lhs}}
    \STATE{$ \mat{A}_i\leftarrow  \argmin{\mat{X}} \left\|\hat{L}_iX - \hat{R}_i\right\|_F^2$\label{line:cp_solve}}
	}\ENDFOR\label{line:cpiteration_end}
}\ENDFOR
\RETURN $\mat{A}_1,\ldots,\mat{A}_N$ 
\end{algorithmic}
\end{algorithm}

We present the detailed sketched CP-ALS algorithm in \cref{alg:sketchals}. Here we analyze the overall computational cost of the algorithm. 

Line~\ref{line:cp_rhs} yields a preparation cost of the algorithm. Note that we construct $S_i$ based on \cref{subsec:cp_contractiontree}, where they share common tensors. Contracting $S_1X_{(1)}^T$ yields a cost of $\Theta(s^Nm)$. On top of that, contracting $S_iX_{(i)}^T$ for $i\geq 2$ also only yields a cost of $\Theta(s^Nm)$, making the overall preparation cost $\Theta(s^Nm)$.

Within each ALS iteration (Lines \ref{line:cpiteration_start}-\ref{line:cpiteration_end}), based on \cref{subsec:cp_contractiontree}, computing $S_iL_i$ for $i\in[N]$ costs $\Theta(N(smR + m^{2.5}R))$. For each $i\in [N]$, line \ref{line:cp_solve} costs $\Theta(mR^2)$, making the cost of per-iteration least squares solves $\Theta(NmR^2)$. 
Based on \cref{subsec:cp_sketchsize}, a sketch size of $m= \Tilde{\Theta}(NR/\epsilon^2)$ is sufficient for the least squares solution to be $(1+\epsilon)$-accurate with probability at least $1-\delta$. 
Overall, the per-iteration cost is $\Theta(N(smR + m^{2.5}R)) =  \Tilde{\Theta}(N^2(N^{1.5}R^{3.5}/\epsilon^3 + sR^2)/\epsilon^2)$.

\section{Computational cost analysis of sketching for tensor train rounding}\label{sec:tt}
We provide the computational cost lower bound analysis of computing $SX$, where $\mat{X}$ denotes a matricization of  the tensor train data shown in \cref{fig:ttX}.
This step is the computational bottleneck of the tensor train  randomized rounding algorithm proposed in 
\cite{daas2021randomized}.
As is discussed in 
\cref{sec:app}, we assume 
the tensor train has order $N$ with the output dimension sizes equal $s$, the tensor train rank is $R<s$, and the goal is to round the rank to $r<R$. 
The sketch size $m$ of $S$ is $r$ plus some constant, and is assumed to be smaller than $R$. The lower bound is derived within all embeddings satisfying the sufficient condition in \cref{thm:subgaussian_embedding} and only have one output sketch dimension with size $m$.

\begin{figure}
\centering
\includegraphics[width=.4\textwidth, keepaspectratio]{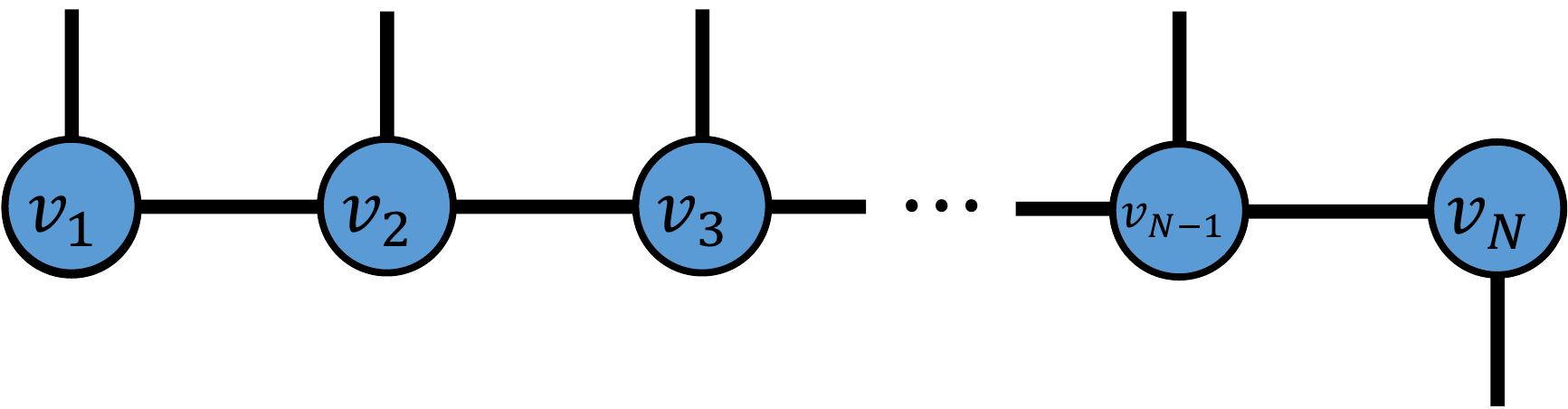}
  \caption{Illustration of the matricization of the tensor train ($X$) to be sketched. The $N-1$ uncontracted edges adjacent to $v_1,\ldots, v_{N-1}$ are to be sketched.
  }
  \label{fig:ttX}
\end{figure} 

For the data contraction tree that contracts the tensor train shown in \cref{fig:ttX} from left to right, we have $a_i = 1, b_i = R, c_i = R, d_i=1$ for $i\in [N-2]$, where $a_i,b_i,c_i,d_i$ are expressed in \eqref{eq:label}. Based on \cref{thm:lowerbound_gen}, the sketching asymptotic cost lower bound is 
\begin{align*}
  \Omega\Bigg(
\sum_{j=1}^{N-1} \stt{j} + 
\sum_{i\in \st{S}}a_ib_ic_id_im^2 + Nm^{2.5} + 
&\sum_{i\in \st{I}}z_i
\Bigg)
=\Omega\left(
\sum_{j=1}^{N-1} \stt{j} + 
\sum_{i\in \st{S}}a_ib_ic_id_im^2
\right)
\\
  & =\Omega\left(
\sum_{j=1}^{N-1} \exp(\cut_G(v_j))\cdot m + 
\sum_{j=1}^{N-2}a_ib_ic_id_im^2
\right) \\
& = \Omega\left(
N sR^2 m + NR^2m^2
\right)  = \Omega\left(
N sR^2 m 
\right) .
\end{align*}
Above we use the fact that  $\exp(\cut_G(v_1)) = sR$, and for $j\in\{2,\ldots,N-1\}$, we have  $\exp(\cut_G(v_j)) = sR^2$. 
Sketching with \cref{alg:contract_sketch}, tree embedding and tensor train embedding all would yield this optimal asymptotic cost.